\documentclass[intlim,righttag,10pt]{amsart}
\usepackage{amscd}
\usepackage{amssymb}
\usepackage[all]{xy}
\def\invlim{\mathop{\vtop{\ialign{##\crcr$\hfill{\lim}\hfil$\crcr
\noalign{\kern1pt\nointerlineskip}\leftarrowfill\crcr\noalign
{\kern -3pt}}}}\limits}
\def\dirlim{\mathop{\vtop{\ialign{##\crcr$\hfill{\lim}\hfil$\crcr
\noalign{\kern1pt\nointerlineskip}\rightarrowfill\crcr\noalign
{\kern -3pt}}}}\limits}
\def\lomapr#1{\smash{\mathop{\relbar\joinrel\longrightarrow}\limits^{#1}}}
 \def\verylomapr#1{\smash{\mathop{\relbar\joinrel\relbar\joinrel\relbar\joinrel\longrightarrow}\limits^{#1}}}

\def\phi{\varphi}
\def\epsilon{\varepsilon}
\def\what{\widehat}

\newtheorem{theorem}[equation]{Theorem}
 \newtheorem{lemma}[equation]{Lemma}   
 \newtheorem{proposition}[equation]{Proposition}
 \newtheorem{corollary}[equation]{Corollary}
\newtheorem{conjecture}[equation]{Conjecture}

\theoremstyle{definition}
\newtheorem{definition}[equation]{Definition}
\theoremstyle{remark}
\newtheorem{remark}[equation]{Remark}
\newtheorem{conventions}[equation]{Conventions}

\newtheorem{example}[equation]{Example}
\newtheorem*{acknowledgments}{Acknowledgments}

\usepackage{fge}
\renewcommand{\setminus}{\mathbin{\fgebackslash}}

\newcommand{\ovk}{\overline{K} }
 \newcommand{\ovv}{\so_{\ovk} }

 \newcommand{\skl}{\operatorname{sk} }

 \newcommand{\holim}{\operatorname{holim} }
\newcommand{\cofiber}{\operatorname{cofiber} }

 \newcommand{\Feet}{\operatorname{F\acute{e}t} }

 \newcommand{\eet}{\operatorname{\acute{e}t} }

 \newcommand{\Zar}{\operatorname{Zar} }

 \newcommand{\Spec}{\operatorname{Spec} }

 \newcommand{\Hom}{\operatorname{Hom} }
 
  \newcommand{\hocolim}{\operatorname{hocolim} }

 \newcommand{\Gal}{\operatorname{Gal} }
 \newcommand{\tr}{ \operatorname{tr} }
 \newcommand{\can}{ \operatorname{can} }
 
 \newcommand{\id}{ \operatorname{Id} }
\newcommand{\synt}{ \operatorname{syn} }
 
 \newcommand{\Cone}{\operatorname{Cone} }

\newcommand{\st}{\operatorname{st} }

  \newcommand{\hk}{\operatorname{HK} }
   \newcommand{\htt}{\operatorname{HT} }
   \newcommand{\dr}{\operatorname{dR} }
 \newcommand{\kker}{\operatorname{Ker} }
 \newcommand{\crr}{\operatorname{cr} }
 \newcommand{\gr}{\operatorname{gr} }
 \newcommand{\im}{\operatorname{Im} }

 \newcommand{\ve}{ \varepsilon  }
 \newcommand{\kr}{^{\scriptscriptstyle\bullet}}

 \newcommand{\sff}{{\mathcal{F}}}  
 \newcommand{\sii}{{\mathcal{I}}}
 
 \newcommand{\sh}{{\mathcal{H}}}

 \newcommand{\sk}{{\mathcal{K}}}
 \newcommand{\sll}{{\mathcal{L}}}
 
 \newcommand{\so}{{\mathcal O}}
 \newcommand{\sj}{{\mathcal J}}
 
 \newcommand{\sa}{{\mathcal{A}}}
 
 \newcommand{\sx}{{\mathcal{X}}}
 \newcommand{\sss}{{\mathcal{S}}}
\newcommand{\sd}{{\mathcal{D}}} 
 
\newcommand{\spp}{{\mathcal{P}}}

 \newcommand{\wt}{\widetilde}
 \newcommand{\wh}{\widehat}

 \newcommand{\Qp}{\mathbf{Q}_p}

\newcommand{\Q}{\mathbf{Q}}
\newcommand{\Z}{\mathbf{Z}}
 \newcommand{\B}{\mathbf{B}}
  \newcommand{\A}{\mathbf{A}}
  \newcommand{\R}{\mathrm{R}}
   \newcommand{\rg}{\R\Gamma}
\numberwithin{equation}{section}
 
\begin{document} 
 \title[On uniqueness of $p$-adic period morphisms, II]
 {On uniqueness of $p$-adic period morphisms, II}
 \author{Wies{\l}awa Nizio{\l}}
 \address{CNRS, UMPA, \'Ecole Normale Sup\'erieure de Lyon, 46 all\'ee d'Italie, 69007 Lyon, France}
\email{wieslawa.niziol@ens-lyon.fr} \date{\today}
\thanks{This research was supported in part by  the NSF grant DMS0703696 and the grant ANR-14-CE25.}
 \maketitle 
 \begin{abstract}
 
 We prove equality of the various rational $p$-adic period morphisms for smooth, not necessarily 
proper, schemes.  We start with showing that the $K$-theoretical uniqueness criterium we had found earlier for proper smooth schemes extends to  proper finite simplicial schemes in the good reduction case and to cohomology with compact support in the semistable reduction case. It yields the equality of the period morphisms  for cohomology with compact support 
 defined using the syntomic, almost \'etale, and  motivic constructions.  
 
 We continue with showing  that the $h$-cohomology period morphism agrees with the syntomic and almost \'etale period morphisms whenever the latter morphisms are defined (and up to a change of Hyodo-Kato cohomology). We do it by lifting the syntomic and almost \'etale period morphisms to the $h$-site of varieties over a field,
 where their equality with the $h$-cohomology period morphism can be checked directly using the Beilinson Poincar\'e Lemma and the case of dimension $0$.  This also shows that the syntomic and almost \'etale period morphisms have a natural extension to  the Voevodsky triangulated category of motives and enjoy many useful properties (since so does the $h$-cohomology period morphism).
 \end{abstract}
 \tableofcontents
 \section{Introduction}Recall that rational $p$-adic period morphisms\footnote{We also discuss in this paper integral $p$-adic period morphisms in the context of Fontaine-Lafaille theory and the motivic approach to comparison theorems; see Section \ref{only-integral}.}  make it possible to describe the $p$-adic \'etale cohomology of algebraic varieties over local fields of mixed characteristic 
in terms of differential forms. This is advantagous since the latter can often be computed. There are by now
four main different approaches to the construction of  these period morphisms:
\begin{itemize}
\item syntomic: Fontaine-Messing \cite{FM}, Hyodo-Kato \cite{HK}, Kato \cite{K1},  Tsuji \cite{Ts}, Yamashita \cite{Ya}, Colmez-Nizio{\l} \cite{CN}, 
\item  almost \'etale: Faltings \cite{Fi}, \cite{Fa}, Scholze \cite{Sch}, Li-Pan \cite{LP}, Diao-Lan-Liu-Zhu \cite{DLLZ}, Tan-Tong \cite{TT}, Bhatt-Morrow-Scholze \cite {BMS}, \cite{BMS2},  \v{C}esnavi\v{c}ius-Koshikawa \cite{KC},  
 \item motivic:  Nizio{\l} \cite{N4}, \cite{N2}, 
 \item $h$-cohomology: Beilinson \cite{BE1}, \cite{BE2}, Bhatt \cite{BH}.
 \end{itemize}
Each of these approaches has its advantages and it is important to be able to compare
the resulting period morphisms in the case one needs to  pass from one to another.
 Since all the above  period morphisms are normalized using Chern classes we expect  them to be equal.

  The two theorems below are examples of the results we obtain in  the paper. Let ${\so_K}$ be a complete discrete valuation ring with fraction field
$K$  of characteristic 0 and with perfect
residue field $k$ of positive characteristic $p$. Let $\pi$ be a uniformizer of $\so_K$. Let
$\so_F$ be the ring of Witt vectors of $k$ with 
 fraction field $F$.   Let $X$ be a proper 
scheme over $\so_K$ with semistable reduction and of 
 pure relative dimension $d$. Let 
$i:D\hookrightarrow X$ be the horizontal divisor and set $U=X\setminus D$. Equip $X$ with the log-structure induced by $D$ and the special fiber. Denote by  $\so_F^0$ the scheme $\Spec ({\so_F})$ with the log-structure given by 
 $({\mathbf N}\to {\so_K}, 1\mapsto 0)$. 
 
  The first theorem is a generalization of the $K$-theoretical uniqueness criterium for $p$-adic period isomorphisms  from \cite{N10} as well as its applications.
 \begin{theorem}
 \label{intro1}
\begin{enumerate}
\item 
There exists a unique natural $p$-adic period isomorphism
  $$
  \alpha_i:\quad H^i_{\eet,c}(U_{\overline{K}},{\mathbf Q}_p)\otimes \B_{\st}\stackrel{\sim}{\to} H^i_{\hk}(X)\otimes_{F} \B_{\st},\quad i\geq 0,
  $$
  where $H^i_{\hk}(X)= H^i_{\crr}(X_0/\so_F^0)_{\Q}$ is the Hyodo-Kato cohomology,
  such that
  \begin{enumerate}
  \item 
  $ \alpha_i$  is $\B_{\st}$-linear, Galois equivariant, and compatible with Frobenius;
  \item $ \alpha_i $, extended to $\B_{dR}$ via the Hyodo-Kato morphism $\rho_{\pi}:H^i_{\hk}(X)\to H^i_{\dr}(X_K)$ and the morphism $\iota_{\pi}:\B_{\st}\to\B_{\dr}$, induces a filtered isomorphism
 $$
      \alpha^{\dr}_i:\quad H^i_{\eet,c}(U_{\overline{K}},{\mathbf Q}_p)\otimes \B_{\dr}\stackrel{\sim}{\to} H^i_{\dr,c}(X_K)\otimes_{K} \B_{\dr};
$$
\item   $ \alpha_i$ 
  is compatible with the \'etale and syntomic higher Chern classes from $p$-adic $K$-theory.
  \end{enumerate}
  \item 
  The syntomic, almost \'etale, and motivic semistable period morphisms for cohomology with compact support are equal.
  \end{enumerate} 
 \end{theorem}

  The second theorem takes a different approach to comparing $p$-adic period morphisms. It uses $h$-topology, Beilinson (filtered) Poincar\'e Lemma  \cite{BE2}, and the computations from \cite{NN} to formulate a simple
  uniqueness criterium using the fundamental exact sequence of $p$-adic Hodge Theory hence, basically, the case of dimension $0$. 
 \begin{theorem}\label{intro2}The syntomic, Faltings almost \'etale, and $h$-cohomology period morphisms lift to the Voevodsky category of motives over $K$. They are  equal. In particular, they are compatible with (mixed) products. 
 \end{theorem}
\begin{remark}The above theorems do not cover the $p$-adic period morphisms of Bhatt-Morrow-Scholze \cite{BMS}, \cite{BMS2} and \v{C}esnavi\v{c}ius-Koshikawa \cite{KC} (which fall into the "almost \'etale" category) but these morphisms  are  already known (at least the ones from \cite{BMS}, \cite{KC}) to be the same as the syntomic period morphisms:
\begin{enumerate}
\item It is likely that one can use the $K$-theory criterium from Theorem \ref{intro1} to show this fact. Some compatibilities with Chern classes were already checked in \cite{CDN4}.  The $h$-topology method of comparing period morphisms from Theorem \ref{intro2}  can not be applied directly in this case because the period morphism of Bhatt-Morrow-Scholze, as of now, are not allowing horizontal divisors.
\item However, the compatibility of the period morphism from \cite{BMS}, \cite{KC}  with the other period morphisms has been already checked in the forthcoming thesis of Sally Gilles (at ENS-Lyon) by a more direct method. 
This involves the period morphism defined in  \cite{CN}: Gilles lifted the local definition of this morphism to the geometric setting, globalized it together with its comparison with the Fontaine-Messing period morphism, and then directly compared the resulting morphism with the period morphism from \cite{KC} (which is a reasonable approach since  both morphisms are defined using very similar complexes). 
\end{enumerate}
\end{remark}
\begin{remark}
Recently, there has been considerable interest in generalizing Faltings' original approach to $p$-adic comparison theorems. This started with the work of Scholze \cite{Sch} on the de Rham comparison theorem for proper smooth  rigid varieties and nontrivial coefficients  that extended Faltings' proof of the algebraic de Rham comparison theorem using Scholze's powerful almost purity theorem and his proof of the finiteness of $p$-adic \'etale cohomology. Recall that Faltings proof of the de Rham comparison theorem  used the Faltings site, Faltings Poincar\'e lemma, a basic comparison theorem and worked for all smooth algebraic varieties and trivial coefficients. This  was extended to nontrivial coefficients in the thesis of Tsuzuki, which was, unfortunately, never published.

 More work followed: Li-Pan \cite{LP} extended Scholze's de Rham comparison for trivial coefficients to the open case (with a nice compactification), Diao-Lan-Liu-Zhu \cite{DLLZ} added a treatment of nontrivial coefficients; from another angle: Tan-Tong \cite{TT} extended Scholze's proof to the case of good reduction (over an unramified base) proving the Crystalline conjecture in this setting. 

  When specialized to algebraic varieties all these constructions of $p$-adic period morphisms are   modifications of the original construction of Faltings (recall that Faltings' construction works for any smooth variety) the main one being a replacement of Faltings site with the pro-\'etale site (see the discussion in \cite[Sec. 3]{LP}). Their equality with Faltings period morphisms is conceptually clear but, with all the  modifications involved,  the detailed proof of this fact is best left for the time when it is  really needed (and then it can be checked in a direct, if tedious,  way, or, in some cases, using our $K$-theory approach). 
  \end{remark}
   \subsection{Proof of Theorem \ref{intro1}}
  To prove Theorem \ref{intro1} we start with showing that the $K$-theoretical uniqueness criterium we had found  for proper smooth schemes  in \cite{N10}  extends to  finite simplicial schemes in the good reduction case and to cohomology with compact support in the semistable reduction case. Using it we show  the equality of the period morphisms  for cohomology with compact support 
 defined by  the syntomic and  almost \'etale methods.
 Along the way we extend our definition of the motivic period morphisms from \cite{N4}, \cite{N2} to  the above mentioned setting. By construction this period morphism satisfies the $K$-theoretical uniqueness criterium hence it is  equal to the syntomic and almost \'etale period morphisms.

 To present  the proof of Theorem \ref{intro1} in more details, recall the definition of the motivic period  morphisms 
 in the simpler case of good reduction (see also  the survey \cite{NC}).  
   Let $X$ be a smooth proper scheme  over $\so_K$. 
 Using the Suslin comparison theorem between $p$-adic motivic cohomology and $p$-adic \'etale cohomology \cite{Su} we lift \'etale cohomology classes of $X_{\ovk}$  to $p$-adic motivic cohomology classes via the \'etale regulator (here we use $\lambda$-graded pieces of $p$-adic $K$-theory as a substitute for $p$-adic motivic cohomology), then we lift those to the integral model $X_{\so_{\ovk}}$, and, finally, we project them via the syntomic regulator to the syntomic cohomology of $X_{\so_{\ovk}}$ that maps canonically to the absolute crystalline cohomology of $X_{\so_{\ovk}}$.

This  extends rather easily to simplicial schemes: there is no problem in defining  the $p$-adic regulators and the fact that the \'etale regulator and the localization map from the integral model to the generic fiber are  isomorphisms can be reduced to the case of schemes using the filtration of simplicial schemes by skeletons. 
 
    We have shown in \cite{N10} that the construction of the motivic period morphisms for proper smooth schemes  implies a simple  $K$-theoretical uniqueness criterium for period morphisms. This can be extended now to proper smooth finite simplicial schemes:
   two period morphisms are equal if and only if the induced period morphisms from \'etale to syntomic cohomology are equal and this is true if and only if the latter agree on  the values of \'etale regulators from $p$-adic $K$-theory. This, in turn, would follow if the  period morphisms were  compatible with
     the \'etale and syntomic regulators from $p$-adic $K$-theory. For motivic period morphisms this compatibility follows from the definition; for the syntomic and almost \'etale period morphisms of Tsuji \cite{Ts} and Faltings \cite{Fa}, respectively,  this can be checked on the level of the universal Chern classes and this was done in \cite{N10}. 
     
     \subsection{Proof of Theorem \ref{intro2}}
 To prove Theorem \ref{intro2}  we take a different approach to comparing $p$-adic period morphisms: we compare them with the $h$-cohomology period morphism.
 First, we note that it is enough to compare  the induced morphisms, after a change of Hyodo-Kato cohomology,  from syntomic cohomology to \'etale cohomology (we  call them {\em syntomic period morphisms}). Then we take the syntomic period morphism (in the derived category) and sheafify it in the $h$-topology of $X_{\ovk}$. This is possible because Beilinson has shown \cite{BE1} that de Jong augmentations allow us to exhibit a basis of $h$-topology that consists of proper (strictly) semistable schemes over $\so_K$. We obtain a map between the $h$-sheafification of syntomic cohomology and the $h$-sheafification of \'etale cohomology.  Now, for $r\geq 0$,  the \'etale cohomology of the Tate twist  $\Z/p^n(r)^{\prime}:=(p^aa!)^{-1}{\mathbf Z}/p^n(r)$, for $r=(p-1)a+b,a,b\in{\Z}, 0\leq b < p-1$, $h$-sheafifies to the constant sheaf $\Z/p^n(r)^{\prime}$. Using Beilinson filtered Poincar\'e Lemma \cite{BE2} we see that the syntomic cohomology of the $r$'th twist sheafifies to the kernel of the surjective map of constant sheaves $F^r_p\A_{\crr}\lomapr{1-\phi_r} \A_{\crr}$, $\phi_r$ being the divided Frobenius $\phi/p^r$ and $F^r_p\A_{\crr}$ -- the Frobenius-divisible filtration. By the fundamental exact sequence  this is $\Z/p^n(r)^{\prime}$ and the syntomic period morphism, by functoriality, is the map that sends $t^{\{r\}}:=t^{b}(t^{p-1}/p)^{a}$ to $1$. But, as was shown in \cite{NN},  this is the same map as the one  induced  by the $h$-cohomology period morphism. The argument for the almost \'etale period morphism is analogous. 
 
 The last claim of the theorem was proved for the Beilinson period isomorphism in \cite{DN}; hence it is true for the other period maps as well.
 
  \begin{acknowledgments}
Parts of this paper were written during my visits to the Institut Henri Poincar\'e in Paris, the Institut de Math\'ematiques de Jussieu,
Columbia University, and the IAS at Princeton. 
I would like to thank these institutions  for their support, hospitality,
 and for the  wonderful working conditions they have supplied.
 
  I thank Frederic D\'eglise for helpful comments related to the content of this paper and the referee for a very careful reading of the manuscript and useful comments. 
 \end{acknowledgments}
 \begin{conventions}
 We assume all the schemes (outside of some obvious exceptions)  to be locally noetherian. We work in the category of fine log-schemes. For a scheme $X$ over $\Z_p$, we will denote by $X_n$ its reduction modulo $p^n$. 
 \end{conventions}
 \section{Preliminaries}\label{preliminaries}
 We collect in this section basic cohomological computations, the study of the localization map in $K$-theory, and the study of the \'etale cycle class map. All of this is done in the context of cohomology with compact support and generalizes the computations done for the usual cohomology in \cite{N4}, \cite{N2}.

 Let ${\so_K}$ be a complete discrete valuation ring with fraction field
$K$  of characteristic 0 and with perfect
residue field $k$ of characteristic $p$. Let
$W(k)=\so_F$ be the ring of Witt vectors of $k$ with 
 fraction field $F$.  Let $\ovk$ be an algebraic closure of $K$ and let $C$ be its $p$-adic compeltion. Set $G_K=\Gal(\overline {K}/K)$ and 
let $\sigma$ be the absolute
Frobenius on $W(\overline {k})$. 
For an ${\so_K}$-scheme $X$, let $X_0$ denote
the special fiber of $X$. We will denote by ${\so_K}$, 
$\so_K^{\times}$, and $\so_K^0$ the scheme $\Spec ({\so_K})$ with the trivial, canonical
(i.e., associated to the closed point), and $({\mathbf N}\to {\so_K}, 1\mapsto 0)$ 
log-structure respectively. We will freely use the notation from \cite{NR}.

\subsection{Cohomological identities}We briefly review here
certain facts involving syntomic and crystalline cohomologies that we will need.

 \subsubsection{Rings of periods}   We start with reviewing
 basic facts concerning the rings of periods. Consider the ring $R=\invlim \so_{\ovk}/p\so_{\ovk}$, where 
the maps in the projective system are the $p$-th power maps.
With addition and multiplication defined coordinatewise $R$ is a ring of 
characteristic $p$. 
Take its ring of Witt vectors $W(R)$. Then $\A_{\crr}$ is the $p$-adic
completion of the divided power envelope $D_\xi (W(R))$ of the ideal
$\xi W(R)$ in $W(R)$.
Here $\xi =[p^{\flat}]-p$  and, for $x \in R$,
$[x]=[x,0,0,...] \in W(R)$ is its Teichm\"{u}ller representative.

{\em {\rm (i)} The rings $\B_{\crr}$ and $\B_{\dr}$.} The ring $\A_{\crr}$ is a topological $W(k)$-module having the following
properties:
\begin{enumerate}
\item
$W(\overline {k})$ is embedded as a subring of $\A_{\crr}$ and $\sigma
$ extends naturally to a Frobenius $\phi $ on $\A_{\crr}$;
\item
$\A_{\crr}$ is equipped with a decreasing separated filtration
$F^{n}\A_{\crr}$ such that, for $n<p$,
 $\phi (F^{n}\A_{\crr})\subset 
p^{n}\A_{\crr}$
 (in fact, $F^{n}\A_{\crr}$ is the closure of 
the $n$-th divided power of  the PD ideal of $D_\xi (W(S))$);
\item
$G_K$ acts on $\A_{\crr}$; the action is $W(\overline
{k})$-semilinear, continuous, commutes with $\phi $ and preserves the
filtration;
\item
there exists an element $t\in F^1\A_{\crr}$ such that $\phi (t)=pt$ and
$G_K$ acts on $t$ via the cyclotomic character: if we fix
$\varepsilon \in R$ -- a sequence of nontrivial $p$-roots of unity,
then $t=\log ([\varepsilon ])$.
\end{enumerate}
$\B_{\crr}^+$ and $\B_{\crr}$ are defined as the rings $\A_{\crr}[p^{-1}]$ and $\A_{\crr}[p^{-1},t^{-1}]$, respectively,  
with the induced topology, filtration, Frobenius and the Galois action. 
For us, in this paper, it will be essential that the ring $\A_{\crr}$
can be thought of as a cohomology of an 'arithmetic point', namely that
$$
\A_{\crr,n}\simeq H_{\crr}^*(\Spec(\so_{\ovk,n})),
$$ 
where, for a scheme $Y$ over $W(k)$, we set $$
\R\Gamma_{\crr}(Y_n):=\R\Gamma_{\crr}(Y_n/W_n(k)),\quad 
H^*_{\crr}(Y):=H^*_{\crr}(Y/W(k)):=H^*\holim_n\R\Gamma_{\crr}(Y_n).$$
The canonical morphism $\A_{\crr,n}\to \so_{\ovk}/p^n$ is surjective.
 Let $J_{\crr,n}$ denote its kernel. Let
$$
\B^+_{\dr}=\invlim_r({\bold Q}\otimes \invlim_n \A_{\crr,n}/
J_{\crr,n}^{[r]}),\quad \B_{\dr}=\B^+_{\dr}[t^{-1}].
$$
The ring $\B_{\dr}^+$ has a discrete valuation given by powers of $t$. 
Its quotient
field is $\B_{\dr}$. We will denote by $F^n\B^+_{\dr}$ the filtration induced on 
$\B^+_{\dr}$ by powers of $t$.

{\em {\rm (ii)} The rings $\B_{\st}$, $\wh{\B}_{\st}$.}
 Let us now recall the definition of the ring $\B_{\st}$ 
\cite{F6}. Set $\B^+_{\st}:=\B^+_{\crr}[u]$, $\phi(u)=pu$, $Nu=-1$. 
  Let $\pi $ be a uniformizer of ${\so_K}$ (which we will fix in the rest of the paper). Let $\iota=\iota_{\pi}:\B^+_{\st}\hookrightarrow \B^+_{\dr}$ denote the  embedding
   $u\mapsto u_{\pi}=\log([\pi^{\flat}]/\pi)$. We   use it to induce the Galois action on $\B^+_{\st}$ from the one on $\B^+_{\dr}$. 
  Let $\B_{\st }=\B_{\crr}[u_{\pi }]$. 
  

  We will need the following crystalline interpretation of the ring 
$\B^+_{\st}$ (see \cite{K1}, \cite{Ts}).
Let $R_{\pi,n}$ denote the PD-envelope of the ring $W_n(k)[x]$ with respect 
to the closed
immersion $W_n(k)[x]\to {\so_{K,n}}$, $x\mapsto \pi $, equipped with the 
log-structure associated to ${\bold N}\to R_{\pi,n},$ $1\mapsto x$. Set $R_{\pi}:=\invlim_nR_{\pi,n}$. Let
$$
\wh{\A}^+_{\st}=\invlim_nH^0_{\crr}(\Spec(\so_{\ovk,n})/R_{\pi,n}),\quad \wh{\B}^+_{\st}:=\wh{\A}^+_{\st}
[1/p].
$$
The ring $\wh{\B}^+_{\st}$ has a natural action of $G_K$, Frobenius $\phi $, 
and a monodromy 
operator $N$. Kato \cite[3.7]{K1}  shows that the ring $\B^+_{\st}$ is canonically (and compatibly with all the structures) 
isomorphic to the subring of elements of $\wh{\B}^+_{\st}$ annihilated by a power of 
the monodromy 
operator $N$. The map $\iota: \B^+_{\st}\to\B^+_{\dr}$ extends naturally to a map $\iota: \wh{\B}^+_{\st}\to\B^+_{\dr}$.
\subsubsection{Syntomic cohomology} We will recall  briefly the definition of syntomic cohomology. 
 For a log-scheme $X$ we denote by $X_{\synt}$ the small log-syntomic site of $X$. 
For a log-scheme $X$ log-syntomic over $\Spec(W(k))$, define 
$$
\so^{\crr}_n(X) =H^0_{\crr}(X_n,\so_{X_n}),\qquad 
\sj_n^{[r]}(X) =H^0_{\crr}(X_n,\sj^{[r]}_{X_n}),
$$
where $\so_{X_n}$ is the structure sheaf of the absolute log-crystalline site (i.e., over $W_n(k)$), 
$\sj_{X_n}=\kker(\so_{X_n/W_n(k)}\to \so_{X_n})$, and $\sj^{[r]}_{X_n}$ is its $r$'th divided power of $\sj_{X_n}$.
Set $\sj^{[r]}_{X_n}=\so_{X_n}$ if $r\leq 0$.
 There is a  canonical, compatible with Frobenius,  and 
functorial isomorphism 
$$
H^*(X_{\synt},\sj_n^{[r]})\simeq H^*_{\crr}(X_n,\sj^{[r]}_{X_n}).
$$
It is easy to see that $\phi(\sj_n^{[r]} )\subset p^r\so^{\crr}_n$ for $0\leq r\leq p-1$. This fails in general and we modify
$\sj_n^{[r]}$:  $$\sj_n^{<r>}:= \{x\in \sj_{n+s}^{[r]}\mid \phi(x)\in p^r\so^{\crr}_{n+s}\}/p^n ,$$
for some $s\geq r$.
This definition is independent of $s$. We can define
the divided  Frobenius $\phi_r="\phi/p^r": \sj_n^{<r>} \to \so^{\crr}_n$.
Set $$\sss_n(r):=\Cone(\sj_n^{<r>} \stackrel{1-\phi_r}{\longrightarrow}\so^{\crr}_n)[-1].$$

     We will write $\sss_n(r)$ for the syntomic sheaves on $X_{m,\synt}$, $m\geq n$,  as well as on $X_{\synt}$. We will also need the "`undivided"' version of syntomic complexes of sheaves:
$$\sss'_n(r):=\Cone(\sj_n^{[r]} \stackrel{p^r-\phi}{\longrightarrow}\so^{\crr}_n)[-1].$$
The natural map $\sss^{\prime}_n(r)\to \sss_n(r)$ induced by the maps $p^{r}: \sj_n^{[r]}\to \sj_n^{<r>}$ and $\id : \so^{\crr}_n \to \so^{\crr}_n $ has kernel and cokernel  killed by $p^{r}$. 
  We will also  write $\sss_n(r)$, $\sss^{\prime}_n(r)$ for  $\R\ve_*\sss_n(r)$, $\R\ve_*\sss^{\prime}_n(r)$, respectively, 
where $\varepsilon: X_{n,\synt}\to X_{n,\eet}$ is the canonical projection to the \'etale site.

 The $p$-adic syntomic cohomology of $X$  is defined as 
 \begin{align*}
\R\Gamma_{\eet}(X,\sss(r)):=\holim_n\rg_{\eet}(X,\sss_n(r)),\quad \R\Gamma_{\eet}(X,\sss^{\prime}(r)):=\holim_n\rg_{\eet}(X,\sss^{\prime}_n(r)).
 \end{align*}
\subsubsection{Cohomology with compact support} \label{reduction-lyon}Let $X$ be a finite and saturated  log-smooth log-scheme  over $\so_K^{\times}$ (resp. over ${\so_K}$). Since $X$ is log-regular it is  normal and the maximal open subset $U=X_{\tr}\subset X$, where the log-structure $M_X$ is trivial is dense in $X$. We have $M_X=\so_X\cap j_*\so_U^*$, where $j:U\hookrightarrow X$ is the open immersion. 
  By \cite[Th. 5.10]{N7} there exists a log-blow-up of $X$ that has Zariski log-structure and is (classically) regular. 
  
  Assume that $X$ itself has these properties. Then $U$ is a complement of a divisor with simple normal crossings that is a union $D_0\cup D$  (resp.  $D$) of the reduced special fiber and the horizontal part $D$. 
  The scheme $X$ has {\em generalized semistable reduction}, i.e., 
Zariski locally on $X$, there exists an \'etale morphism over ${\so_K}$:
\begin{equation*}
X\to \Spec({\so_K}[T_1,\ldots,T_u]/(T_1^{n_1}\cdots T_u^{n_u}-\pi)[U_1,\ldots,U_m,V_{1},\ldots,V_t])
\end{equation*}
for some integers $u\geq 1$ (resp. $u=0$), $m,t\geq 0, n_i>0$. The divisor $D$ is the inverse image of $U_1\cdots U_m=0$. 
In particular, all the closed strata of $D$ are log-smooth over $\so_K^{\times}$ and regular (resp. smooth over ${\so_K}$).
If all $n_i=1$ we say that $X$ has {\em semistable reduction}.

 Take $X$ as above with semistable reduction. Recall the following definitions. 
The $p$-adic {\em \'etale cohomology} of $X_{\ovk}$ with compact support\footnote{If $X$ is proper this is, of course, isomorphic to $\R\Gamma_{\eet,c}(U_{\ovk},\Q_p)$.}.
$$
\R\Gamma_{\eet,c}(X_{\ovk},\Q_p)=\R\Gamma_{\eet}(X_{\ovk},\overline{j}_{K!}\Q_p).
$$
The {\em de Rham cohomology} of $X_K$ with compact support \cite[Def. 3.2]{Tsc} 
$$
\R\Gamma_{\dr,c}(X_K)=\R\Gamma(X_K,\sii_{D_K}\Omega\kr_{X_K}),
$$
where $\sii_{D_K}\subset j_{K*}\so^*_{U_K}\cap\so_{X_K}$ is the ideal of $\so_{X_K}$ corresponding to $D_K$. We filter it by
$F^r\R\Gamma_{\dr,c}(X_K)=\R\Gamma(X_K,\sii_{D_K}\Omega^{\geq r}_{X_K})$, $r\in\Z$. 
The {\em crystalline  cohomology}  of $X_0$ over $W(k)^0$ with compact support \cite[Def. 5.4]{Tsc} 
$$
\R\Gamma_{\crr,c}(X_0/W(k)^0)=\R\Gamma_{\crr}(X_0/W(k)^0,\sk_{D_0}),
$$
where $\sk_{D_0}$ is an ideal  sheaf induced by the sheaf $\sii_{D_0}$ \cite[Lemma 5.3]{Tsc}. The crystalline cohomology
 $\R\Gamma_{\crr,c}(X)$ is defined in a similar way. We filter it by setting $F^r\R\Gamma_{\crr,c}(X)=\R\Gamma_{\crr}(X,\sk_{D_0}\sj^{[r]}_X)$, $r\in\Z$. 
This allows us to define the {syntomic cohomology} with compact support $\R\Gamma_{\synt,c}(X,\sss_n(r))$ and $\R\Gamma_{\synt,c}(X,\sss^{\prime}_n(r))$. 

 The above cohomologies with compact support are special cases of cohomologies of finite simplicial schemes. Define
  $C(X,D):=\cofiber(\wt{D}_{{\scriptscriptstyle\bullet}}\stackrel{i_*}{\to} X),
 $ where $\wt{D}_{{\scriptscriptstyle\bullet}}$ is the \v{C}ech nerve of the map $\coprod_i D_i\to D$, $D_i$ being an irreducible component of $D$. The log-structure on the schemes in $C(X,D)$ is trivial if $X$ is over $\so_K$ and induced from the special fiber if $X$ is over $\so_K^{\times}$.
\begin{lemma}
\label{lyon1}
 Let $\R\Gamma(X)$ denote one of the cohomologies mentioned above. We have a natural (filtered) quasi-isomorphism
$$
\R\Gamma_c(X)\simeq \rg(C(X,D)).
$$
It is compatible with products\footnote{The product on the cohomology of a simplicial scheme is defined as the $\holim$-product induced by the cosimplicial degree-wise products.}.
\end{lemma}
\begin{proof}The \'etale and de Rham cases 
  follow immediately from the following exact sequences ($r\in\Z$)
\begin{align}
\label{expression}
& 0\to \overline{j}_{K!}\Q_p\to \Q_{p,X_{\ovk}} \to\overline{ i}_{1*}\Q_{p,D^1_{\ovk}}\to \overline{i}_{2*}\Q_{p,D^2_{\ovk}}\to \cdots\\
& 0\to \sii_{D_K}\Omega^{\geq r}_{X_K}\to \Omega^{\geq r}_{X_K}\to i_{1*}\Omega^{\geq r}_{D^{1}_K}\to i_{2*}\Omega^{\geq r}_{D^2_K}\to \cdots\notag
\end{align}
Here $D^m:=\wt{D}_m$ is the direct sum of  the intersections of $m$ irreducible components of $D$. We note that $C(X,D)_{\ovk}\simeq C(X_{\ovk},D_{\ovk})$ even if $(X,D)$ is not geometrically irreducible.

     The crystalline case over $W(k)^0$ follows from
 a mixed characteristic analog of the second sequence. And the case over $W(k)$ reduces to this sequence as well. Indeed, if $\so_K=W(k)$ this is clear. In general, locally, we have an embedding into such a situation. Because, by assumption, this embedding is regular, the above mentioned sequence remains exact after tensoring with
 the divided power envelope and computes cohomology with compact support.
 
  For the syntomic case, it suffices to check that the above crystalline quasi-isomorphism  preserves filtrations. But this follows easily from the fact that the associated grading of the filtration on the divided power envelope is free over $\so_X$.
  
  Concerning compatibility with products, the \'etale, de Rham, and the crystalline cases are  immediate from the expressions (\ref{expression}). In the syntomic case, compatibility  follows from the fact that syntomic cohomology is defined as a mapping fiber of (filtered) crystalline cohomology and the syntomic product is the mapping fiber product induced from the crystalline product.
 \end{proof}
\subsubsection{Fontaine-Lafaille theory}The main reference for this section is \cite{FL}.
Assume first  that $ {\so_K}=W(k) $. For the integral crystalline theory (Fontaine-Laffaille theory) 
we will need the following abelian categories:
\begin{enumerate}
\item $ \mathcal{M }\sff_{big}({\so_K})$ -- an object is given by a $p$-torsion
${\so_K}$-module $M$ and a family of $p$-torsion ${\so_K}$-modules $F^{i}M$ together
with
${\so_K}$-linear maps $ F^{i}M\rightarrow F^{i-1}M,$ $ F^{i}M\rightarrow M$ 
 and $\sigma $-semilinear maps $\phi_{i} :  F^{i}M\rightarrow M $ 
satisfying certain compatibility conditions;
\item  ${\mathcal M}\sff ({\so_K}) $ -- the full subcategory of ${\mathcal M}\sff 
_{big}({\so_K}) $ with objects --  finite ${\so_K}$-modules $M$ such that $ F^{i}M=0\text{ for } i\gg 0$, the
maps $F^{i}(M)\rightarrow M $ are injective and $\sum \text{Im}\, \phi_{i} = M$ ;
\item  $ {\mathcal M}\sff _{[a,b]}({\so_K})$ -- the full subcategory of
objects $M$ of ${\mathcal M}\sff ({\so_K}) $ such that $ F^{a}M=M $ and $F^{b+1}M=0. $ 
\end{enumerate}
Consider the category $ {\mathcal M}\sff _{[a,b]}({\so_K})\text{ with }
b-a\leq p-2. $ There exists an exact and fully faithful  functor 
$$
{\bold L}(M)=\ker(F^0(M\otimes \A_{\crr}\{-b\}(-b))\lomapr{1-\phi_0}  M\otimes 
\A_{\crr}(-b)),
$$
where $\{-b\}$, $(-b)$ are 
the ${\mathcal M}\sff $ and Tate twists\footnote{For $M\in {\mathcal M}\sff$, we set $F^jM\{i\}:=F^{j-i}M, \phi_{M\{i\},j}:=p^{i}\phi_{M,j-i}$.} respectively, from
${\mathcal M}\sff _{[a,b]}({\so_K}) $ to finite $ {\mathbf Z}_{p} $-Galois representations. Its
essential image is called the category of {\em crystalline representations 
of weight between } $a$ {\em and } $b$. This category is closed under taking tensor
 products and duals (assuming we stay in the admissible range of the filtration).

 The following proposition generalizes \cite[2.7]{FM}, Faltings \cite[4.1]{Fi}, and \cite[Lemma 2.3]{N10} from schemes to finite simplicial schemes.
\begin{proposition}
\label{reff1}
Let $X$ be a smooth and proper $m$-truncated simplicial scheme over ${\so_K}=W(k)$ whose   components have dimension smaller than $d$. Then, 
for $d\leq p-2$ or for $i\leq p-2$, the filtered Frobenius module 
$H^i_{\crr}(X_n)$ lies in 
${\mathcal M}\sff_{[0,d]} ({\so_K}) $ or ${\mathcal M}\sff_{[0,i]} ({\so_K})$, respectively. Moreover, then  the natural morphism
$$\psi_n:\quad H^i_{\eet}(X_{\ovv},\sss_n(r))\stackrel{\sim}{\to}
{\mathbf L}(H^i_{\crr}(X_n)\{-r\})
\simeq  F^rH^i_{\crr}(X_{{\ovv},n})^{\phi_r=1}
$$
is an isomorphism for $p-2\geq r\geq d$ or for $0\leq i\leq r\leq p-2$, respectively.
\end{proposition}
 Here,
$$
H^i_{\crr}(X_n)\simeq H^i_{\dr}(X_n/{\so_{K,n}}):=H^i(X_n,\Omega\kr_{X_n/{\so_{K,n}}})
$$
and  the maps
$$
\phi_k="\phi/p^k": F^kH^i_{\crr}(X_n)\to
H^i_{\crr}(X_n),
$$
where $\phi$ denotes the crystalline Frobenius.
   The Hodge filtration
$$
F^kH^i_{\crr}(X_n)\simeq \im(H^i(X_n,\Omega^{\geq k}_{X_n/{\so_{K,n}}})\to
H^i(X_n,\Omega\kr_{X_n/{\so_{K,n}}}))
$$
since the Hodge-de Rham spectral sequence of $X_n$ degenerates: by devissage, we can reduce to $n=1$ and then it follows from the results of Deligne-Illusie \cite[Cor. 3.7]{DI}.
\begin{proof} 
The proof of \cite[2.7]{FM} for schemes goes through for truncated simplicial schemes proving the first claim of the proposition. For the second claim,
we argue by induction on $m\geq 0$ such that $X\simeq \skl_mX$. The case of $m=0$ is treated in \cite[2.7]{FM}. Assume that our proposition is true for $m-1$. To show it for $m$ consider the homotopy cofiber sequence
$$
\skl_{m-1}X_{\so_{\ovk}}\to \skl_mX_{\so_{\ovk}}\to \skl_{m}X_{\so_{\ovk}}/\skl_{m-1}X_{\so_{\ovk}}
$$ 
and apply the maps $\psi_n$ to it. We get the   
 map of  sequences
$$
\xymatrix{
H^{i-1}_{\synt}(\skl_{m-1}X)\ar[r] \ar[d]^{\psi_n}_{\wr} & H^{i-1}_{\synt}(X^{\prime}_m) \ar[r]\ar[d]^{\psi_n}_{\wr} &
H^i_{\synt}(\skl_{m}X)\ar[r]\ar[d]^{\psi_n} & H^i_{\synt}(\skl_{m-1}X)\ar[r] \ar[d]^{\psi_n}_{\wr} &
H^i_{\synt}(X^{\prime}_m)\ar[d]^{\psi_n}_{\wr}\\
{\mathbf L}(H^{i-1}_{\crr}(\skl_{m-1}X))\ar[r]  & {\mathbf L}(H^{i-1}_{\crr}(X^{\prime}_{m}))\ar[r] & {\mathbf L}(H^i_{\crr}(\skl_{m}X))\ar[r] & {\mathbf L}(H^i_{\crr}(\skl_{m-1}X))\ar[r] & 
{\mathbf L}(H^{i}_{\crr}(X^{\prime}_{m}))
}
$$
Here we set $H^*_{\synt}(Y)=H^*_{\eet}(Y_{\so_{\ovk}}, \sss_n(r))  $, $ {\mathbf L}(H^*_{\crr}(Y))={\mathbf L}(H^*_{\crr}(Y_n)\{-r\})$.
We also put
$$
H^{*}_{\alpha}(X^{\prime}_{m},*) =H^{*}_{\alpha}(X_{m},*)\cap\ker s_0^*\cap\cdots\cap\ker s_{m-1}^*,\quad \alpha=\synt,\crr,
$$
where each $s_i: X_{m-1}\to X_m$ is a degeneracy map.
The top  sequence is exact. 
So is the bottom: it is clearly exact before applying ${\mathbf L}$ and it stays exact because the relevant categories ${\mathcal M}\sff$ are closed under taking subobjects and the functor ${\mathbf L}$ is exact.

By the inductive hypothesis  we have the isomorphisms shown. It follows that
the map  $$
\psi_n:\quad H^i_{\eet}(\skl_{m}X_{\so_{\ovk}},\sss_n(r)){\to}{\mathbf L}(H^*_{\crr}(\skl_mX_n)\{-r\})
$$ is an isomorphism as well. Since  $H^i_{\eet}(\skl_{m}X_{\so_{\ovk}},\sss_n(r))\stackrel{\sim}{\to}H^i_{\eet}(X_{\so_{\ovk}},\sss_n(r))$
and $H^*_{\crr}(\skl_mX_n)\stackrel{\sim}{\to}H^*_{\crr}(X_n)$, we are done. 
\end{proof}
The above proposition can be applied to  cohomology with compact support.
\begin{corollary}
\label{reff2}
Let $X$ be a smooth and proper scheme over ${\so_K}=W(k)$ with a divisor $D$ that has relative simple normal crossings and all the closed strata smooth over ${\so_K}$. Equip $X$ with the log-structure coming from $D$. Then, if the relative dimension $d$ of $X$ is  $\leq p-2$ or if $i\leq p-2$, the filtered Frobenius module 
$H^i_{\crr}(X_n)$ lies in 
${\mathcal M}\sff_{[0,d]} ({\so_K}) $ or ${\mathcal M}\sff_{[0,i]} ({\so_K})$, respectively. Moreover, then  the natural morphism
$$\psi_n:\quad H^i_{\eet}(X_{\ovv},\sss_n(r))\stackrel{\sim}{\to}
{\mathbf L}(H^i_{\crr,c}(X_n)\{-r\})
\simeq  F^rH^i_{\crr,c}(X_{n})^{\phi_r=1}
$$
is an isomorphism for $p-2\geq r\geq d$ or for $0\leq i\leq r\leq p-2$, respectively.
\end{corollary}
\begin{proof}
 By Lemma \ref{lyon1}, 
we have a canonical isomorphism $$H^i_{\crr,c}(X_n)\simeq H^i_{\crr}(C(X,D)_n). $$ 
  Our corollary follows now from Proposition \ref{reff1}.
\end{proof}

\subsubsection{More cohomological identities}\label{cohohoho}
Let  $\so_K$ be general and 
  let $X$ be an ${\so_K}$-scheme. 
Recall that, if  $X$ is smooth and proper, Kato and Messing \cite{KM}
have constructed the following isomorphisms
\begin{align*}
 h_{\crr}:\quad H ^i_{\crr} & (X_{0})_{\bold Q}
\otimes \B^+_{\crr}  \stackrel{\sim}{\to} 
H^i_{\crr}(X_{\so_{\ovk}})_{\Q}
\quad \text{
\cite[1.2]{KM}},\\
H^i_{\dr} & (X_K)
\otimes \B^+_{\dr}
  \simeq
\invlim_N(\invlim_n 
H^i_{\crr}(X_{\so_{\ovk},n},\so_{n}/J_n^{[N]}))_{\Q}
\quad \text{
\cite[1.4]{KM}},\\
h_{\dr}:\quad F^r & (H^i_{\dr}(X_K)
\otimes \B^+_{\dr})  \stackrel{\sim}{\to} 
\invlim_N(\invlim_n 
H^i_{\crr}(X_{\so_{\ovk,n}},J_n^{[r]}/J_n^{[N]}))_{\Q}.
\end{align*}
    We will need also  to know that \cite[Lemma 2.2]{N4}
\begin{lemma}
\label{KM}
The following two compositions of maps are equal
\begin{align*}
{\bold Q}  \otimes\invlim_n & 
H^i_{\eet}(X_{\ovv},\sss_n'(r))  \to 
\invlim_N( {\bold Q}\otimes\invlim_n 
H^i_{\crr}(X_{\ovv,n},J_n^{[r]}/J_n^{[N]}))
   \lomapr{ h^{-1}_{\dr}}  F^r(H^i_{\dr}(X_K) \otimes \B^+_{\dr})\\
 &  \to H^i_{\dr}(X_K) \otimes \B^+_{\dr};\\ 
{\bold Q}  \otimes\invlim_n & H^i_{\eet}(  X_{\ovv},\sss_n'(r))  \to
{\bold Q}\otimes\invlim_n 
H^i_{\crr}(X_{\so_{\ovk,n}})
    \lomapr{ h^{-1}_{\crr } }
 H ^i _{\crr}(X_{0})\otimes_{W(k)} 
 \B^+_{\crr} \lomapr{ \delta} H^i_{\dr}(X_K) \otimes \B^+_{\dr},
\end{align*}
where $\delta$ is induced by the Berthelot-Ogus isomorphism \cite[2.2]{BO}
$H ^i _{\crr}(X_{0})\otimes_{W(k)} K \simeq H^i_{\dr}(X_K)$. 
\end{lemma}

    Let $X$ be any fine  log-scheme, 
which is log-smooth and proper over
$\so_K^{\times}$ with saturated log-structure on the generic fiber.
We will need the  crystalline interpretation of $\B^+_{\dr}\otimes _{K} 
H^i_{\dr}(X_K)$ from 
\cite{K1} (see also \cite[4.7]{Ts}):
\begin{align}
\label{derham}
\B^+_{\dr}\otimes _{K} & 
H^i_{\dr}(X_K)\stackrel{\sim}{\to}
\invlim_s H^{i}_{\crr}(X_{\so_{\ovk}}/
\so_K^{\times },\so/J^{[s]})_{\Q}\quad \text{\cite[4.7.6]{Ts}},\\
F^r(\B^+_{\dr}\otimes _{K} & 
H^i_{\dr}(X_K))\stackrel{\sim}{\to} 
\invlim_{s\geq r}  H^{i}_{\crr}(X_{\so_{\ovk}}/
\so_K^{\times },J^{[r]}/J^{[s]})_{\Q}\quad \text{\cite[4.7.13]{Ts}}.\notag
\end{align}

  Finally, let us recall briefly the {\em Hyodo-Kato isomorphism}.
We define the Hyodo-Kato cohomology as
$$
H^i_{\hk}(X):= H^i_{\crr}(X_0/W(k)^0)_{\Q}.
$$
If the special fiber of $X$ is of Cartier type,  Kato  defines 
\cite[4.2,4.5]{K1} canonical morphisms (that however depend on the choice of $\pi$)
\begin{equation}
\label{cafe1}
  H^i_{\crr}(X_{\so_{\ovk}})_{\Q}\stackrel{h_{\pi}}{\rightarrow}
(\wh{\B}^+_{\st}\otimes_{F}H^i_{\hk}(X))^{N=0}
\stackrel{\sim}{\leftarrow} (\B^+_{\st} \otimes _{F} H^i_{\hk}(X))^{N=0}.
\end{equation}
It can be checked (see \cite[4.5.6-7]{Ts}) that
 these morphisms are compatible with Galois action and the Frobenius. 
Moreover, Hyodo and Kato \cite[5.1]{HK} have constructed a canonical 
$K$-isomorphism
\begin{equation}
\label{cafe2}
\rho_{\pi }: K\otimes_{F}H^i_{\hk}(X)
\stackrel{\sim}{\to}H^i_{\dr}(X_K).
\end{equation}
Hence the composition
$$
\rho_{\pi }h_{\pi }: H^i_{\crr}(X_{\so_{\ovk}})_{\Q}\to \B^+_{\st}\otimes _{F} 
H^i_{\dr}(X_K)
$$
is functorial in $X$ and compatible with 
Galois action. 

 It is easy to check that all the above extends to finite simplicial (log-)schemes $X$:
 \begin{enumerate}
 \item The map $h_{\crr}$ actually lifts in a functorial way to a statement in the $\infty$-derived category\footnote{A good source  of the quasi-isomorphisms of this type is \cite{BE2} as well as \cite{NN}.}, hence extends to simplicial schemes. Similarly, for the morphisms in (\ref{cafe1}).
 \item Similarly for the Hyodo-Kato isomorphism (\ref{cafe2}) though here finiteness of the simplicial scheme is an important assumption one needs to make to control the denominators (for details see \cite[6.3]{Tc} and \cite[2.8]{Ki}).
 \item Similarly for the map $h_{\dr}$, the maps in Lemma \ref{KM}, and the maps in (\ref{derham}), where in addition one needs to use that the Hodge-de Rham  spectral sequence for $X$ degenerates (which, by passing to the complex numbers, follows from the classical Hodge Theory, see \cite[7.2.8]{D2}).
 \end{enumerate}
 \subsubsection{A key isomorphism}
 Let $X$ be a proper semistable scheme over $\so_K$. The following lemma will be crucial in the comparison of period morphisms.
 \begin{lemma}
 \label{lyon-added}
 Let $r\geq i$.  There exists a natural isomorphism 
 $$
 H^i_{\eet}(X_{\so_{\overline{K}}},\sss'(r))_{\Q}\stackrel{\sim}{\to} (H^i_{\hk}(X)\otimes_{F}\B_{\st})^{N=0,\phi=p^r}\cap
F^r(H^i_{\dr}(X_K)\otimes_K\B_{\dr}).
 $$
\end{lemma}
 \begin{proof}This is well-known; see \cite[Cor. 3.23]{NN}, \cite[Prop. 5.22]{CN}. We will sketch here the construction of the map  for future reference; see \cite[Cor. 3.23]{NN} for details. 
 Consider the  following sequence of maps of homotopy limits; they are all quasi-isomorphisms. Homotopy limits are taken in the $\infty$-derived category.  
\begin{align}
\label{lyon-diag}
\R\Gamma_{\eet}(X_{\so_{\overline{K}}},\sss'(r))_{\mathbf Q}
 & \stackrel{\sim}{\to}
   \xymatrix{[\R\Gamma_{\crr}(X_{\so_{\overline{K}}})^{\phi=p^r}_{\mathbf Q}\ar[r]^{\can}&   \R\Gamma_{\crr}(X_{\so_{\overline{K}}})_{\Q}/F^r ]}\\\notag
      &  \stackrel{\sim}{\to}
 \xymatrix{[\R\Gamma_{\crr}(X_{\so_{\overline{K}}}/R_{\pi})^{N=0,\phi=p^r}_{\mathbf Q}\ar[r]^-{p_{\pi}} & \R\Gamma_{\crr}(X_{\so_{\overline{K}}}/\so_K^{\times})_{\Q}/F^r]}\\\notag
   &  \stackrel{\sim}{\leftarrow} 
    \xymatrix{[(\R\Gamma_{\crr}(X/R_{\pi})\otimes^L_{R_{\pi}}\wh{\B}^+_{\st})^{N=0,\phi=p^r}\ar[r]^-{p_{\pi}\otimes\iota}&(\R\Gamma_{\dr}(X_K)\otimes^L_K\B^+_{\dr})/F^r ]}\\\notag
&  \stackrel{\iota_{\pi}}{\leftarrow}
  \xymatrix@C=40pt{[(\R\Gamma_{\hk}(X)\otimes^L_F\wh{\B}^+_{\st})^{N=0,\phi=p^r}\ar[r]^-{\rho_{\pi}\otimes\iota}  & 
    (\R\Gamma_{\dr}(X_K) \otimes^L_K\B^+_{\dr})/F^r]}\\\notag
    &  \stackrel{\sim}{\leftarrow}
  \xymatrix{[(\R\Gamma_{\hk}(X)\otimes^L_F{\B}^+_{\st})^{N=0,\phi=p^r}\ar[r]^-{\rho_{\pi}\otimes\iota}  & 
    (\R\Gamma_{\dr}(X_K) \otimes^L_K\B^+_{\dr})/F^r]}\notag
\end{align}
Here the eigenspaces are taken in the derived sense and we used the brackets $[-]$ to denote a mapping fiber. The first two maps and the last map are the canonical maps.  We wrote $p_{\pi}$ for the projection $x\mapsto \pi$. The second map   is induced by the distinguished triangle  
$$\R\Gamma_{\crr}(X_{\so_{\overline{K}}})\to \R\Gamma_{\crr}(X_{\so_{\overline{K}}}/R_{\pi})\stackrel{N}{\to}\R\Gamma_{\crr}(X_{\so_{\overline{K}}}/R_{\pi}).$$ 
The third map is induced by the K\"unneth map; we also used here the quasi-isomorphism (\ref{derham}). 
The fourth map is induced 
by the section
$\iota_{\pi}: \R\Gamma_{\hk}(X)\to \R\Gamma_{\crr}(X/R_{\pi})_{\mathbf Q}$ of the projection $x\mapsto 0$ (recall that $\rho_{\pi}=p_{\pi}\iota_{\pi}$).

\end{proof}

\subsection{Localization map}
For (finite simplicial) schemes $X$ over ${\so_K}$ that are smooth or log-smooth and regular the localization map 
$$j^*:K_i(X_{\so_{\ovk}},{\bold Z}/n)\to
K_i(X_{\overline{K}},{\bold Z}/n),\quad i\geq 0,$$
where $j:X_{\ovk}\hookrightarrow X_{\so_{\ovk}}$ is the natural open immersion, is easy to understand as the two following lemmas show. Here $K_i(-,{\bold Z}/n)$ is the $K$-theory with coefficients $\Z/n$ (see \cite[Sec. 4.1.1]{NR}).
\begin{lemma}
\label{kt}
Let $X$ be a finite  smooth simplicial ${\so_K}$-scheme. For any  integer $n$,   
the localization  morphism $$j^*:K_i(X_{\so_{\ovk}},{\bold Z}/n)\to
K_i(X_{\overline{K}},{\bold Z}/n),\quad i\geq 0,$$
is an isomorphism.
\end{lemma}
\begin{proof}
Recall that we have proved in \cite[Lemma 3.1]{N4} that this lemma is true  if    $X$ is a single smooth scheme over ${\so_K}$.  
By the same method we get the other hypercohomology spectral sequences, namely, the weight spectral sequence \cite[5.13, 5.48]{T}.
\begin{align}
E^{st}_2  =H^s(m\mapsto \pi_t(K(X_m),{\mathbf Z}/n)) \Rightarrow
H^{s-t}(X,K;{\mathbf Z}/n),  \quad t-s\geq 3.\notag
\end{align}
Here $K$ is  the presheaf ${\mathbf Z}\times {\mathbf Z}_{\infty}BGL$, where $BGL(U)=\injlim_nBGL_n(U)$.
Since the natural inclusion
$j:X_{\ovk}\hookrightarrow X_{\so_{\ovk}}$ induces a localization map on the corresponding spectral sequences compatible with 
the localization maps on  individual schemes we get isomorphisms on the terms of the spectral sequences that induce an isomorphism on the abutments, as wanted. 
\end{proof} 

  Let $X$ be a finite and saturated  Zariski log-smooth log-scheme  over $\so_K^{\times}$ (resp. over ${\so_K}$) that is classically regular. The maximal open subset $U=X_{\tr}\subset X$ where the log-structure $M_X$ is trivial is dense in $X$ and we have $M_X=\so_X\cap l_*\so_U^*$, where $l:U\hookrightarrow X$ is the open immersion. 
 $U$ is a complement of a divisor with simple normal crossings that is a union $D_0\cup D$  (resp. $D$) of the reduced special fiber and the horizontal part $D$. 
  
 Let $K_1$ be a finite extension of $K$ and let ${\so_{K_1}}$ be its ring of integers. 
The log-scheme $X_{\so_{K_1}}$ is in general singular but  it can be desingularized  by a log-blow-up, i.e., there exists a log-blow-up $f: Y\to X_{\so_{K_1}}$ that does not modify the regular locus and such that $Y$ is a (classically) regular Zariski log-scheme. Below we will only consider log-blow-ups of $X_{\so_{K_1}}$ that are {\em vertical}, i.e., we blow-up only closed strata involving the vertical divisor $D_{0,\so_{K_1}}$. More precisely, let $F(X)$ be the fan 
of $X$ \cite[10]{K2} (recall that $X$ is assumed to be Zariski and regular). It is a fan over the fan
 $F(\so_K^{\times})=\Spec({\mathbf N})$,
$\pi: F(X)\to \Spec({\mathbf N})$. Let $F_0(X)$ be the vertical fan of $F(X)$, i.e., the maximal open subfan of $F(X)$ containing the closed fiber $\pi^{-1}(s)$, where $s=\{n\geq 1|n\in {\mathbf N}\}$ is the closed point of $\Spec({\mathbf N})$ \cite[proof of Lemma 2.5]{TS}. We have a natural map $F(X)\to F_0(X)$. 

  The log-scheme $X_{\so_{K_1}}$ has the fan $F(X_{\so_{K_1}})=F_e(X)=F(X)\times_{\Spec({\mathbf N})}\Spec({\mathbf N}_e)$, where $e$ denotes the ramification index of $\so_{K_1}/{\so_K}$. We have the natural map 
$F(X_{\so_{K_1}})\to F_{0,e}(X)$. From now on we consider only log-blow-ups $Y\to X_{\so_{K_1}}$ induced from regular subdivisions of the vertical fan $F_{0,e}(X)$.  In the local picture above, we consider only log-blow-ups of $X_{\so_{K_1}}$ induced from log-blow-ups of the vertical part $X_{\so_{K_1}}^v$.
Notice that the scheme $Y$ has generalized semistable reduction as well and the horizontal divisor
 $D_Y$ is the preimage  of $D_{\so_{K_1}}$.

 Let $\sx_{\ovv}$ denote the projective system of such pairs $(f:Y\to Y_{\so_{K_1}}, \so_{K_1})$ (that we will sometimes just call $Y$) and 
 $\sd_{\ovv}$ denote  the induced projective system $(D_Y\subset Y,f, \so_{K_1})$, for $(f:Y\to X_{\so_{K_1}},\so_{K_1})\in \sx_{\ovv}$. 
 We will show 
that  
we can pass from the $K$-theory with compact support of the generic fiber $X_{\ovk}$  to the K-theory 
with compact support of the regular model $\sx_{\ovv}$ that we define as 
$$K^c_j(\sx_{\ovv},\sd_{\ovv},{\mathbf  Z}/p^n):=\dirlim_{Y\in \sx_{\ovv}}
K_j(C(Y,D_Y),{\mathbf Z}/p^n).
$$
\begin{lemma}
\label{model}
Let $j: X_{\ovk}\hookrightarrow \sx_{\ovv}$ be the natural open immersion.  Then
the restriction $$j^*: K^c_j(\sx_{\ovv},\sd_{\ovv},{\mathbf  Z}/p^n)
 \stackrel{\sim}{\to}  K^c_j(X_{\ovk},D_{\ovk},{\mathbf  Z}/p^n), \quad j>d+1,$$ is an isomorphism
and the induced map on the $\gamma$-graded pieces
$$
j^*: F^i_{\gamma}/F^{i+1}_{\gamma}K_j^c(\sx_{\ovv},\sd_{\ovv},{\mathbf  Z}/p^n) \to
 F^i_{\gamma}/F^{i+1}_{\gamma}K^c_j(X_{\ovk},D_{\ovk},{\mathbf  Z}/p^n),\qquad j> d+1, 
$$
has kernel and cokernel annihilated by $M(2d,i+1,2j)$ and 
$M(2d,i,2j),$ respectively. Here  $F^i_{\gamma}K_j(-,{\mathbf  Z}/p^n)$ is a $\gamma$-filtration (see \cite[Sec. 4.1.4]{NR}). 
\end{lemma}
\begin{remark}
 The integers
$ M(k,m,n)$ are defined by the following procedure \cite[3.4]{So}.
Let $l$ be a positive integer, and let $w_l$ be the greatest common divisor of
 the set of integers $k^{N}(k^l-1)$, as $k$ runs over the positive integers
 and $N$ is large enough with respect to $l$. Let
$M(k)$ be the product of the $w_l$'s for $2l<k$.
 Set $M(k,m,n)=\prod_{2m\leq 2l\leq n+2k+1}M(2l)$. An odd prime $p$ divides $M(d,i,j)$ if and only if 
$p<(j+2d+3)/2$, and divides $M(l)$ if and only if 
$p<(l/2)+1$.

\end{remark}
\begin{proof}It suffices to argue on finite levels. So we may simply assume that we have a regular scheme $X$ over ${\so_K}$ with a  divisor $D$ that has relative simple normal crossings and whose irreducible components  are all regular. We need to show the above lemma just for the pair $(X,D)$.

  For the first statement of the lemma consider the following commutative diagram with the horizontal sequences exact.
$$
\xymatrix{\ar[r] & K_{j+1}(\wt{D}_{{\scriptscriptstyle\bullet}},{\mathbf  Z}/p^n)\ar[d]^{j^*}\ar[r] & K^c_j(X,D,{\mathbf  Z}/p^n)\ar[d]^{j^*}\ar[r] &
K_j(X,{\mathbf  Z}/p^n)\ar[d]^{j^*}_{\wr}\ar[r]^{i^*} & 
K_j(\wt{D}_{{\scriptscriptstyle\bullet}},{\mathbf  Z}/p^n)\ar[d]^{j^*}\ar[r] & 
\\
\ar[r] & K_{j+1}(\wt{D}_{K{{\scriptscriptstyle\bullet}}},{\mathbf  Z}/p^n)\ar[r] & K^c_j(X_K,D_K,{\mathbf  Z}/p^n)\ar[r] & 
K_j(X_K,{\mathbf  Z}/p^n)\ar[r]^{i^*} & K_j(\wt{D}_{K{{\scriptscriptstyle\bullet}}},{\mathbf  Z}/p^n)\ar[r] & 
}
$$
It shows that it suffices to prove that the restriction  map 
$$j^*:\quad K_j(\wt{D}_{{\scriptscriptstyle\bullet}},{\mathbf  Z}/p^n){\to}K_j(\wt{D}_{K{{\scriptscriptstyle\bullet}}},{\mathbf  Z}/p^n),\quad j>d+1,
$$ is an isomorphism. To see that  write $D=\cup_{i=1}^{i=m}D_i$ as a union of irreducible components $D_i$
and argue by induction on  $m$. 
Recall that we have proved in \cite[Lemma 3.5]{N2} that the above lemma  is true  if   $m=1$. Assume now that the above isomorphism holds  for $m-1$. To prove it for $m$ consider  the restriction map of 
the following long exact sequences.
$$
\xymatrix{\ar[r] & K_{j+1}(\wt{D}_{Y{{\scriptscriptstyle\bullet}}},{\mathbf  Z}/p^n)\ar[d]^{j^*}_{\wr}\ar[r] &
K_{j}(\wt{D}_{{\scriptscriptstyle\bullet}},{\mathbf  Z}/p^n)\ar[d]^{j^*}\ar[r] & 
K_{j}(Y,{\mathbf  Z}/p^n)\oplus K_{j}(\wt{D}^{\prime}_{{\scriptscriptstyle\bullet}},{\mathbf  Z}/p^n)\ar[d]^{j^*}_{\wr}\ar[r] & K_{j}(\wt{D}_{Y{{\scriptscriptstyle\bullet}}},{\mathbf  Z}/p^n)\ar[d]^{j^*}_{\wr}\ar[r] &\\
\ar[r] & K_{j+1}(\wt{D}_{Y,K{{\scriptscriptstyle\bullet}}},{\mathbf  Z}/p^n)\ar[r] &
K_{j}(\wt{D}_{{K}{{\scriptscriptstyle\bullet}}},{\mathbf  Z}/p^n)\ar[r] & 
K_{j}(Y_{{K}},{\mathbf  Z}/p^n)\oplus K_{j}(\wt{D}^{\prime}_{{K}{{\scriptscriptstyle\bullet}}},{\mathbf  Z}/p^n)\ar[r] & K_{j}(\wt{D}_{Y,{K}{{\scriptscriptstyle\bullet}}},{\mathbf  Z}/p^n)\ar[r] &
}
$$
Here we wrote $Y=D_1$, $D^{\prime}=\cup_{i=2}^{i=m}D_i$, and $D_Y=D^{\prime}\cap Y$. By the inductive hypothesis we have the isomorphisms shown. 
It follows that we have the isomorphism $$j^*: \quad K_{j}(\wt{D}_{{\scriptscriptstyle\bullet}},{\mathbf  Z}/p^n)\stackrel{\sim}{\to} K_{j}(\wt{D}_{K{{\scriptscriptstyle\bullet}}},{\mathbf  Z}/p^n), \quad j>d+1, $$ as wanted.

 Hence the first statement of the lemma is true. It implies  that, for $j>d+1$, 
 the top map in the following commutative diagram is an isomorphism
$$
\begin{CD}
\wt{F}^i_{\gamma}/\wt{F}^{i+1}_{\gamma}K_j^c(X,D,{\mathbf  Z}/p^n) @>j^*>\sim >
 \wt{F}^i_{\gamma}/\wt{F}^{i+1}_{\gamma}K^c_j(X_{K},D_{K},{\mathbf  Z}/p^n)\\
@VVV @VVV\\
F^i_{\gamma}/F^{i+1}_{\gamma}K_j^c(X,D,{\mathbf  Z}/p^n) @>j^*>>
 F^i_{\gamma}/F^{i+1}_{\gamma}K^c_j(X_{K},D_{K},{\mathbf  Z}/p^n).
\end{CD}
$$
Here $\wt{F}^i_{\gamma}$ refers to a modified $\gamma$-filtration (see \cite[Sec. 4.1.4]{NR} for details). 
Since, by \cite[Lemma 4.4]{NR}, 
$ M(2d,i,2j)F^i_{\gamma}K^c_j(X_{K},D_{K},{\mathbf  Z}/p^n)\subset  \wt{F}^i_{\gamma}K^c_j(X_{K},D_{K},{\mathbf  Z}/p^n)$, we 
 get the second statement of our lemma.
\end{proof}
 \subsection{\'Etale Chern classes}
  The following proposition 
shows that we can invert  \'etale    Chern classes 
 modulo some constants. 
\begin{proposition}
\label{diagr}Let $Y$ be a smooth  finite simplicial scheme  over 
$\overline{K}$ such that $Y\simeq \skl_mY$. Set $d=\max_{s\leq m}\dim Y_s$. Let $p^n\geq 5$, $j\geq \max\{2d,2\}$,  $j\geq 3$ for $d=0$ and $p=2$, and $2i-j\geq 0$.
There exists an integer $D(d,m,i,j)$ 
depending only on $d$, $m$, $i$, and $j$ such that,
 the kernel and  cokernel of the Chern classe map
$$\overline{c}_{ij}^{\eet}:\gr^i_{\gamma}K_j(Y,{\bold Z}/p^n)
\rightarrow H^{2i-j}_{\eet}(Y,{\bold Z}/p^n(i))$$ are annihilated by $D(d,m,i,j)$. Any prime  $p> d+m+j+1$ 
does not divide $D(d,m,i,j)$.
\end{proposition}
\begin{remark}This proposition is a  $K$-theory version of the following theorem of Suslin \cite{Su}, \cite{Ge}.
\begin{theorem}(Suslin) For $Y$ a smooth scheme of dimension $d$ over $\ovk$, the change of topology map 
$$H^j_{\Zar}(Y,{\bold Z}/p^n(i)_M) \to
H^j_{\eet}(Y,{\bold Z}/p^n(i)_M)
$$
is an isomorphism for $i\geq d$. Here $\Z/p^n(i)_M$ is the complex of motivic sheaves (Bloch higher Chow complex).
\end{theorem}
\end{remark}
\begin{proof}To prove the proposition we are going to argue   by induction on $m$. The case of $m=0$ was  treated in \cite[Prop. 3.2]{N2}. We computed there that 
$$D(d,0,i,j)=(i-1)!M(d,i,j)M(d,i+1,j)M(d,i+1,2j)M(d,i,2j)M(2d)^{2d}.$$
Assume that $m\geq 1$. For the inductive step we need to filter $Y$ by its skeletons. 
We work on the site of schemes smooth over $\ovk$ equipped with the Zariski topology. Take a fibrant replacement $K\to K^f$.  The pointed simplicial sets $\Hom(\skl_tY,K^f)$ form a tower of fibrations converging to 
$\Hom(Y,K^f)$ \cite[X.3.2]{BKa}.  Let $F_t$ be the fiber over $*$ of $\Hom(\skl_tY,K^f)\to \Hom(\skl_{t-1}Y,K^f)$. Then, by Bousfield-Kan
\cite[Prop. X.6.3]{BKa}, 
$$F_t\simeq \Hom(\skl_{t}Y/\skl_{t-1}Y,K^f)\simeq \Omega^tN^tK^f(Y_t),
$$
where $$N^tK^f(Y_t)=K^f(Y_t)\cap \ker s_0^*\cap \ldots\cap \ker s_{t-1}^*$$
and $s_i:Y_{t-1}\to Y_{t}$ is a codegeneracy. In particular, the natural map
$$\Hom(Y,K^f)\stackrel{\sim}{\to} \Hom(\skl_mY,K^f)$$
is a weak-equivalence.

  For $j\geq 2$ and  $j+t\geq 3$, using again \cite[Prop. X.6.3]{BKa}, we get the long exact sequence 
\begin{equation}
\label{longexact1}
\to K_{j+t}(Y^{\prime}_t,\mathbf{Z}/p^n)\to K_j(\skl_{t}Y,\mathbf{Z}/p^n)\to K_j(\skl_{t-1}Y,\mathbf{Z}/p^n)\to K_{j+t-1}(Y^{\prime}_t,\mathbf{Z}/p^n)\to 
\end{equation}
Here we set $$K_{j+t}(Y^{\prime}_t,\mathbf{Z}/p^n)=K_j(\skl_tY/\skl_{t-1}Y,{\mathbf Z}/p^n)=K_{j+t}(Y_t,\mathbf{Z}/p^n)\cap \ker s_0^*\cap\ldots\cap \ker s_{t-1}^*.$$
By functoriality, $\lambda$-operations act on this exact sequence and this yields a sequence of $\gamma$-gradings
$$\to \gr^i_{\gamma}K_{j+t}(Y^{\prime}_t,\mathbf{Z}/p^n)\stackrel{d}{\to} \gr^i_{\gamma}K_j(\skl_{t}Y,\mathbf{Z}/p^n)\stackrel{d_1}{\to }\gr^i_{\gamma}K_j(\skl_{t-1}Y,\mathbf{Z}/p^n)\stackrel{d_2}{\to} \gr^i_{\gamma}K_{j+t-1}(Y^{\prime}_t,\mathbf{Z}/p^n)\stackrel{d_3}{\to} 
$$
that is exact only up to certain universal constants. More precisely, we have the following lemma.
\begin{lemma}
\label{exactseq}
If the element $[x]$ at any level of the above long sequence is a cocycle then  $C[x]$ is a coboundary for the following  constant $C$ 
\begin{enumerate}
\item $d_1([x])=0$ then $C=M(2i)M(2(j+t+d-i))$;
\item $d_2([x])=0$ then $C=M(2i)M(2(j+t+d-i))$;
\item $d([x])=0$ then $C=M(2i)M(2(j+t+d+1-i))$.
\end{enumerate}
\end{lemma} 
\begin{proof}
Before we proceed, note that $F^{j+t+1}_{\gamma}K_{j+t}(Y^{\prime}_t,\mathbf{Z}/p^n)=0$ since 
$K_{j+t}(Y^{\prime}_t,\mathbf{Z}/p^n)\subset K_{j+t}(Y_t,\mathbf{Z}/p^n)$ and we have
 \cite[Lemma 4.3]{NR}.
We will prove (1). The other cases can be proved in a similar way.  Assume that 
$[x]\in \gr^i_{\gamma}K_j(\skl_{t}Y,\mathbf{Z}/p^n)$ and look at the sequence 
$$\gr^i_{\gamma}K_{j+t}(Y^{\prime}_t,\mathbf{Z}/p^n)\stackrel{d}{\to} \gr^i_{\gamma}K_j(\skl_{t}Y,\mathbf{Z}/p^n)\stackrel{d_1}{\to}  \gr^i_{\gamma}K_j(\skl_{t-1}Y,\mathbf{Z}/p^n)
$$
Assume that 
$x\in F_{\gamma}^iK_j(\skl_{t}Y,\mathbf{Z}/p^n)$ is such that $d_1([x])=0$. That means that on the level of the long exact sequence
(\ref{longexact1}) $d_1(x)\in F_{\gamma}^{i+1}K_j(\skl_{t-1}Y,\mathbf{Z}/p^n)$.  
We will need certain projectors \cite[2.8]{So}.
For two natural numbers $a\neq b$, denote by $A_{abk}$, $k\geq 2$,  
a family of integers such that  $w_{|b-a|}=\sum_{k\geq 2}A_{abk}(k^a-k^b)$. Let
$$
\phi_{a,b}=\sum_{k\geq 2} A_{abk}(\psi_k-k^{b}),\quad  \phi_a=\prod_{2\leq b\leq a-1}\phi_{a,b}, \quad \phi^a_m=\prod_{a+1\leq b\leq j+m+d}\phi_{a,b}, \quad a\geq 2.
$$
Note that, for any $x\in K_j(-,{\mathbf Z}/p^n)$, we have $
\phi_a(x)\in F^a_{\gamma}K_j(-, {\mathbf Z}/p^n)$. 
Since the $k$'th Adams operation $\psi_k$ acts on  $gr^c_{\gamma}K_j(\skl_{t}Y,\mathbf{Z}/p^n)$ as $k^c$ we have $$M(2(j+t+d-i))x-\phi^{i}_{t-1}(x)\in F^{i+1}_{\gamma}K_j(\skl_{t}Y,\mathbf{Z}/p^n)$$ so 
$M(2(j+t+d-i))[x]=[\phi^i_{t-1}(x)]$. 
Since $d_1x\in F_{\gamma}^{i+1}K_j(\skl_{t-1}Y,\mathbf{Z}/p^n)$ and  by \cite[Lemma 4.3]{NR} the length of the $\gamma$-filtration is $j+t-1+d$ we compute that $d_1(\phi^i_{t-1}(x))=\phi^i_{t-1}(d_1x)=0$. Hence $M(2(j+t+d-i))[x]=[y]$ such that $d_1(y)=0$ and $y\in F_{\gamma}^iK_j(\skl_{t}Y,\mathbf{Z}/p^n)$. 

 From the  long exact  sequence (\ref{longexact1})
 we then get $w\in K_{j+t}(Y^{\prime}_t,\mathbf{Z}/p^n)$ such that $dw=y$. Consider $w_1=\phi_i(w)\in F_{\gamma}^{i}K_{j+t}(Y^{\prime}_t,\mathbf{Z}/p^n)$. We have
 $$
 [d(w_1)]=[\phi_i(dw)]=\prod_{2\leq b\leq i-1}(\sum_{k\geq 2} A_{ibk}([\psi_k(dw)]-k^{b}[dw]))=\prod_{2\leq b\leq i-1}(\sum_{k\geq 2} A_{ibk}(k^{i}-k^{b}))[dw]=M(2i)[dw].
$$
Hence $M(2i)M(2(j+t+d-i))[x]$ is a cobundary, as wanted.
\end{proof}

  To proceed, we will need the following two lemmas.
\begin{lemma}
\label{BSC}
For a $d$-dimensional scheme $Y$ smooth over $\ovk$ we have 
$$M(d,i+1,2j)M(d,i,2j)\gr^i_{\gamma}K_j(Y,{\mathbf Z}/p^n)=0,\quad 2i-j<0.$$ 
\end{lemma}
\begin{proof}
This is the $K$-theory version of the mod-$p^n$ Beilinson-Soul\'e Conjecture. 
   Recall that we know its motivic version to be true. That
   is $H^{2i-j}_{\Zar}(Y,{\mathbf Z}/p^n)=0$ for $2i-j<0$ \cite{A}. So we just need to translate this statement into $K$-theory. Recall that Levine \cite{Ls} has constructed Zariski  Atiyah-Hirzebruch spectral
 sequence from
motivic cohomology to $K$-theory:
\begin{align*}
E^{s,q}_2 & =H^s_{\Zar}(Y,{\bold Z}/p^n(q/2)_M)\Rightarrow K_{s-q}(Y,{\mathbf Z}/p^n)
\end{align*}
Here the differential $d_r: E^{s,q}_r\to E^{s+r,q+r-1}_r$. Denote by $F^i_{AH}$ the filtration on $K$-theory groups defined by this spectral sequence. 
Levine shows \cite[13.11]{Ls}   that
$$
M(d,i,2j) F^i_{AH}K_{j}(Y,{\mathbf Z}/p^n)\subset  \wt{F}^i_{\gamma}K_{j}(Y,{\mathbf Z}/p^n)
\subset F^i_{AH}K_{j}(Y,{\mathbf Z}/p^n).
$$
By the above, the kernel of the map
$$\wt{F}^i_{\gamma}/\wt{F}^{i+1}_{\gamma}K_j(Y,{\mathbf Z}/p^n) \to F^i_{AH}/F^{i+1}_{AH}K_j(Y,{\mathbf Z}/p^n) 
$$
is annihilated by $M(d,i+1,2j)$ and the cokernel by $M(d,i,2j)$. By \cite[(4.4)]{NR}, same holds for the map
$$\wt{F}^i_{\gamma}/\wt{F}^{i+1}_{\gamma}K_j(Y,{\mathbf Z}/p^n) \to F^i_{\gamma}/F^{i+1}_{\gamma}K_n(Y,{\mathbf Z}/p^n).
$$
Since $F^i_{AH}/F^{i+1}_{AH}K_j(Y,{\mathbf Z}/p^n)$ is a subquotient of 
$E^{2i-j,2i}_2=H^{2i-j}_{\Zar}(Y,{\mathbf Z}/p^n(i)_M)$, we are done.
\end{proof}
\begin{lemma}
\label{special}
\begin{enumerate}
 \item For $i,j$ as in Proposition \ref{diagr},
the kernel and cokernel of the Chern class map
$$
 \overline{c}_{i,j}^{\eet}:\quad \gr^i_{\gamma}K_{j}(Y^{\prime}_m,\mathbf{Z}/p^n)\to H^{2i-j}_{\eet}(Y^{\prime}_m,{\bold Z}/p^n(i)),
$$
where $
H^{2i-j}_{\eet}(Y^{\prime}_m,{\bold Z}/p^n(i))
=H^{2i-j}_{\eet}(Y_m,{\bold Z}/p^n(i))\cap\ker s_0^*\cap\ldots\cap\ker s_{m-1}^*,
$
is annihilated by a constant $T(d,m,i,j)$. Any prime $p>d+j+1$ does not divide $T(d,m,i,j)$.
\item For $2i-j<0$, we have
$$i!(i+1)!\cdots (j+d)!M(d,i+1,2j)M(d,i,2j)\gr^i_{\gamma}K_j(Y^{\prime}_m,{\mathbf Z}/p^n)=0.$$ 
\end{enumerate}
\end{lemma}
\begin{proof}
Let us start with the first statement. For the kernel, take $x\in F^i_{\gamma}K_{j}(Y^{\prime}_m,\mathbf{Z}/p^n)$ such that 
$\overline{c}_{i,j}^{\eet}(x)=0$. Then $
D(d,0,i,j)x\in F^{i+1}_{\gamma}K_{j}(Y_m,\mathbf{Z}/p^n)\cap  K_{j}(Y^{\prime}_m,\mathbf{Z}/p^n).
$
Set $y=D(d,0,i,j)x$.  We have $$
\gamma^{i+1}(y)=(-1)^{i}i!y \mod F^{i+2}_{\gamma}K_{j}(Y_m,\mathbf{Z}/p^n)\cap  K_{j}(Y^{\prime}_m,\mathbf{Z}/p^n).
$$
 Since, by \cite[Lemma 4.3]{NR}, $F^{j+d+1}_{\gamma}K_{j}(Y_m,\mathbf{Z}/p^n)=0$, by the inductive argument we get
$i!(i+1)!\cdots (j+d)!D(d,0,i,j)x\in F^{i+1}_{\gamma}K_{j}(Y^{\prime}_m,\mathbf{Z}/p^n)$. 
So the kernel  is annihilated by $i!(i+1)!\cdots (j+d)!D(d,0,i,j)$.

  For the  cokernel, take $x\in H^{2i-j}_{\eet}(Y^{\prime}_m,{\bold Z}/p^n(i))$. Then 
$D(d,0,i,j)x=\overline{c}_{i,j}^{\eet}(y)$ for $y\in F^i_{\gamma}K_{j}(Y_m,\mathbf{Z}/p^n)$.
We need to show that some multiple of $y$ lies in $F^i_{\gamma}K_{j}(Y^{\prime}_m,\mathbf{Z}/p^n)$.
For each $l$, $0\leq l\leq m-1$, consider the following commutative diagram
$$
\xymatrix{
\gr^i_{\gamma}K_j(Y_m,\mathbf{Z}/p^n)\ar[r]^{s_l^*}\ar[d]^{\overline{c}_{i,j}^{\eet}} & \gr^i_{\gamma}K_j(Y_{m-1},\mathbf{Z}/p^n)\ar[d]^{\overline{c}_{i,j}^{\eet}}\\
H^{2i-j}_{\eet}(Y_m,{\bold Z}/p^n(i))\ar[r]^{s_l^*} & H^{2i-j}_{\eet}(Y_{m-1},{\bold Z}/p^n(i))
}
$$
Since $s_l^*(x)=0$ we have $D(d,0,i,j)s_l^*(y)\in F^{i+1}_{\gamma}K_{j}(Y_{m-1},\mathbf{Z}/p^n)$. Arguing just 
like in the proof of Lemma \ref{exactseq}, we find that 
$$
M(2(j+d-i))D(d,0,i,j)[y]=[y^{\prime}],\quad y^{\prime}\in F^i_{\gamma}K_{j}(Y_{m},\mathbf{Z}/p^n), s_l^*(y^{\prime})=0.
$$
Hence, repeating this argument for all $l$, we get 
$$
D(d,0,i,j)^mM(2(j+d-i))^m[y]=[y^{\prime}],\quad y^{\prime}\in F^i_{\gamma}K_{j}(Y_{m}, \mathbf{Z}/p^n)\cap K_{j}(Y^{\prime}_{m},\mathbf{Z}/p^n) 
$$
As above, $i!(i+1)!\cdots (j+d)!D(d,0,i,j)^mM(2(j+d-i))^m[y]=[y^{\prime}]$, $y^{\prime}\in F^i_{\gamma}K_{j}(Y^{\prime}_{m},\mathbf{Z}/p^n)$. Hence the cokernel is annihilated by
$i!(i+1)!\cdots (j+d)!D(d,0,i,j)^{m+1}M(2(j+d-i))^m$.
Set $$T(d,m,i,j)=i!(i+1)!\cdots (j+d)!D(d,0,i,j)^{m+1}M(2(j+d-i))^m,\quad 2i\geq j+m.$$

  For the second statement, assume that $2i-j<0$ and take 
$x\in F^i_{\gamma}K_j(Y^{\prime}_m,{\mathbf Z}/p^n)$. By Lemma \ref{BSC} $M(d,i+1,2j)M(d,i,2j)x\in
F^{i+1}_{\gamma}K_j(Y_m,{\mathbf Z}/p^n)\cap K_j(Y^{\prime}_m,{\mathbf Z}/p^n).$ Arguing as above $i!(i+1)!\cdots (j+d)!M(d,i+1,2j)M(d,i,2j)x\in F^{i+1}_{\gamma}K_j(Y^{\prime}_m,{\mathbf Z}/p^n).$
\end{proof} 

  Consider now the homotopy cofiber sequence
$$
  \skl_{m-1}Y\to \skl_m Y\to \skl_{m}Y/\skl_{m-1}Y
$$
By \cite[Remark 5.4]{NR}, the \'etale Chern class maps are compatible with it and we get the following commutative diagram 
(where we skipped the coefficients $\mathbf{Z}/p^n$ and ${\bold Z}/p^n(i)$, respectively).
$$
\begin{CD}
\gr^i_{\gamma}K_{j+1}(\skl_{m-1}Y)@>d_2>>\gr^i_{\gamma}K_{j+m}(Y^{\prime}_m)@>d>>  \gr^i_{\gamma}K_j(\skl_{m}Y)@>d_1>> \gr^i_{\gamma}K_j(\skl_{m-1}Y)@>d_2>> \gr^i_{\gamma}K_{j+m-1}(Y^{\prime}_m)\\
@VV\overline{c}_{i,j+1}^{\eet} V @VV\overline{c}_{i,j+m}^{\eet} V @VV\overline{c}_{ij}^{\eet} V @VV\overline{c}_{ij}^{\eet} V @VV\overline{c}_{i,j+m-1}^{\eet} V\\
H^{2i-j-1}_{\eet}(\skl_{m-1}Y) @>>> H^{2i-j-m}_{\eet}(Y^{\prime}_m)@>>> H^{2i-j}_{\eet}(\skl_mY) @>>> H^{2i-j}_{\eet}(\skl_{m-1}Y) @>>> H^{2i-j-m+1}_{\eet}(Y^{\prime}_m)
\end{CD}
$$
Here we put $H^{*}_{\eet}(Y^{\prime}_m)=H^{*}_{\eet}(Y_m)\cap \ker s_0^*\cap\ldots\cap\ker s_{m-1}^*$. 

Let's first look at the kernel of the map $\overline{c}_{ij}^{\eet}:\gr^i_{\gamma}K_j(\skl_{m}Y,\mathbf{Z}/p^n)\to H^{2i-j}_{\eet}(\skl_mY,{\bold Z}/p^n(i))$. 
Diagram chasing and the inductive hypothesis together with Lemma \ref{exactseq} and Lemma \ref{special}
imply easily that this kernel is annihilated by
$$T(d,m,i,j+m)D(d,m-1,i,j+1)D(d,m-1,i,j)M(2i)M(2(j+m+d-i))i!(i+1)!\cdots (j+m+d)!,
$$
if $2i\geq j+m$; if $2i< j+m$ we can drop the first term. 
Here we used the fact that the numbers $M(d,i+1,2j)$ and $M(d,i,2j)$ that appear in Lemma \ref{BSC}
divide $D(d,0,i,j)$.

  By a very similar argument, we get that the  cokernel of the map  $\overline{c}_{ij}^{\eet}:\gr^i_{\gamma}K_j(\skl_{m}Y,\mathbf{Z}/p^n)\to 
H^{2i-j}_{\eet}(\skl_mY,{\bold Z}/p^n(i))$ 
  is annihilated by $$T(d,m,i,j+m)T(d,m,i,j+m-1)D(d,m-1,i,j)M(2i)M(2(j+m+d-i))i!(i+1)!\cdots (j+m-1+d)!,
  $$
  if $2i\geq j+m$; if  $2i=j+m-1$ we can drop the first term;  if $2i < j+m-1$ we can drop the first two terms. 

 Set 
\begin{align*}
D(d,m,i,j) & =T(d,m,i,j+m)T(d,m-1,i,j+m-1)\\
 & D(d,m-1,i,j+1)D(d,m-1,i,j)M(2i)M(2(j+m+d-i))i!(i+1)!\cdots (j+m+d)!
\end{align*}
for $2i\geq j+m$; if $2i=j+m-1$ we drop the first term; if $2i < j+m-1$ we drop the first two terms.
Since an odd prime $p$ divides $M(l)$ if and only if 
$p<(l/2)+1$ and $H^t_{\eet}(\skl_mY)=0$ for $t>2d+m+1$, we get the last statement of the proposition.
 \end{proof}
\section{Comparison theorems for finite simplicial schemes via $K$-theory}\label{periodsniziol}
We are now ready to prove comparison theorems for finite simplicial schemes using $K$-theory. 
\subsection{Crystalline conjecture for finite simplicial schemes}
We start with the  Crystalline conjecture.
\subsubsection{Integral Crystalline conjecture}\label{only-integral} We treat first  its integral version.
 Let $X$ be a smooth proper finite simplicial scheme over ${\so_K}$, ${\so_K}=W(k)$. Assume that $X\simeq \skl_mX$ and that the dimension $d\leq p-2$, $d=\max_{s\leq m}\dim X_s$.
We would like  to construct 
functorial Galois equivariant morphisms
$$
\alpha_{ab}:\quad H^a_{\eet}(X_{\ovk},{\bold Z}/p^n(b)) \to  {\bold L}(
H^a_{\crr}(X_n)\{-b\}).
$$
We will be able to do it   under certain additional 
restrictions
on the integers $a,b$ and $d$. Our construction is  based on the following diagram
\begin{equation}
\label{basic}
\begin{CD}
F^b_{\gamma}/F^{b+1}_{\gamma}K_{2b-a}(X_{\ovv},{\bold Z}/p^n)
@>\sim > j^* > F^b_{\gamma}/F^{b+1}_{\gamma}
K_{2b-a}(X_{\overline{K}},{\bold Z}/p^n)\\
@VV\overline{c}_{b,2b-a}^{\synt} V @V\wr V \overline{c}_{b,2b-a}^{\eet} V\\
H^{a}_{\eet}(X_{\so_{\ovk}},\sss_n(b)) @. H^{a}_{\eet}(X_{\ovk},{\bold Z}/p^n(b)).
\end{CD}
\end{equation}
Here $1\leq  b<p-1$, $2b-a\geq 3$, $p^n\geq 5$, $p\neq 2$. The Chern class map
$$\overline{c}_{b,2b-a}^{\synt}:
F^b_{\gamma}K_j(X_{\so_{\ovk}},{\bold Z}/p^n) \to
H^{a}_{\eet}(X_{\so_{\ovk}},\sss_n(b))$$ is defined as the limit over finite 
extensions $\so_K'/{\so_K}$ of the syntomic 
Chern class maps $F^b_{\gamma}K_{2b-a}(X_{\so_K'},{\bold Z}/p^n) \to
H^{a}_{\eet}(X_{\so_K'},\sss_n(b))$.
 Due to \cite[Lemma 5.3]{NR}, the Chern class maps
$\overline{c}_{b,2b-a}^{\eet}$ and $\overline{c}_{b,2b-a}^{\synt}$  
factor through $F^{b+1}_{\gamma}$ 
yielding the maps in the above diagram. The restriction map 
$$j^*:\quad F^b_{\gamma}/F^{b+1}_{\gamma}K_{2b-a}(X_{\so_{\ovk}},{\bold Z}/p^n)
{\to} F^b_{\gamma}/F^{b+1}_{\gamma}
K_{2b-a}(X_{\overline{K}},{\bold Z}/p^n)$$
is an isomorphism by Lemma \ref{kt}. By Proposition \ref{diagr} the \'etale Chern class map 
$$\overline{c}_{b,2b-a}^{\eet}:\quad F^b_{\gamma}/F^{b+1}_{\gamma}
K_{2b-a}(X_{\overline{K}},{\bold Z}/p^n)\to H^{a}_{\eet}(X_{\ovk},{\bold Z}/p^n(b))
$$ is an isomorphism if $p>d+m+2b-a+1$.

 Assume  now that 
$
b\geq d$, $2b-a\geq 3$, and $ p-2\geq d+m+2b-a$. 
Define the morphisms
$$
\alpha_{ab}:\quad H^a_{\eet}(X_{\ovk},{\bold Z}/p^n(b))\to 
{\bold L}(H^a_{\crr}(X_n)\{-b\})
$$
as the composition $\alpha_{ab}:=
\psi_n 
\overline{c}^{\synt}_{b,2b-a}(j^*)^{-1}(\overline{c}^{\eet}_{b,2b-a})^{-1}$,
where $\psi_n $ is the natural map $H^a_{\eet}(X_{\so_{\ovk}},\sss_n(b))\to 
{\bold L}(H^a_{\crr}(X_n)\{-b\})$. Note that, by Proposition \ref{reff1}, this map is an isomorphism.

 The following  theorem  generalizes our \cite[Th. 4.1]{N4} from schemes  to finite simplicial schemes.
\begin{theorem}
\label{finalFL}
For any proper, smooth finite  simplicial scheme $X$ over ${\so_K}=W(k)$, $X\simeq \skl_mX$,
 the functorial Galois equivariant 
morphism
$$
\alpha_{ab}:H^a_{\eet}(X_{\ovk},{\bold Z}/p^n(b)) \stackrel{\sim}{\to}
{\bold L}(H^a_{\crr}(X_n)\{-b\})
$$
is an isomorphism, if the numbers $p,b,d$ are such that 
$b\geq 2d+3$, 
$p-2\geq 2b+d+m$, for $d=\max_{s\leq m}\dim X_s.$ 
\end{theorem}
\begin{remark}
The original constants that appear in  \cite{N4} are different (worse) than the ones we have quoted here. Also there we have assumed that the scheme $X$ was projective over ${\so_K}$. 
However one can easily modify the proof of Theorem 4.1 from \cite{N4} by replacing  the weak Proposition 4.1 used in \cite{N4} with its improved version (Prop. 3.2) from \cite{N2} to get the above theorem for schemes.
\end{remark}
\begin{proof} By Lemma \ref{kt}, Proposition \ref{diagr}, and Proposition \ref{reff1}, it suffices to show that the syntomic Chern class map 
$$\overline{c}^{\synt}_{b,2b-a}:\quad \gr^b_{\gamma}K_{2b-a}(X_{\so_{\ovk}},\mathbf{Z}/p^n)
\to H^a_{\eet}(X_{\so_{\ovk}},\sss_n(b))
$$ is an isomorphism. Note that for $a<0$ this is an isomorphism by Lemma \ref{BSC}. 

   We argue by induction on $m\geq 0$ such that $X\simeq \skl_mX$. The case of $m=0$ is treated by \cite[Th. 4.1]{N4}. Assume that our theorem is true for $m-1$. To show it for $m$ consider the homotopy cofiber sequence
$$
\skl_{m-1}X_{\so_{\ovk}}\to \skl_mX_{\so_{\ovk}}\to \skl_{m}X_{\so_{\ovk}}/\skl_{m-1}X_{\so_{\ovk}}
$$ 
and apply the syntomic Chern class maps  to it. We get the   
 map of  sequences
$$
\xymatrix{
K^b_{2b-a+1}(\skl_{m-1}X)\ar[r] \ar[d]^{\overline{c}^{\synt}_{b,2b-a+1}}_{\wr} &K^b_{2b-a+m}(X^{\prime}_m) \ar[r]\ar[d]^{\overline{c}^{\synt}_{b,2b-a+m}}_{\wr} &
K^b_{2b-a}(\skl_{m}X)\ar[r]\ar[d]^{\overline{c}^{\synt}_{b,2b-a}} & K^b_{2b-a}(\skl_{m-1}X)\ar[r] \ar[d]^{\overline{c}^{\synt}_{b,2b-a}}_{\wr} &
 K^b_{2b-a+m-1}(X^{\prime}_m)\ar[d]^{\overline{c}^{\synt}_{b,2b-a+m-1}}_{\wr}\\
H^{a-1}(\skl_{m-1}X,b)\ar[r]  & H^{a-m}(X^{\prime}_{m},b)\ar[r] & H^a(\skl_{m}X,b)\ar[r] & H^a(\skl_{m-1}X,b)\ar[r] & H^{a-m+1}(X^{\prime}_{m},b)
}
$$
Here we set $K^*_{*}(Y)=\gr^*_{\gamma}K_{*}(Y_{\so_{\ovk}})  $, $ H^*(Y, *)=H^*_{\eet}(Y_{\so_{\ovk}},\sss_n(*))$, and skipped the coefficients $\mathbf{Z}/p^n$ in $K$-theory. 
We also put
$$
K^*_{*}(X^{\prime}_m)  =K^*_{*}(X_m)\cap\ker s_0^*\cap\cdots\cap\ker s_{m-1}^*,\quad 
H^{*}(X^{\prime}_{m},*) =H^{*}(X_{m},*)\cap\ker s_0^*\cap\cdots\cap\ker s_{m-1}^*,
$$
where each $s_i: X_{m-1}\to X_m$ is a degeneracy map.
The bottom  sequence is exact. By Lemma \ref{exactseq} so is the top.
By the inductive hypothesis and by the case $m=0$ of this theorem plus Lemma \ref{special} and Lemma \ref{kt}, we have the isomorphisms shown. It follows that
the syntomic Chern class map $$
\overline{c}^{\synt}_{b,2b-a}:\quad \gr^b_{\gamma}K_{2b-a}(\skl_{m}X_{\so_{\ovk}},\mathbf{Z}/p^n){\to} H^a_{\eet}(\skl_{m}X_{\so_{\ovk}},\sss_n(b))
$$ is an isomorphism as well. Since $K_{2b-a}(\skl_{m}X_{\so_{\ovk}},\mathbf{Z}/p^n)\stackrel{\sim}{\to}K_{2b-a}(X_{\so_{\ovk}},\mathbf{Z}/p^n)$ and $H^a_{\eet}(\skl_{m}X_{\so_{\ovk}},\sss_n(b))\stackrel{\sim}{\to}H^a_{\eet}(X_{\so_{\ovk}},\sss_n(b))$, we are done. 
\end{proof}
\begin{example}({\em Integral Crystalline conjecture for cohomology with compact support}.)
    As a corollary of the above comparison theorem we obtain a comparison theorem for  cohomology with compact support.
Consider  a proper smooth  scheme $X$ over ${\so_K}=W(k)$. Let  $i:D\hookrightarrow X$, built from $m$ irreducible components that are smooth over ${\so_K}$, be the divisor at infinity of $X$. Let $U=X\setminus D$. 
 Consider the simplicial scheme $
 C(X,D):=\cofiber(\wt{D}_{{\scriptscriptstyle\bullet}}\stackrel{i_*}{\to} X).
 $ We have 
$C(X,D)\simeq \skl_mC(X,D)$. Equip $X$ with the log-structure associated to $D$. Applying the above constructions  to $C(X,D)$ we obtain the basic diagram
$$
\begin{CD}
F^b_{\gamma}/F^{b+1}_{\gamma}K^c_{2b-a}(X_{\so_{\ovk}},D_{\ovv},{\bold Z}/p^n)
@>\sim > j^* > F^b_{\gamma}/F^{b+1}_{\gamma}
K^c_{2b-a}(X_{\overline{K}},D_{\ovk},{\bold Z}/p^n)\\
@VV\overline{c}_{b,2b-a}^{\synt} V @V\wr V \overline{c}_{b,2b-a}^{\eet} V\\
 H^{a}_{\eet}(C(X_{\so_{\ovk}},D_{\ovv}),\sss_n(b))@. H^{a}_{\eet}(C(X_{\ovk},D_{\ovk}),{\bold Z}/p^n(b))
\end{CD}
$$
and the induced period morphism
$$\alpha^{\prime}_{ab}:\quad H^{a}_{\eet}(C(X_{\ovk},D_{\ovk}),{\bold Z}/p^n(b))\to H^{a}_{\eet}(C(X_{\ovv},D_{\ovv}),\sss_n(b)).
$$
But, by Lemma \ref{lyon1},
\begin{align*}
H^{a}_{\eet}(C(X_{\ovk},D_{\ovk}),{\bold Z}/p^n(b))\simeq H^{a}_{\eet,c}(U_{\ovk},{\bold Z}/p^n(b)), \quad
H^{a}_{\eet}(C(X_{\ovv},D_{\ovv}),\sss_n(b))\simeq H^{a}_{\eet,c}(X_{\ovv},\sss_n(b)).
\end{align*}
Hence we obtained a period morphism
$$\alpha^{\prime}_{ab}:\quad H^{a}_{\eet,c}(U_{\ovk},{\bold Z}/p^n(b))\to H^{a}_{\eet,c}(X_{\ovv},\sss_n(b))
$$
that composed with the map 
$ H^{a}_{\eet,c}(X_{\ovv},\sss_n(b))
 \to {\bold L}(H^a_{\crr,c}(X_n)\{-b\})$
yields  a Galois-equivariant map
$$\alpha_{ab}:\quad H^a_{\eet,c}(U_{\ovk},{\bold Z}/p^n(b)) \to
{\bold L}(H^a_{\crr,c}(X_n)\{-b\}).
$$

  We  get the following corollary of Theorem \ref{finalFL}.
\begin{corollary}
The Galois equivariant 
morphism 
$$
\alpha_{ab}:\quad H^a_{\eet,c}(U_{\ovk},{\bold Z}/p^n(b)) \to
{\bold L}(H^a_{\crr,c}(X_n)\{-b\})
$$
is an isomorphism, if the numbers $p,b,d$ are such that $b\geq 2d+3$, 
$p-2\geq 2b+d+m$.
\end{corollary}
\end{example}
\subsubsection{Rational Crystalline conjecture}
We will treat now the rational Crystalline conjecture. Let $X$ be a smooth proper finite simplicial scheme over ${\so_K}$, where  the ring 
${\so_K}$ is possibly ramified over $W(k)$. Assume that $X\simeq \skl_mX$ and set $d=\max_{s\leq m}\dim X_s$. For large  $b$, 
we will construct 
Galois equivariant functorial period morphisms
$$
\alpha_{ab}:H^a_{\eet}(X_{\ovk},{\bold Q}_p(b))\to
H^a_{\crr}(X_{0})\otimes  \B^+_{\crr}.
$$
 Assume that $p^n\geq 5$, $2b-a\geq \max\{2d,2\}$,  $2b-a\geq 3$ for $d=0$ and $p=2$, and $a\geq 0$.
  \cite[Lemma 5.3]{NR} and Lemma \ref{kt}
give us the following diagram
$$
\begin{CD}
F^b_{\gamma}/F^{b+1}_{\gamma}K_{2b-a}(X_{\ovv},{\bold Z}/p^n) @>\sim >j^* > 
F^b_{\gamma}/F^{b+1}_{\gamma}K_{2b-a}(X_{\overline{K}},{\bold Z}/p^n)\\
@VV\overline{c}_{b,2b-a}^{\synt}V @VV\overline{c}_{ij}^{\eet}V \\
H^{a}_{\eet}(X_{\ovv},\sss^{\prime}_n(b)) 
@.H^{a}_{\eet}(X_{\overline{K}},{\bold Z}/p^n(b)).
\end{CD}
$$
Define the morphisms
\begin{equation*}
\alpha^n_{ab}:\quad H^a_{\eet}(X_{\ovk},{\bold Z}/p^n(b))\to 
H^a_{\crr}(X_{{\ovv},n})\{-b\}
\end{equation*}
as the composition
$$
\alpha^n_{ab}(x):=
\psi_n\overline{c}_{b,2b-a}^{\synt}(j^*)^{-1}
D(d,m,b,2b-a)(\overline {c}_{b,2b-a}^{\eet})^{-1}(D(d,m,b,2b-a)x),
$$
where $\psi_n$ is the natural projection
$$
\psi_n: H^a_{\eet}(X_{\ovv},\sss^{\prime}_n(b))\to 
H^a_{\crr}(X_{{\ovv},n}).
$$
Here $(\overline{c}_{b,2b-a}^{\eet})^{-1}(D(d,m,b,2b-a)x)$ is defined by taking 
any element in the preimage of $D(d,m,b,2b-a)x$ (by 
Proposition \ref{diagr}, $D(d,m,b,2b-a)x$ lies in 
the image of $\overline{c}_{b,2b-a}^{\eet}$). By Proposition \ref{diagr},  
any ambiguity in 
that definition comes from a class of $y$ such that $D(d,m,b,2b-a)[y]=[z]$,
$z\in F^{b+1}_{\gamma}K_{2b-a}(X_{\overline{K}},{\bold Z}/p^n)$ and this ambiguity we killed by twisting the definition of $\alpha^n_{ab}$ by a factor of $D(d,m,b,2b-a)$.

 Define the  morphism
$$
\alpha_{ab}:H^a_{\eet}(X_{\ovk},{\bold Q}_p(b))\to
H^a_{\crr}(X_{0})\otimes_{W(k)} \B_{\crr}\{-b \}
$$
as the composition of ${\bold Q}\otimes\invlim_n \alpha^n_{ab}$ 
with the Kato-Messing isomorphism 
$h_{\crr}: 
H^a_{\crr}(X_{{\ovv}})_{\Q}\simeq 
H^a_{\crr}(X_{0})\otimes_{W(k)} \B^+_{\crr}
$
and the division by $D(d,m,b,2b-a)^2$. 

The following  theorem generalizes our \cite[Th. 3.8]{N2} from schemes to finite simplicial schemes.
\begin{theorem}
\label{finalcrystalline}
Let $X$ be any proper smooth finite simplicial ${\so_K}$-scheme. Assume that $X\simeq \skl_mX$ and let  $d=\max_{s\leq m}\dim X_m$. 
  Then, assuming 
$b\geq 2d+2$, the functorial, Galois equivariant morphism
$$
\alpha_{ab}:\quad H^a_{\eet}(X_{\ovk},{\bold Q}_p(b))
\otimes_{{\bold Q}_p} \B_{\crr} \to
H^a_{\crr}(X_{0})\otimes_{W(k)} \B_{\crr}\{-b \} 
$$
is an isomorphism.
Moreover, the map $\alpha_{ab}$ preserves the Frobenius,  is compatible with products and Tate twists,  and, after extension to $\B_{\dr}$, induces 
an isomorphism of filtrations. 
\end{theorem}
\begin{proof}
 We argue by induction on $m\geq 0$. The case $m=0$ is treated by \cite[Th. 3.8]{N2}. Assume that our theorem is true for $m-1$. To show it for $m$ consider the homotopy cofiber sequence
$$
\skl_{m-1}X_{\ovv}\to \skl_mX_{\ovv}\to \skl_{m}X_{\ovv}/\skl_{m-1}X_{\ovv}
$$ 
and apply the period morphisms $\alpha_{*,*}$  to it. We get the   
 following map of  sequences.  
$$
\xymatrix{
 H^{a-1}_{\eet}(\skl_{m-1}X,b)\ar[r] \ar[d]^{\alpha_{a+1,b}}_{\wr} & H^{a-m}_{\eet}(X^{\prime}_{m},b)\ar[d]^{\alpha_{a-m,b}}_{\wr}\ar[r]  & H^a_{\eet}(\skl_mX,b)\ar[r]\ar[d]^{\alpha_{ab}} & H^a_{\eet}(\skl_{m-1}X,b)\ar[r] \ar[d]^{\alpha_{ab}}_{\wr} &
 H^{a-m+1}_{\eet}(X^{\prime}_{m},b)\ar[d]^{\alpha_{a-m+1,b}}_{\wr}\\
 H^{a-1}_{\crr}(\skl_{m-1}X_{0},b)\ar[r] & H^{a-m}_{\crr}(X^{\prime}_{m,0},b)\ar[r] & H^a_{\crr}(\skl_mX_{0},b)\ar[r] & H^a_{\crr}(\skl_{m-1}X_{0},b)\ar[r] & H^{a-m+1}_{\crr}(X^{\prime}_{m,0},b)
}
$$
Here we put $H^*_{\eet}(T,b)=H^{*}_{\eet}(T_{\ovk},{\bold Q}_p(b))\otimes \B_{\crr}$, $H^*_{\crr}(T,b)= H^{*}_{\crr}(T)\otimes \B_{\crr}\{-b \}$. And we defined
$$
H^{*}_{\eet}(X^{\prime}_{m},b) =H^{*}_{\eet}(X_{m},b)\cap\ker s_0^*\cap\cdots\cap\ker s_{m-1}^*,\quad
H^{*}_{\crr}(X^{\prime}_{m,0},b) =H^{*}_{\crr}(X_{m,0},b)\cap\ker s_0^*\cap\cdots\cap\ker s_{m-1}^*,
$$
where each $s_i: X_{m-1}\to X_{m}$ is a degeneracy map. 
The horizontal sequences are  exact by functoriality and finiteness of the \'etale and crystalline cohomologies. By the inductive hypothesis we have the isomorphisms shown in the diagram. Hence the period morphism
$$\alpha_{ab}:\quad H^a_{\eet}(\skl_mX_{\ovk},{\bold Q}_p(b))
\otimes_{{\bold Q}_p} \B_{\crr} \to
H^a_{\crr}(\skl_mX_{0})\otimes_{W(k)} \B_{\crr}\{-b \} 
$$
is an isomorphism. Since $H^a_{\eet}(\skl_mX_{\ovk},{\bold Q}_p(b))\stackrel{\sim}{\to}H^a_{\eet}(X_{\ovk},{\bold Q}_p(b))$ and $H^a_{\crr}(\skl_mX_{0})\stackrel{\sim}{\to}H^a_{\crr}(X_{0})$ this proves the first claim of the theorem.

We will now check that the morphism
$\alpha_{ab}$ is compatible with products. This follows from the fact that 
the morphism
$h_{\crr }$
is compatible with products and from the following  lemma:
\begin{lemma}
\label{above}
Let $x\in H^a(X_{\ovk},{\bold Z}/p^n(b))$, $
y\in H^c(X_{\ovk},{\bold Z}/p^n(e))$, $2b-a>2$, $2e-c>2$, and $p^n\geq 5$.
Set  $K(b,e)=-(b+e-1)!/((b-1)!(e-1)!)$. Then 
(assuming that all the indices are in the valid  range)
\begin{align*}
 K(b,e)D(d,m,b,2b-a)^2 & D(d,m,e,2e-c)^2 \alpha^n_{a+c, b+e}(x\cup y)\\
 & =K(b,e)D(d,m,b+e,2b+2e-a-c)^2\alpha^n_{ab}(x)\cup \alpha^n_{ce}(y).
\end{align*}
\end{lemma}
\begin{proof}
Use   the product formulas from \cite[Lemma 5.3]{NR} and \cite[Rem. 5.4]{NR}.
\end{proof}

 The claim about Tate twists follows from the following computation: 
\begin{lemma}
\label{kwak1}Let $p^n\geq 5$ and $b\geq 2d+2$. 
We have the following relationship between Tate twists
$$(-b)D(d,m,b,2b-a)^2\alpha^n_{a,b+1}(\zeta_n x)
=(-b)D(d,m,b+1,2b+2-a)^2\alpha^n_{ab}(x)t.
$$
\end{lemma}
\begin{proof}
This follows just as in \cite[Lemma 3.6]{N2} from Lemma \ref{above} and the fact 
$\overline{c}_{1,2}^{\eet}(\beta_n)=\zeta_n$ and $\overline{c}_{1,2}^{\synt}(\tilde{\beta}_n)=t$ (see \cite[Lemma 4.1]{N4}). Here $\beta_n\in K_2(\ovk,\Z/p^n)$, $\tilde{\beta}_n\in K_2(\so_{\ovk},\Z/p^n)$ are the Bott elements associated to $\zeta_n$.
\end{proof}

  Now, to prove the claim about filtrations  first we evoke Lemma \ref{KM} that yields compatibility of the period morphism with filtrations and then we note that is  suffices to prove the analog of our claim
  for the associated grading, i.e., that, for $i\in\Z$,  the induced map
  $$
  \alpha_{ab}:\quad H^a_{\eet}(X_{\ovk},{\bold Q}_p(b))
\otimes_{{\bold Q}_p} C(i) {\to}
\bigoplus_{j\in\Z}H^{a-j}(X_{K},\Omega^j_{X_K/K})\otimes_{K} C(i+b-j)
$$
  is an isomorphism. But this can be proved by an analogous argument to the one we used to prove the first claim of the theorem. 
\end{proof}
\begin{example}({\em Rational Crystalline conjecture for  cohomology with compact support.})
  Again, as a special case consider  a smooth proper scheme $X$ over ${\so_K}$ with a divisor $D$. 
We assume $D$ to have relative simple normal crossings and all the irreducible components smooth over ${\so_K}$.
Let $U$ denote  the complement of $D$ in $X$ and $d$ be the relative dimension of $X$. Equip $X$ with the log-structure induced by $D$. Consider the simplicial scheme $C(X,D):=\cofiber(\wt{D}_{{\scriptscriptstyle\bullet}}\stackrel{i}{\to} X),$ where all the schemes have trivial log-structure. We have 
$C(X,D)\simeq \skl_mC(X,D)$, where $m$ is the number of irreducible components of $D$. Applying the above constructions  to $C(X,D)$ we obtain the basic diagram
$$
\begin{CD}
F^b_{\gamma}/F^{b+1}_{\gamma}K^c_{2b-a}(X_{\ovv},D_{\ovv},{\bold Z}/p^n)
@>\sim > j^* > F^b_{\gamma}/F^{b+1}_{\gamma}
K^c_{2b-a}(X_{\overline{K}},D_{\ovk},{\bold Z}/p^n)\\
@VV\overline{c}_{b,2b-a}^{\synt} V @VV \overline{c}_{b,2b-a}^{\eet} V\\
H^{a}_{\eet,c}(X_{\ovv},\sss^{\prime}_n(b)) @. H^{a}_{\eet,c}(U_{\ovk},{\bold Z}/p^n(b)).
\end{CD}
$$
Recall that we have
\begin{align*}
H^{a}_{\eet,c}(X_{\ovv},\sss^{\prime}_n(b)) \simeq H^{a}_{\eet}(C(X_{\ovv},D_{\ovv}),\sss^{\prime}_n(b)).
\end{align*}

   From this we get a Galois-equivariant map
$$\alpha_{ab}:\quad H^a_{\eet,c}(U_{\ovk},{\bold Q}(b)) \to
H^a_{\crr,c}(X_0)\otimes \B_{\crr}\{-b\}
$$
and the following corollary of Theorem \ref{finalcrystalline}.
\begin{corollary}
The Galois equivariant 
morphism 
$$
\alpha_{ab}:H^a_{\eet,c}(U_{\ovk},{\bold Q}_p(b))\otimes \B_{\crr} \to
H^a_{\crr,c}(X_0)\otimes \B_{\crr}\{-b\}
$$
is an isomorphism for $b\geq 2d+2$.  Moreover, the map $\alpha_{ab}$ preserves the Frobenius, is compatible with products and Tate twists,  and, after extension to $\B_{\dr}$, induces 
an isomorphism of filtrations. 
\end{corollary}
\end{example}
\subsection{Semistable  conjecture for cohomology with compact support}\label{diagram22}
We will now  prove a comparison theorem  for cohomology with compact support  in the semistable case using $K$-theory. We start with the definition of the period morphism.
Let $X$ be a proper 
scheme over $\so_K$ with (strictly) semistable reduction and of 
 pure relative dimension $d$. Let 
$i:D\hookrightarrow X$ be the horizontal divisor and set $U=X\setminus D$. Equip $X$ with the log-structure induced by $D$ and the special fiber.
Assume  that $p^n\geq 5$ and $b\geq 2d+2$.
  We will define a period morphism 
$$
\alpha^n_{ab}:\quad H^a_c(U_{\overline{K}},{\mathbf  Z}/p^n(b)) \to 
H^a_{\crr,c}(X_{{\ovv},n})\{-b\}.
$$
We will use the following diagram.
$$
\begin{CD}
F^b_{\gamma}/F^{b+1}_{\gamma}K_{2a-b}^c(\sx_{\ovv},\sd_{\ovv},{\mathbf  Z}/p^n) @>j^*>> 
F^b_{\gamma}/F^{b+1}_{\gamma}K_{2a-b}^c(X_{\overline{K}},D_{\overline{K}},{\mathbf  Z}/p^n)\\
@VV\overline{c}_{b,2a-b}^{\synt}V @VV\overline{c}_{b.2a-b}^{\eet}V \\
H^{a}_{\eet,c}(\sx_{\ovv},\sss^{\prime}_n(b)_X(D)) 
@.H^{a}_{\eet,c}(U_{\overline{K}},{\mathbf  Z}/p^n(b)),
\end{CD}
$$
where  $j: X_{\ovk}\hookrightarrow \sx_{\ovv}$ is the natural open immersion and we set
$$H^{a}_{\eet,c}(\sx_{\ovv},\sss^{\prime}_n(b)_X(D))=
\dirlim_{Y\in \sx_{\ovv}}H^{a}_{\eet}(C(Y,D_Y),\sss^{\prime}_n(b)).
$$
Here the log-structure on the schemes $Y,D_Y$ is trivial.

Define
$$
\alpha^n_{ab}(x):=
\psi_n(\pi^*)^{-1}\varepsilon\overline{c}_{b,2b-a}^{\synt}M(2d,b+1,2(2b-a))(j^*)^{-1}
M(2d,b,2(2b-a))D(d,d,b,2b-a)(\overline {c}_{b,2b-a}^{\eet})^{-1}(D(d,d,b,2b-a)x),
$$
where
$\psi_n(\pi^*)^{-1}\varepsilon$ is the composition 
\begin{align*}
H^a_{\eet,c}(
\sx_{\ovv},\sss^{\prime}_n(b)_X(D)) & \stackrel{\varepsilon}{\to}  H^a_{\eet,c}(
\sx_{\ovv},\sss^{\prime}_n(b))
  \lomapr{(\pi^*)^{-1}}
H^a_{\eet,c}(X_{\ovv},\sss^{\prime}_n(b))\\
 & \stackrel{\psi_n}{\to} 
H^a_{\crr,c}(X_{{\ovv},n})\{-b\},
\end{align*}
where we set
\begin{align*}
 H^a_{\eet,c}(
\sx_{\ovv},\sss^{\prime}_n(b)) =\dirlim_{Y\in\sx_{\ovv}} H^a_{\eet}(C(Y,D_Y),\sss^{\prime}_n(b)).
\end{align*}
Here  the log-structure on the schemes defining $C(Y,D_Y)$ is induced from the special fiber.
The pullback map
$$\pi^*:\quad H^a_{\eet,c}(X_{\ovv},\sss^{\prime}_n(b))
\stackrel{\sim}{\to}H^a_{\eet,c}(
\sx_{\ovv},\sss^{\prime}_n(b))
$$
is an isomorphism by a simplicial (and easy to proof) version of \cite[Corollary 2.4]{N2}.

  In the definition of $\alpha_{ab}^n(x)$, for $x\in  H^a_c(U_{\overline{K}},{\mathbf  Z}/p^n(b))$, we take 
$(\overline{c}^{\eet}_{b,2b-a})^{-1}(D(d,d,b,2b-a)x)\in 
F^b_{\gamma}/F^{b+1}_{\gamma}K_{2b-a}^c
(X_{\overline{K}},D_{\overline{K}},{\mathbf  Z}/p^n)$ to be any
 element in the preimage of $D(d,d,b,2b-a)x$ 
(this is possible by Proposition \ref{diagr}). 
 By Proposition \ref{diagr},  
any ambiguity in 
that definition comes from a class of $y$ such that $D(d,d,b,2b-a)[y]=[z]$,
$z\in F^{b+1}_{\gamma}K_{2b-a}^c(X_{\overline{K}},D_{\overline{K}},{\mathbf  Z}/p^n)$ and that we killed by twisting the definition of $\alpha^n_{ab}$ by a factor of $D(d,d,b,2b-a)$. Similarly,
for $x\in F^b_{\gamma}/F^{b+1}_{\gamma}K_{2b-a}^c(X_{\overline{K}},D_{\overline{K}},{\mathbf  Z}/p^n)$
 we take $(j^*)^{-1}(M(2d,b,2(2b-a))x)$ to be any element in the preimage of $M(2d,b,2(2b-a))x$ under $j^*$. This is possible by
 Lemma \ref{model} and by the same lemma any ambiguity is killed by twisting 
the definition of $\alpha^n_{ab}$ by $M(2d,b+1,2(2b-a))$.

  Let $b\geq 2d+2$.
 We can now define the rational period morphism
$$\alpha_{ab}:\quad H^a_{\eet,c}(U_{\ovk},{\mathbf  Q}_p(b))\to 
H^a_{\crr,c}(X_{0}/W(k)^{0})
\otimes_{W(k)} \B_{\st}\{-b\}
$$
as the composition of 
${\mathbf  Q}\otimes\invlim_n \alpha^n_{ab}$ with the  map \cite[4.2, 4.5]{K1}
$$h_{\pi }:\quad {\mathbf  Q}\otimes\invlim_n H^a_{\crr,c}(X_{{\ovv},n})
\to 
H^a_{\crr,c}(X_{0}/W(k)^0)
\otimes_{W(k)} \B_{\st}
$$
and with the division by $M(2d,b+1,2(2b-a))M(2d,b,2(2b-a))D(d,d,b,2b-a)^2$. 

 The morphism $\alpha_{ab}$ 
preserves  the Frobenius, 
the action of $\Gal(\ovk/K)$ and the monodromy operator, 
and, after extension to $\B_{\dr}$, is compatible with  filtrations 
(use the simplicial analog of Lemma 4.8.4 from \cite{Ts} -- which can be easily shown, as in Section \ref{cohohoho}, by  lifting all the maps functorially to the $\infty$-derived category as was done in detail in \cite{BE2} and \cite{NN}; see also \cite[Sec. 7]{Tc}).

  We have the following generalization of our  \cite[Th. 3.8]{N2} (where the divisor at infinity $D$ is trivial).
\begin{theorem}
\label{finalsemistable}
Let $X$ be a proper  scheme over
 $\so_K$ with  semistable reduction. Let $D$ be the horizontal divisor, let $U=X\setminus D$,  and let $d$ be the relative dimension of $X$. Equip $X$ with the log-structure induced by $D$ and the special fiber. 
Then, assuming 
$b\geq 2d+2$, the  morphism
$$
\alpha_{ab}:H^a_{\eet,c}(U_{\ovk},{\mathbf  Q }_p(b))
\otimes_{{\mathbf  Q }_p} \B_{\st} \to 
H^a_{\crr,c}(X_{0}/W(k)^{0})\otimes_{W(k)} \B_{\st}\{-b\} 
$$
is an isomorphism. The map $\alpha_{ab}$ preserves the Frobenius, 
the action of $\Gal(\ovk/K)$, and the monodromy operator. It is 
independent of the choice of $\pi $ and compatible with products and 
Tate twists. Moreover, 
 after extension to $\B_{\dr}$, it induces 
a filtered  isomorphism 
$$
\alpha_{ab}:H^a_{\eet,c}(U_{\ovk},{\mathbf  Q }_p(b))
\otimes_{{\mathbf  Q }_p} \B_{\dr} \to 
H^a_{\dr,c}(X_K)\otimes_{K} \B_{\dr}\{-b\}
$$
\end{theorem}
\begin{proof}
Consider  the finite  semistable vertical simplicial log-scheme $C=C(X,D)$. The individual schemes in the simplicial scheme are equipped with the log-structure induced from the special fiber. We have $C(X,D)\simeq \skl_mC(X,D)$ if $D$ has $m$ irreducible components. 
We filter $C(X,D)$ by its skeleta
 $\skl_iC(X,D)$ and will show, by induction on $i\geq 0$, that the period morphism\footnote{It is easy to see that the definition of our period morphism 
 extends, in a compatible manner,  to the skeleta of $C(X,D)$.}
$$\alpha_{ab}:H^a_{\eet}(\skl_iC(X,D)_{\ovk},{\mathbf  Q }_p(b))
\otimes_{{\mathbf  Q }_p} \B_{\st} \to 
H^a_{\crr}(\skl_iC(X,D)_{0}/W(k)^{0})\otimes_{W(k)} \B_{\st}\{-b\} 
$$
is an isomorphism. Start with $i=0$ where the statement is known. 
For $i\geq 1$, assume that our theorem is true for $i-1$. To show it for $i$ 
consider the homotopy cofiber sequences
$$
\skl_{i-1}C(Y,D_Y)\to \skl_iC(Y,D_Y)\to \skl_{i}C(Y,D_Y)/\skl_{i-1}C(Y,D_Y)
$$ 
and apply the period morphisms $\alpha_{*,*}$  to it.
We get the   
 following map of exact  sequences.  
$$
\xymatrix{
 H^{a-1}_{\eet}(\skl_{i-1}C,b)\ar[r] \ar[d]^{\alpha_{a+1,b}}_{\wr} & H^{a-i}_{\eet}(C^{\prime}_{i},b)\ar[d]^{\alpha_{a-i,b}}_{\wr}\ar[r]  & H^a_{\eet}(\skl_iC,b)\ar[r]\ar[d]^{\alpha_{ab}} & H^a_{\eet}(\skl_{i-1}C,b)\ar[r] \ar[d]^{\alpha_{ab}}_{\wr} &
 H^{a-i+1}_{\eet}(C^{\prime}_{i},b)\ar[d]^{\alpha_{a-i+1,b}}_{\wr}\\
 H^{a-1}_{\crr}(\skl_{i-1}C_{0},b)\ar[r] & H^{a-i}_{\crr}(C^{\prime}_{i,0},b)\ar[r] & H^a_{\crr}(\skl_iX_{0},b)\ar[r] & H^a_{\crr}(\skl_{i-1}C_{0},b)\ar[r] & H^{a-i+1}_{\crr}(C^{\prime}_{i,0},b)
}
$$
Here we put $H^*_{\eet}(T,*)=H^{*}_{\eet}(T_{\ovk},{\bold Q}_p(b))\otimes \B_{\st}$, $H^*_{\crr}(T,b)=H^{*}_{\crr}(T)\otimes \B_{\st}\{-b \}$. And we defined
\begin{align*}
H^{*}_{\eet}(C^{\prime}_{i},b) & =H^{*}_{\eet}(C_i,b)\cap\ker s_0^*\cap\cdots\cap\ker s_{i-1}^*,\\
H^{*}_{\crr}(C^{\prime}_{i,0},b) & =H^{*}_{\crr}(C_{i,0},b)\cap\ker s_0^*\cap\cdots\cap\ker s_{i-1}^*,
\end{align*}
where each $s_i: \skl_{i-1}C\to \skl_iC$ is a degeneracy map. 
 By the inductive hypothesis we have the isomorphisms shown in the diagram. Hence the period morphism
$$\alpha_{ab}:\quad H^a_{\eet}(\skl_iC_{\ovk},{\bold Q}_p(b))
\otimes_{{\bold Q}_p} \B_{\st} \to
H^a_{\crr}(\skl_iC_{0})\otimes_{W(k)} \B_{\st}\{-b \} 
$$
is an isomorphism. Since $H^a_{\eet}(\skl_mC_{\ovk},{\bold Q}_p(b))\stackrel{\sim}{\to}H^a_{\eet}(C_{\ovk},{\bold Q}_p(b))$ and $H^a_{\crr}(\skl_mC_{0})\stackrel{\sim}{\to}H^a_{\crr}(C_{0})$ this proves the first claim of the theorem. 

 For the claim about the filtrations, we need to show that $\alpha^{\dr}_{ab}$ (that is, $\alpha_{ab}$ extended to $\B_{\dr}$) induced an isomorphism on filtrations. Passing to the associated grading one reduces to showing that the induced Hodge-Tate period map 
 $$
 {\alpha}^{\htt}_{ab}:\quad C\otimes H^a_{\eet}(X_{\ovk},{\bold Q}_p(b)) \to H^a_{\htt}(X_K,b),
$$
where we set
$$
H^a_{\htt}(X_K,b):=\bigoplus_{j\in\Z} C(b-j)\otimes_KH^{a-j}(X_K,\Omega^j_{X_K}),
         $$
 is an isomorphism. But this can be  checked exactly as above.
 
 The claim about the uniformizer  can be checked as in the proof of  \cite[Th. 3.8]{N2}. The claims about products and Tate twists can be checked as in the proof of Theorem \ref{finalcrystalline} using analogs of Lemma \ref{above} and Lemma \ref{kwak1} (where the constants   have to be modified accordingly to the definition of the maps $\alpha^n_{ab}$).
 \end{proof}
\section{Comparison of period morphisms} 
This section has two parts. In the first part we formulate a $K$-theoretical uniqueness criterium for $p$-adic period morphisms  for cohomology with compact support and, using it, we prove that the period morphisms defined using the syntomic, almost \'etale, and motivic methods are equal. In the second part we use $h$-topology and the Beilinson (filtered) Poincar\'e Lemma to formulate a simple uniqueness criterium for $p$-adic period morphisms. Using it, we show that the $p$-adic period morphisms of Faltings, Tsuji (and Yamashita), and Beilinson are the same whenever they are defined (so, in particular, for open varieties with semistable compactifications). Moreover, they are all compatible with (possibly mixed) products. This all holds up to a change of the Hyodo-Kato cohomology described in Section \ref{paris1}.
\subsection{A simple uniqueness criterium}\label{comp2} We start with a very simple uniqueness criterium.
\subsubsection{The case of schemes}
Recall the following formulation of the Semistable conjecture of  Fontaine and Jannsen.
\begin{conjecture}(Semistable conjecture)\label{comp1}
Let $X$ be a  proper, log-smooth,  fine and saturated
 $\so_K^{\times}$-log-scheme with Cartier type reduction.
There exists a  natural $\B_{\st}$-linear Galois equivariant period isomorphism
$$
\alpha_{i}:H^i_{\eet}(X_{\ovk,\tr},{\mathbf Q }_p)
\otimes_{{\mathbf Q }_p} \B_{\st} \stackrel{\sim}{\to}
H^i_{\hk}(X)\otimes_{F} \B_{\st}
$$
that preserves the Frobenius and the monodromy operators,
and, after extension to $\B_{\dr}$, induces
a filtered isomorphism
$$
\alpha_{i}:H^i_{\eet}(X_{\ovk,\tr},{\mathbf Q }_p)
\otimes_{{\mathbf Q }_p} \B_{\dr} \stackrel{\sim}{\to}
H^i_{\dr}(X_K)\otimes_{K} \B_{\dr}.
$$
\end{conjecture}
This conjecture was proved, possibly under additional assumptions,  by Kato \cite{K1}, Tsuji \cite{Ts}, \cite{Ts1}, Yamashita \cite{Ya}, Faltings \cite{Fa}, Nizio{\l} \cite{N2}, and Beilinson \cite{BE2}. It was generalized to formal schemes by Colmez-Nizio{\l} \cite{CN} and by \v{C}esnavi\v{c}ius-Koshikawa \cite{KC} (who generalized the proof of the Crystalline conjecture by  Bhatt-Morrow-Scholze \cite{BMS}) in the case when there is no horizontal divisor.

      Let  $r\geq 0$. For   a period isomorphism $\alpha_i$ as above, we define its twist
$$\alpha_{i,r}:
H^i_{\eet}(X_{\ovk,\tr},{\mathbf Q }_p(r))
\otimes_{{\mathbf Q }_p} \B_{\st} \to
H^i_{\hk}(X)\otimes_{F} \B_{\st}\{-r\}
$$
as $\alpha_{i,r}:=t^r\alpha_{i}\epsilon^{-r}$. Clearly, it  is an isomorphism. It follows from Conjecture \ref{comp1}
 that we can recover the \'etale cohomology with the Galois action from the Hyodo-Kato cohomology:
 \begin{equation}
 \label{formula}
\alpha_{i,r}: H^i_{\eet}(X_{\overline{K},\tr},{\mathbf Q}_p(r))\stackrel{\sim}{\to}
  (H^i_{\hk}(X)\otimes_{F}\B_{\st})^{N=0,\phi=p^r}\cap
F^r(H^i_{\dr}(X_K)\otimes_K\B_{\dr}).
\end{equation}
For $r\geq i$,  by Lemma \ref{lyon-added}, the right hand side is  isomorphic to
 $ H^i_{\eet}(X_{\so_{\overline{K}}},\sss'(r))_{\Q}$, i.e., there exists a natural  isomorphism
 $$
 H^i_{\eet}(X_{\so_{\overline{K}}},\sss'(r))_{\Q}\stackrel{\sim}{\to} (H^i_{\hk}(X)\otimes_{F}\B_{\st})^{N=0,\phi=p^r}\cap
F^r(H^i_{\dr}(X_K)\otimes_K\B_{\dr}).
 $$
We will denote by 
 $$\tilde{\alpha}_{i,r}:H^i_{\eet}(X_{\ovk},{\mathbf Q}_p(r))\stackrel{\sim}{\to} H^i_{\eet}(X_{\so_{\overline{K}}},\sss'(r))_{\Q}
$$
the induced isomorphism and call it the {\em syntomic period isomorphism}.

 The following lemma is immediate.
\begin{lemma}
\label{basic1}
Let $r\geq i$. 
A period isomorphism $\alpha_{i,r}$, hence also a period isomorphism $\alpha_i$ 
 satisfying Conjecture \ref{comp1}, is uniquely determined by the induced syntomic period isomorphism $\tilde{\alpha}_{i,r}$.
\end{lemma}
\subsubsection{The case of simplicial schemes}
 The above discussion carries over to finite simplicial schemes. That is, we assume that we have a period isomorphism $\alpha_i$ as in Conjecture \ref{comp1} but for  a finite simplicial scheme $X$ with components as in Conjecture \ref{comp1}. It then yields an isomorphism $\alpha_{i,r}$ as in (\ref{formula}) for $i\leq r$. We will need the following analog of Lemma \ref{lyon-added}
  \begin{lemma}
 \label{lyon-added1}
 Let $r\geq i$.  There exists a natural isomorphism 
 $$
 H^i_{\eet}(X_{\so_{\overline{K}}},\sss'(r))_{\Q}\stackrel{\sim}{\to} (H^i_{\hk}(X)\otimes_{F}\B_{\st})^{N=0,\phi=p^r}\cap
F^r(H^i_{\dr}(X_K)\otimes_K\B_{\dr}).
 $$
\end{lemma}
\begin{proof}
By functoriality of all the maps involved, the proof of Lemma \ref{lyon-added} yields a quasi-isomorphism
$$
\R\Gamma_{\eet}(X_{\so_{\overline{K}}},\sss'(r))_{\mathbf Q}
 \simeq 
  \xymatrix{[(\R\Gamma_{\hk}(X)\otimes_F{\B}^+_{\st})^{N=0,\phi=p^r}\ar[r]^-{\rho_{\pi}\otimes\iota}  & 
    (\R\Gamma_{\dr}(X_K) \otimes_K\B^+_{\dr})/F^r]}.
    $$
    We have natural isomorphisms
    \begin{align*}
    H^i((\R\Gamma_{\hk}(X)\otimes_F{\B}^+_{\st})^{N=0,\phi=p^r}) & \simeq (H^i_{\hk}(X)\otimes_F{\B}^+_{\st})^{N=0,\phi=p^r},\\
     H^i((\R\Gamma_{\dr}(X_K) \otimes_K\B^+_{\dr})/F^r) & \simeq  (H^i_{\dr}(X_K) \otimes_K\B^+_{\dr})/F^r.
    \end{align*}
    The first isomorphism holds because $H^i_{\hk}(X)\otimes_F{\B}^+_{\st})$ is a  $(\phi,N)$-module (see \cite[proof of Cor. 3.25]{NN} for an argument) and the second one -- because we have a degeneration of the Hodge-de Rham spectral sequence for $X$. This yields a natural long exact sequence
    \begin{align*}
   (H^{i-1}_{\hk}(X)\otimes_F{\B}^+_{\st})^{N=0,\phi=p^r} & \lomapr {\rho_{\pi}\otimes\iota}(H^{i-1}_{\dr}(X_K)   \otimes_K\B^+_{\dr})/F^r \lomapr{\partial} H^i_{\eet}(X_{\so_{\overline{K}}},\sss'(r))_{\mathbf Q}\\
    & \to (H^i_{\hk}(X)\otimes_F{\B}^+_{\st})^{N=0,\phi=p^r}\lomapr {\rho_{\pi}\otimes\iota}(H^i_{\dr}(X_K) \otimes_K\B^+_{\dr})/F^r.
    \end{align*}
    
    It suffices thus to show that, for $i\leq r$, the map $\partial$ in the above exact sequence is zero. Or that
    the map
    $$
    (H^{i-1}_{\hk}(X)\otimes_{F}\B_{\st})^{N=0,\phi=p^r}\lomapr{\rho_{\pi}\otimes\iota_{\pi}}(H^{i-1}_{\dr}(X_K)\otimes_K\B_{\dr})/F^r
    $$
    is surjective. But this follows from the fact that the pair $H^{i-1}_{\hk}(X),H^{i-1}_{\dr}(X_K)$ is an admissible filtered $(\phi,N)$-module such that $F^{r}H^{i-1}_{\dr}(X_K)=0$  (see \cite[Prop. 5.20]{CN}).
    
    As above, we will denote by 
 $$\tilde{\alpha}_{i,r}:H^i_{\eet}(X_{\ovk},{\mathbf Q}_p(r))\stackrel{\sim}{\to} H^i_{\eet}(X_{\so_{\overline{K}}},\sss'(r))_{\Q}
$$
the induced isomorphism and call it the {\em syntomic period isomorphism}.
Again,  the following lemma is immediate.
\begin{lemma}
\label{basic1s}
Let $r\geq i$. 
A period isomorphism $\alpha_{i,r}$, hence also a period isomorphism $\alpha_i$ 
 satisfying Conjecture \ref{comp1} for $X$, is uniquely determined by the induced syntomic period morphism $\tilde{\alpha}_{i,r}$.
\end{lemma}

\end{proof}
 
\subsection{Comparison of period morphisms for cohomology with compact support}
We will prove in this section that the comparison morphisms for cohomology with compact support defined using the syntomic, almost \'etale, and motivic methods are equal. We will use for that a motivic uniqueness criterium.
\subsubsection{ A $K$-theoretical  uniqueness criterium.}We will prove  now a uniqueness criterium for period morphisms that generalizes the one stated in \cite{N10}.
Let $X$ be a proper 
scheme over $\so_K$ with semistable reduction and of 
 pure relative dimension $d$. Let 
$i:D\hookrightarrow X$ be the horizontal divisor and set $U=X\setminus D$. Equip $X$ with the log-structure induced by $D$ and the special fiber.

 \begin{proposition}
\label{unique3}
Let $r\geq 2d+2$. 
There exists a unique semistable period morphism $$
\tilde{\alpha}_{i,r}: H^{i}_{\eet,c}(U_{\ovk},{\mathbf Q}_p(r))\to H^{i}_{\eet}(X_{\so_{\ovk}},S'(r)(D))_{\Q}
$$ that makes the  diagram from Section \ref{diagram22} commute.
\end{proposition}
\begin{proof}Consider the  diagram mentioned and 
use the fact that the \'etale Chern classes $ c_{r,2r-i}^{\eet}$ are isomorphisms rationally by Proposition \ref{diagr} and that the restriction map $j^*$ is an isomorphism by Lemma \ref{model}.
\end{proof}

\subsubsection{Comparison of period morphisms for cohomology with compact support.}
The comparison morphisms of Faltings \cite{Fi}, \cite{Fa} and Tsuji \cite{Ts} extend easily to finite simplicial schemes. This was done explicitly in \cite{Ki}, \cite{Tc}. In particular, they extend to cohomology with compact support. We will show in this section that they are equal to the period morphisms constructed in Section \ref{periodsniziol}. We will use for that the uniqueness criterium for period morphisms stated above. We will do the computations just for cohomology with compact support in the semistable case. The arguments in other cases are analogous.
 \begin{theorem}
 \label{FTN}
\begin{enumerate}
\item 
There exists a unique natural $p$-adic period isomorphism
  $$
  \alpha_i:\quad H^i_{\eet,c}(U_{\overline{K}},{\mathbf Q}_p)\otimes \B_{\st}\stackrel{\sim}{\to} H^i_{\crr,c}(X_0/W(k)^0)\otimes_{W(k)} \B_{\st}
  $$
  such that
  \begin{enumerate}
  \item 
  $ \alpha_i$  is $\B_{\st}$-linear, Galois equivariant, and compatible with Frobenius;
  \item $ \alpha_i $, extended to $\B_{dR}$, induces a filtered isomorphism
 $$
      \alpha^{\dr}_i:\quad H^i_{\eet,c}(U_{\overline{K}},{\mathbf Q}_p)\otimes \B_{\dr}\stackrel{\sim}{\to} H^i_{\dr,c}(X_K)\otimes_{K} \B_{\dr};
$$
\item   $ \alpha_i$ 
  is compatible with the \'etale and syntomic higher Chern classes from $p$-adic $K$-theory.
  \end{enumerate}
  \item
  The period morphisms of Faltings, Tsuji, and Nizio{\l} are equal\footnote{By {\em Nizio{\l} period morphisms} we mean the morphisms defined in Section \ref{periodsniziol}.}.
  \end{enumerate} 
 \end{theorem}
 \begin{proof}The first claim follows from Proposition \ref{unique3} and Lemma \ref{basic1s}. 
 
 For the second claim, 
 choose $r$ such that  $r\geq 2d+2$ and  $r\geq i$. It suffices to show that the Faltings, Tsuji, and Nizio{\l}
period morphisms $\alpha^{F}_{i,r}$, $\alpha^{T}_{i,r}$, and $\alpha^{N}_{i,r}$
$$\alpha^*_{i,r}:\quad H^i_{\eet,c}(U_{\ovk},{\mathbf Q}_p(r))\otimes \B_{\st}\stackrel{\sim}{\to} H^i_{\crr,c}(X_0/W(k)^0)\otimes _{W(k)}\B_{\st}\{-r\}
$$
and their de Rham analogous 
are equal. For that apply the first claim.
The needed  compatibility of the  period morphism with higher $p$-adic Chern classes
is clear in the case of the map $\alpha^{N}_{i,r}$ and was proved
 in \cite[Corollary 4.14, Corollary 5.9]{N10} for the other two maps. These corollaries are stated for proper log-schemes but their proofs carry over to the case of finite simplicial schemes (with the same properties).

 \end{proof}

\subsection{Comparison of Tsuji and Beilinson  period morphisms}
We prove in the next two  sections that Beilinson period morphisms \cite{BE1}, \cite{BE2} agree with the period morphisms of Faltings and Tsuji whenever the latter are defined (and modulo a change of Hodo-Kato cohomology). Our strategy is to appeal to  Lemma \ref{basic1} and then to sheafify the syntomic morphisms induced by the latter period morphisms in the $h$-topology on the generic fiber. We   identify the syntomic period morphisms on  the sheaf level as certain  canonical maps appearing in  the fundamental exact sequence. Since we had shown in \cite{NN} that the same maps are used to define the Beilinson syntomic period morphism, it follows that all the period morphisms are equal. Along the way we obtain useful properties of the Faltings and Tsuji period morphisms.

  We start with comparing the period morphisms of Tsuji and Beilinson. 
\subsubsection{Tsuji period morphism.} We will briefly discuss  the period morphism used by Tsuji. 
Let $X$ be a log-smooth log-scheme over $\so_K^{\times}$. Recall that Fontaine-Messing and Kato have  defined 
 natural period morphisms on the \'etale site of $X_0$ \cite{FM}, \cite{Tc}
\begin{align*}
\beta_{r}^{\rm T}:  \sss_n(r)  \rightarrow i^*\R j_*{\mathbf Z}/p^n(r)',\quad r\geq 0,
\end{align*}
where $i:X_0\hookrightarrow X, j: X_K\hookrightarrow X$ are the natural immersions. Here, we set
 ${\mathbf Z}/p^n(r)^{\prime}:=(1/(p^aa!){\mathbf Z}_p(r))\otimes{\mathbf Z}/p^n$, where $a$ is the largest integer $\leq r/(p-1)$.  Recall that we have 
  the fundamental exact sequence \cite[Th. 1.2.4]{Ts}
 $$0\to {\mathbf Z}/p^n(r)^{\prime}\to J_{\crr,n}^{<r>}\lomapr{1-\phi_r}\A_{\crr,n}\to 0,
 $$
 where  $$J_n^{<r>}:= \{x\in J_{n+s}^{[r]}\mid \phi(x)\in p^r\A_{\crr,n+s}\}/p^n ,$$
for some $s\geq r$.

 The above  period morphisms  were used to prove the following comparison theorem.
 \begin{theorem}{\rm (Tsuji, \cite[3.3.4]{Ts})}
 \label{comp23}
 Let $X$ be a semistable scheme over $\so_K$ or a finite base change of such a scheme. Then,
 for any $0\leq i\leq r$, the kernel  and cokernel of the  period morphism
$$
\beta_{r}^{\rm T}: \sh^i(\sss_n(r)_{\overline{X}}) \rightarrow \overline{i}^*\R^i\overline{j}_*{\mathbf Z}/p^n(r)'_{X_{\ovk,\tr}},
$$
is annihilated by $p^N$ for an integer $N$ which depends only on $p$, $r$, and $i$. Here, $\overline{i}$ and $\overline{j}$ are extensions of $i$ and $j$ to $\overline{X}:=X_{\so_{\ovk}}$.
\end{theorem}
For a  proper semistable scheme $X$ over $\so_K$ and $r\geq i$, the modulo $p^n$ and rational semistable  Tsuji period morphisms are defined as 
\begin{align}
\label{tea11}
& \beta_{r,n}^{\rm T}:
\rg_{\eet}(X_{\so_{\ovk}},\sss^{\prime}_n(r))\lomapr{\can} \rg_{\eet}(X_{\so_{\ovk}},\sss_n(r)) \lomapr{\beta^{\rm T}_r} 
\rg_{\eet}(X_{\ovk,\tr},\Z/p^n(r)^{\prime}),\\
& \beta_{r}^{\rm T}:
\rg_{\eet}(X_{\so_{\ovk}},\sss^{\prime}(r))_{\Q} \lomapr{\can} \rg_{\eet}(X_{\so_{\ovk}},\sss(r))_{\Q} \lomapr{\beta^{\rm T}_r} \notag
\rg_{\eet}(X_{\ovk,\tr},{\mathbf Q}_p(r))\stackrel{ p^{-r}}{\to}  \rg_{\eet}(X_{\ovk,\tr},{\mathbf Q}_p(r)).
\end{align}
By Theorem \ref{comp23}, it is a quasi-isomorphism after truncation at $\tau_{\leq r}$. 

Tsuji period morphism $$\alpha^T_{i,r}: H^i_{\eet}(X_{\ovk,\tr},\Q_p(r)){\to} (H^i_{\hk}(X)\otimes_F\B_{\st})^{N=0,\phi=p^r}$$ is defined by composing the above morphism with the map $h_{\pi}$ (and changing $\wh{\B}_{\st}$ to $\B_{\st}$). 
\subsubsection{Beilinson comparison theorem}\label{paris1}In \cite{BE2} Beilinson proved the following comparison theorem.
\begin{theorem}{\rm (Semistable conjecture, \cite{BE2})}\label{comp1B}
Let $X$ be a  proper semistable scheme over
 $\so_K$ endowed with its canonical log-structure. 
There exists a  natural $\B_{\st}$-linear Galois equivariant period isomorphism
$$
\alpha^B_{h,i}:H^i_{\eet}(X_{\ovk,\tr},{\mathbf Q }_p)
\otimes_{{\mathbf Q }_p} \B_{\st} \stackrel{\sim}{\to}
H^{B,i}_{\hk}(X)\otimes_{F} \B_{\st}
$$
that preserves the Frobenius and the monodromy operators,
and, after extension to $\B_{\dr}$, induces
a filtered isomorphism
$$
\alpha^B_{i}:H^i_{\eet}(X_{\ovk,\tr},{\mathbf Q }_p)
\otimes_{{\mathbf Q }_p} \B_{\dr} \stackrel{\sim}{\to}
H^i_{\dr}(X_K)\otimes_{K} \B_{\dr}.
$$
\end{theorem}
We added the subscript $h$ (for $h$-topology) to underscore the different formulation from Theorem \ref{comp1}. Here, $H^{B,i}_{\hk}(X)$ is the Beilinson Hyodo-Kato cohomology \cite[1.16.1]{BE2} and the base change  to the de Rham comparison uses the Beilinson Hyodo-Kato isomorphism
$$
\rho^B: H^{B,i}_{\hk}(X)\otimes_FK\stackrel{\sim}{\to} H^i_{\dr}(X_K)
$$
as well as the canonical map $\iota_p:\B_{\st}\to\B_{\dr}$ \cite[Sec. 2.1]{NN}. 
A priori, this Hyodo-Kato-type constructions are not the same as the original ones (for one thing, they are independent of the choice of the uniformizer $\pi$; in fact, they should be seen, in a sense that can be made precise,  as associated to the canonical choice of $p$). However the two constructions are related by  a natural quasi-isomorphism, i.e., there is a natural map $\kappa$ that makes the following diagram commute \cite[(31)]{NN}
$$
\xymatrix{
H^i_{\hk}(X)\ar[r]^{\rho_{\pi}} &  H^i_{\dr}(X_K)\\
H^{B,i}_{\hk}(X)\ar[u]^{\kappa}_{\wr} \ar[ur]_{\rho^B}
}
$$
\subsubsection{Beilinson equivalence of topoi.} To describe Beilinson period morphism we will need to work with $h$-topology on the generic fiber. Beilinson has shown that $h$-topology has a base consisting of semistable schemes. We will review his result briefly. 

 For a field $K$, let  $ {\mathcal V}ar_K$ denote the category of varieties over $K$. We will equip it with $h$-topology (see \cite[2.3]{BE1}), i.e.,   the coarsest topology finer than the Zariski and proper topologies.\footnote{The latter is generated by a pretopology whose coverings are proper surjective maps.} We note that the $h$-topology  is finer   than the \'etale topology. It is generated
  by the pretopology whose coverings are finite families of maps $\{Y_i\to X\}$ such that $Y:=\coprod Y_i\to X$ is a universal topological epimorphism
  (i.e., a subset of $X$ is Zariski open if and only if its preimage in $Y$ is open). We denote by $\mathcal{V}ar_{K,h},X_h$, $X\in \mathcal{V}ar_{K}$,  the corresponding $h$-sites.

 Let $K$ be now  as in Setion \ref{preliminaries}. An {\em arithmetic pair} over $K$ is an open embedding $j:U\hookrightarrow \overline{U}$ with dense image of a $K$-variety $U$ into a reduced proper flat $V$-scheme $\overline{U}$. A morphism $(U,\overline{U})\to (T,\overline{T})$ of pairs is a map $\overline{U}\to \overline{T}$ which sends $U$ to $T$. In the case that the pairs represent log-regular schemes this is the same as a map of log-schemes.
 For a pair $(U,\overline{U})$, we set $V_U:=\Gamma(\overline{U},\so_{\overline{U}})$ and $K_U:=\Gamma(\overline{U}_K,\so_{\overline{U}})$. $K_U$  is a product of several finite extensions of $K$ (labeled by the connected components of $\overline{U}$) and, if $\overline{U}$ is normal, $V_U$ is the product of the corresponding rings of integers. 
 
  A  {\em semistable pair}  over $K$ \cite[2.2]{BE1} is a pair of schemes $(U,\overline{U})$ over $(K,V)$ such that
 (i) $\overline{U}$ is regular and proper over $V$, (ii) $\overline{U}\setminus U$ is a divisor with normal crossings on $\overline{U}$, and (iii) the closed fiber $\overline{U}_0$ of $\overline{U}$ is reduced and its   irreducible components are regular.
Closed fiber is taken over the closed points of $V_U$.  We will think of semistable pairs as log-schemes  equipped  with log-structure given by the divisor $\overline{U}\setminus U$. The closed fiber $\overline{U}_0$ has the induced log-structure.     

 A {\em semistable pair}  over $\ovk$ \cite[2.2]{BE1} is a pair of connected schemes $(T,\overline{T})$ over $(\ovk,\overline{V})$ such that there exists a semistable pair $(U,\overline{U})$ over $K$ and a $\ovk$-point $\alpha: K_U\to \ovk$ such that $(T,\overline{T})$ is isomorphic to the base change $(U_{\ovk},\overline{U}_{\overline{V}})$.
  We will denote by $\spp_{\ovk}^{ss}$ the category of semistable pairs over $\ovk$.

   Let, for just a moment, $K$ be any field of characteristic zero. A {\em geometric pair} over $K$ is a pair $(U,\overline{U})$ of varieties over $K$ such that $\overline{U}$ is proper and $U\subset \overline{U}$ is open and dense. We say that the pair $(U,\overline{U})$ is a {\em nc-pair}  if $\overline{U}$ is regular and $\overline{U}\setminus U$ is a divisor with normal crossings in $\overline{U}$; it is {strict nc-pair} if the irreducible components of $U\setminus \overline{U}$ are regular.  A morphism of pairs $f:(U_1,\overline{U}_1)\to (U,\overline{U})$ is a map $\overline{U}_1\to \overline{U}$ that sends $U_1$ to $U$. We denote the category of nc-pairs over $K$ by $\spp_K^{\rm nc}$.
   
   For the category $\spp_{\ovk}^{ss}$ mentioned above let $\gamma: (U,\overline{U})\to U$ denote the forgetful functor.  Beilinson proved \cite[2.5]{BE1} that the category $(\spp_{\ovk}^{ss},\gamma)$ forms a base for $\mathcal{V}ar_{\ovk,h}$. This implies that $\gamma$ induces an equivalence of the topoi
  $$
 \gamma: \,  {\rm Shv}_h(\spp_{\ovk}^{ss})\stackrel{\sim}{\to}  {\rm Shv}_h(\mathcal{V}ar_{\ovk}).
  $$
Similarly, for the categories $\spp_{K}^{ss}$ and $\spp_K^{\rm nc}$ (and the category $\mathcal{V}ar_{K}$).

\subsubsection{Definitions of cohomology sheaves.}
We will now recall briefly the definition of  geometric syntomic cohomology, i.e., syntomic cohomology over $\ovk$, from \cite{NN}, and the related cohomologies from \cite{BE2}. 

 (i) {\em Absolute crystalline cohomology.}
  For $(U,\overline{U})\in \spp^{ss}_{\ovk}$, $r\geq 0$,  we have the absolute crystalline cohomology complexes and their completions
\begin{align*}
\R\Gamma_{\crr}(U,\overline{U},\sj^{[r]})_n: &   =\R\Gamma_{\crr}(\overline{U}_{n,\eet},\R u_*\sj^{[r]}),\quad
\R\Gamma_{\crr}(U,\overline{U},\sj^{[r]}):    =
\holim_n\R\Gamma_{\crr}(U,\overline{U},\sj^{[r]})_n,\\
\R\Gamma_{\crr}(U,\overline{U},\sj^{[r]})_\Q: &    =\R\Gamma_{\crr}(U,\overline{U},\sj^{[r]})\otimes\Q_p,
\end{align*}
where $u: \overline{U}_{n,\crr}\to \overline{U}_{n,\eet}$ is the natural projection. 
 The complex $\R\Gamma_{\crr}(U,\overline{U})$ is a perfect $\A_{\crr}$-complex and $$\R\Gamma_{\crr}(U,\overline{U})_n\simeq \R\Gamma_{\crr}(U,\overline{U})\otimes^{L}_{\A_{\crr}}{\A_{\crr}}/p^n\simeq \R\Gamma_{\crr}(U,\overline{U})\otimes^{L}{\mathbf Z}/p^n. $$
 In general, we have
$\R\Gamma_{\crr}(U,\overline{U},\sj^{[r]})_n\simeq \R\Gamma_{\crr}(U,\overline{U},\sj^{[r]})\otimes^{L}{\mathbf Z}/p^n$.
Moreover,  by \cite[1.6.3,1.6.4]{Ts}, 
$$J^{[r]}_{\crr}=\R\Gamma_{\crr}(\Spec(\ovk),\Spec(\overline{V}),\sj^{[r]}).$$ The absolute  crystalline cohomology
 complexes are filtered $E_{\infty}$ algebras over $\A_{\crr,n}$, $\A_{\crr}$, or $\A_{\crr,\Q}$, respectively. Moreover, the rational ones are filtered commutative dg algebras.
 
   Let $\sj^{[r]}_{\crr}$ and $\sa_{\crr}$ be the $h$-sheafifications on $\mathcal{V}ar_{\ovk}$ of the presheaves sending $(U,\overline{U})\in \spp^{ss}_{\ovk}$ to $
\R\Gamma_{\crr}(U,\overline{U},\sj^{[r]})$ and  $
\R\Gamma_{\crr}(U,\overline{U})$, respectively.  Let $\sj^{[r]}_{\crr,n}$ and  $\sa_{\crr,n}$  denote the $h$-sheafifications of
the mod-$p^n$ versions of the respective presheaves; and let $\sj^{[r]}_{\crr,\Q}$ and  $\sa_{\crr,\Q}$ be the $h$-sheafifications
of the rational versions of the same presheaves.

  For $X\in \mathcal{V}ar_{\ovk}$, set $\R\Gamma_{\crr}(X_h):=\R\Gamma(X_h,\sa_{\crr})$. It is a filtered (by $\R\Gamma(X_h,\sj^{[r]}_{\crr})$, $r\geq 0$,) $E_{\infty}$  $\A_{\crr}$-algebra equipped with the Frobenius action $\phi$. The Galois group $G_K$ acts on ${\mathcal V}ar_{\ovk}$ and it acts on $X\mapsto \R\Gamma_{\crr}(X_h)$ by transport of structure. If $X$ is defined over $K$ then $G_K$ acts naturally on $\R\Gamma_{\crr}(X_h)$. 

 (ii) {\em Geometric syntomic cohomology.} 
For $r\geq 0$, the mod-$p^n$, completed, and rational syntomic complexes $\R\Gamma_{\synt}(U,\overline{U},r)_n$, $\R\Gamma_{\synt}(U,\overline{U},r)$, and $\R\Gamma_{\synt}(U,\overline{U},r)_\Q$ are defined by the formulas:
\begin{align*}
\R\Gamma_{\synt}(U,\overline{U},r)_n & :=
[\R\Gamma_{\crr}(U,\overline{U},\sj^{[r]})_n\verylomapr{p^r-\phi} \R\Gamma_{\crr}(U,\overline{U})_n)],\\
  \R\Gamma_{\synt}(U,\overline{U},r) & :=\holim _n\R\Gamma_{\synt}(U,\overline{U},r)_n,\\
 \R\Gamma_{\synt}(U,\overline{U},r)_{\mathbf Q} & :=
[\R\Gamma_{\crr}(U,\overline{U},\sj^{[r]})_\Q\verylomapr{1-\phi_r} \R\Gamma_{\crr}(U,\overline{U})_\Q)].
\end{align*}
We have $\R\Gamma_{\synt}(U,\overline{U},r)_n\simeq \R\Gamma_{\synt}(U,\overline{U},r)\otimes^{L}{\mathbf Z}/p^n$.
  Let $\sss^{\prime}(r)$  be the $h$-sheafification on $\mathcal{V}ar_{\ovk}$ of the presheaf sending $(U,\overline{U})\in \spp^{ss}_{\ovk}$ to 
$\R\Gamma_{\synt}(U,\overline{U},r)$.  Let $\sss^{\prime}_n(r)$ and $\sss^{\prime}(r)_{\Q}$ denote the $h$-sheafifications of
the mod-$p^n$ and the rational versions of the same  presheaf, respectively.

   For $r\geq 0$, set 
$\R\Gamma_{\synt}(X_h,r)_n=\R\Gamma(X_h,\sss^{\prime}_n(r))$, $\R\Gamma_{\synt}(X_h,r):=\R\Gamma(X_h,\sss^{\prime}(r)_\Q)$. We have 
\begin{align*}
\R\Gamma_{\synt}(X_h,r)_n & \simeq [\R\Gamma(X_h,\sj^{[r]}_{\crr,n})\stackrel{p^r-\phi}{\longrightarrow}\R\Gamma(X_h,\sa_{\crr,n})],\\
 \R\Gamma_{\synt}(X_h,r) & \simeq [\R\Gamma(X_h,\sj^{[r]}_{\crr,{\mathbf Q}})\stackrel{1-\phi_r}{\longrightarrow}\R\Gamma(X_h,\sa_{\crr,{\mathbf Q}})] .
\end{align*}
The direct sum $\bigoplus_{r\geq 0}\R\Gamma_{\synt}(X_h,r)$ is a graded $E_{\infty}$ algebra over ${\mathbf Z}_p$.

(iii) {\em de Rham cohomology.}  Consider the presheaf $(U,\overline{U})\mapsto \R\Gamma_{\dr}(U,\overline{U}):=\R\Gamma(\overline{U},\Omega^{\scriptscriptstyle\bullet}_{(U,\overline{U})})$ of filtered dg $K$-algebras on $\spp^{nc}_K$. Let $\sa_{\dr}$ be its $h$-sheafification. It is a sheaf of filtered $K$-algebras on $\mathcal{V}ar_K$. For $X\in \mathcal{V}ar_K$, we have Deligne's de Rham complex of $X$ equipped with Deligne's Hodge filtration: $\R\Gamma_{\dr}(X_h):=\R\Gamma(X_h,\sa_{\dr})$. 

(iv) {\em Beilinson-Hyodo-Kato cohomology.} 
  Let $\sa^B_{\hk}$ be the $h$-sheafification of the presheaf $(U,\overline{U})\mapsto \R\Gamma^B_{\hk}(U,\overline{U})_{\mathbf Q}$ of (arithmetic) Beilinson-Hyodo-Kato cohomology on $\spp_{K}^{ss}$; this is an $h$-sheaf of $E_{\infty}$ $ F$-algebras on ${\mathcal V}ar_{K}$ equipped with a $\phi$-action and a derivation $N$ such that $N\phi=p\phi N$. For $X\in {\mathcal V}ar_{K}$, set $\R\Gamma^B_{\hk}(X_h):=\R\Gamma(X_h,\sa^B_{\hk})$. 

   Let $\sa^B_{\hk}$ be $h$-sheafification of the presheaf $(U,\overline{U})\mapsto \R\Gamma^B_{\hk}(U,\overline{U})$ of (geometric) Beilinson-Hyodo-Kato cohomology on $\spp^{ss}_{\ovk}$. This is an $h$-sheaf of $E_{\infty}$ $F^{\rm nr}$-algebras, where $F^{\rm nr}$ is the maximal unramified extension of $F$, equipped with a $\phi$-action and locally nilpotent derivation $N$ such that $N\phi=p\phi N$. For $X\in\mathcal{V}ar_{\ovk}$, set $\R\Gamma^B_{\hk}(X_h):=\R\Gamma(X_h,\sa^B_{\hk}). $ We have the Beilnson-Hyodo-Kato quasi-isomorphism
   $$
  \rho_h^B:  \rg_{\hk}^B(X_h)\otimes_{F^{\rm nr}}\ovk\stackrel{\sim}{\to} \rg_{\dr}(X_h).
   $$

(v) {\em Comparison statements.} The $h$-topology definitions of cohomology are often compatible with the original definitions. 
\begin{lemma}\label{faithful}We have the following comparison statements:
\begin{enumerate}
\item
 For $(U,\overline{U})\in \spp^{nc}_L$, $L=K,\ovk$, the canonical map $\R\Gamma_{\dr}(U,\overline{U})\stackrel{\sim}{\to} \R\Gamma_{\dr}(U_h)$ is a filtered quasi-isomorphism \cite[2.4]{BE1}.
 \item For any $(U,\overline{U})\in \spp^{ss}_{\ovk}$, $r\geq 0$, the canonical maps 
\begin{equation*}
\R\Gamma_{\crr}(U,\overline{U},\sj^{[r]})_{\mathbf Q}\stackrel{\sim}{\to} \R\Gamma(U_h,\sj^{[r]}_{\crr})_{\mathbf Q},\quad \R\Gamma^B_{\hk}(U,\overline{U})\stackrel{\sim}{\to}\R\Gamma^B_{\hk}(U_h)
\end{equation*} are quasi-isomorphisms  \cite[2.4]{BE2}, \cite[Prop. 3.21]{NN}. In particular, 
$$ \R\Gamma_{\synt}(U,\overline{U},r)\stackrel{\sim}{\to}\R\Gamma_{\synt}(U_h,r).$$
\item For any arithmetic pair $(U,\overline{U})$ that is fine, log-smooth over $\so_K^{\times}$, and of Cartier type, the canonical map
\begin{equation*}
\R\Gamma^B_{\hk}(U,\overline{U})\stackrel{\sim}{\to}\R\Gamma^B_{\hk}(U_h).
\end{equation*} is a  quasi-isomorphism \cite[Prop. 3.18]{NN}. 
 \end{enumerate}
 \end{lemma}

 \subsubsection{Poincar\'e Lemma.} We will recall the Poincar\'e Lemma of Beilinson \cite{BE2} and its syntomic cohomology version \cite{NN}. 
\begin{theorem}{\rm (Filtered Crystalline Poincar\'e Lemma \cite[2.3]{BE2}, \cite[Th. 10.14]{BH})} Let $r\geq 0$. 
The canonical map $J^{[r]}_{\crr,n}\to \sj^{[r]}_{\crr,n}$ is a quasi-isomorphism of $h$-sheaves on ${\mathcal V}ar_{\ovk}$.
\end{theorem}
  
  Set
 $S^{\prime}_n(r):=\Cone(J^{[r]}_{\crr,n}\lomapr{p^r-\phi} \A_{\crr,n})[-1]$.   There is a natural morphism of complexes 
 $\tau_n: S^{\prime}_n(r)\to{\mathbf Z}/p^n(r)^{\prime}$ (induced by $p^r $ on $J_{\crr,n}^{[r]}$ and $\id$ on $\A_{\crr,n}$) , whose kernel and cokernel are annihilated by $p^r$.
The Filtered Crystalline Poincar\'e Lemma implies easily the following Syntomic Poincar\'e Lemma.
 \begin{corollary}
 \label{syntomicP}
There is a  unique quasi-isomorphism $S^{\prime}_n(r)\stackrel{\sim}{\to}\sss^{\prime}_n(r)$ of complexes of sheaves on ${\mathcal V}ar_{\ovk,h}$ that is compatible with the Crystalline Poincar\'e Lemma.
\end{corollary}
\begin{proof}We include here the simple proof from \cite[Cor. 4.5]{NN}. 
Consider the following map of distinguished triangles
$$
\xymatrix{
\sss^{\prime}_n(r)\ar[r] & \sj^{[r]}_{\crr,n}\ar[r]^{p^r-\phi} & \sa_{\crr,n}\\
S^{\prime}_n(r)\ar[r]\ar@{-->}[u]& J^{[r]}_{\crr,n}\ar[u]^{\wr}\ar[r]^{p^r-\phi} & \A_{\crr,n}\ar[u]^{\wr}
}
$$
The  triangles are distinguished by definition.
The vertical continuous arrows are quasi-isomorphisms by the Crystalline Poincar\'e Lemma. They induce the dash arrow  that is clearly a quasi-isomorphism.
\end{proof}

 \subsubsection{Beilinson period morphism.} \label{Beilinson1} We will now recall the definition of the period morphism of Beilinson \cite[3.1]{BE2}. Let $X\in {\mathcal V}ar_{\ovk}$. Recall first the definition of the crystalline period morphism \cite{BE2}
 $$\beta^{\rm B}_{\crr}: \R\Gamma_{\crr}(X_h)\to \R\Gamma(X_{\eet},{\mathbf Z}_p)\what{\otimes} \A_{\crr}.$$ 
 Consider the natural map
$\pi_n: \R\Gamma_{\crr}(X_h)\to \R\Gamma(X_h,\sa_{\crr,n})$ and take  the composition
$$\rho_n: \quad \R\Gamma(X_{\eet},{\mathbf Z}_p)\otimes^{L}_{{\mathbf Z}_p}\A_{\crr,n}\stackrel{\sim}{\to}\R\Gamma(X_{\eet},\A_{\crr,n})
\stackrel{\sim}{\to}\R\Gamma(X_{h},\A_{\crr,n})
\stackrel{\sim}{\to} \R\Gamma(X_{h},\sa_{\crr,n}).
$$ 
 Set $\beta^{\rm B}_{\crr,n}:=\rho_n^{-1}\pi_n$ and $\beta^{\rm B}_{\crr}:=\holim_n\beta^{\rm B}_{\crr,n}$.
 
  The Beilinson Hyodo-Kato period map $$\beta_{\hk}:\R\Gamma^B_{\hk}(X_h)\otimes^L_{F^{\rm nr}}\B^+_{\st}\to \R\Gamma(X_{\eet},{\mathbf Q}_p)\otimes^L \B^+_{\st},\quad  \beta_{\hk}:=\beta_{\crr,{\mathbf Q}}\iota^B_{\st},
$$
 is obtained by composing the map $\beta_{\crr, {\mathbf Q}}$ with the quasi-isomorphism 
 $\rho^B_{\crr}: \R\Gamma^B_{\hk}(X_h)\otimes^L_{F^{\rm nr}}\B^+_{\st}\stackrel{\sim}{\to} \R\Gamma_{\crr}(X_h)_{{\mathbf Q}}$. We have the induced  quasi-isomorphism $\beta_{\hk}:\R\Gamma^B_{\hk}(X_h)\otimes^L_{F^{\rm nr}}\B_{\st}\to \R\Gamma(X_{\eet},{\mathbf Q}_p)\otimes^L \B_{\st}$ and we set $\alpha^B_h:=\beta_{\hk}^{-1}$.

 The Beilinson de Rham period map  $\beta_{\dr}:\R\Gamma_{\dr}(X_h)\otimes^L_{\ovk}\B_{\dr}\to \R\Gamma(X_{\eet},{\mathbf Q}_p)\otimes^L \B_{\dr}$ is obtained from the Beilinson-Hyodo-Kato period map $\beta_{\hk}$ using the canonical map $\iota_p: \B_{\st}\to \B_{\dr}$ and the Beilinson-Hyodo-Kato isomorphism $\rho_{\hk}: \rg^B_{\hk}(X_h)\otimes^L_{F^{\rm nr}}\ovk\stackrel{\sim}{\to}\rg_{\dr}(X_h)$. We set $\alpha^B:=\beta_{\dr}^{-1}$.

 The induced  syntomic period morphism $$\beta_{r}^{\rm B}:\, \R\Gamma_{\synt}(X_h,r)\to \R\Gamma(X_{\eet},{\mathbf Q}_p(r)),\quad r\geq 0$$
can be described in the following way. 
 Take the natural map
$\pi_n: \R\Gamma(X_h,\sss^{\prime}(r))\to \R\Gamma(X_h,\sss^{\prime}_n(r))$ and the zigzag
$$\beta^{\rm B}_n:\, \R\Gamma(X_h,\sss^{\prime}_n(r))\stackrel{\sim}{\leftarrow} \R\Gamma(X_{h},S^{\prime}_n(r))\lomapr{\tau_n}
\R\Gamma(X_{h},{\mathbf Z}/p^n(r)^{\prime}) \stackrel{\sim}{\leftarrow}\R\Gamma(X_{\eet},{\mathbf Z}/p^n(r)^{\prime}).$$ Set $\beta^{\rm B}:=(\holim_n\beta^{\rm B}_{n})\otimes {\mathbf Q}$.  Then the map $$
 \wt{\beta}_{h,r}^{\rm B}:=p^{-r}\beta^{\rm B}\pi:\quad  \R\Gamma_{\synt}(X_h,r)\to \R\Gamma(X_{\eet},{\mathbf Q}_p(r)),
 $$
 where $\pi:=(\holim_n\pi_{n})\otimes {\mathbf Q}$, is the induced syntomic period morphism.  By \cite[Prop. 4.6]{NN}, it is an isomorphism after truncation $\tau_{\leq r}$.
\begin{remark}\label{padova1}
It is worth looking carefully at the composition
$$
\beta^{\rm B}\pi:\R\Gamma_{\synt}(X_h,r)\lomapr{\pi} (\holim_n\R\Gamma(X_h,\sss^{\prime}_n(r)))_\Q  \lomapr{\beta^{\rm B}} \R\Gamma(X_{\eet},{\mathbf Q}_p(r)).
$$
This composition is a quasi-isomorphism after truncation $\tau_{\leq r}$. Since, by Corollary \ref{syntomicP}, the second map is a quasi-isomorphism, it follows that the first map is a quasi-isomorphism after truncation $\tau_{\leq r}$ as well. 
\end{remark}

\subsubsection{A very simple comparison criterium} This is an analog of the criterium in Lemma \ref{basic1} in the context of Beilinson comparison morphisms from Theorem \ref{comp1B}. 

 Let $X$ be a proper, log-smooth, fine and saturated $\so_K^{\times}$-log-scheme with Cartier type reduction. Let  $r\geq 0$. For   a period isomorphism $\alpha_{h,i}$ as in Theorem \ref{comp1B}, we define its twist
$$\alpha_{h,i,r}:
H^i_{\eet}(X_{\ovk,\tr},{\mathbf Q }_p(r))
\otimes_{{\mathbf Q }_p} \B_{\st} \to
H^{B,i}_{\hk}(X)\otimes_{F} \B_{\st}\{-r\}
$$
as $\alpha_{h,i,r}:=t^r\alpha_{h,i}\epsilon^{-r}$. Clearly, it  is an isomorphism. It follows from Theorem \ref{comp1B}
 that we can recover the \'etale cohomology with the Galois action from the Beilinson-Hyodo-Kato cohomology:
 \begin{equation}
 \label{formulaB}
\alpha_{h,i,r}: H^i_{\eet}(X_{\overline{K},\tr},{\mathbf Q}_p(r))\stackrel{\sim}{\to}
  (H^{B,i}_{\hk}(X)\otimes_{F}\B_{\st})^{N=0,\phi=p^r}\cap
F^r(H^i_{\dr}(X_K)\otimes_K\B_{\dr}).
\end{equation}
For $r\geq i$,  by \cite[Prop. 3.25, Cor. 3.26]{NN}, the right hand side is  isomorphic to
 $ H^i_{\synt}(X_{\ovk,h},r)$, i.e., there exists a natural  isomorphism
 \begin{equation}
 \label{formulaB1}
 h_{h,i,r}: H^i_{\synt}(X_{\ovk,\tr,h},r)\stackrel{\sim}{\to} (H^{B,i}_{\hk}(X_{\tr})\otimes_{F}\B_{\st})^{N=0,\phi=p^r}\cap
F^r(H^i_{\dr}(X_{\tr})\otimes_K\B_{\dr}).
 \end{equation}
We will denote by 
 $$\tilde{\alpha}_{h,i,r}:H^i_{\eet}(X_{\ovk,\tr},{\mathbf Q}_p(r))\stackrel{\sim}{\to} H^i_{\synt}(X_{\ovk,\tr,h},r)
$$
the induced isomorphism and call it the {\em syntomic period isomorphism}.

 The following lemma is immediate.
\begin{lemma}
\label{basic1B}
Let $r\geq i$. 
A period isomorphism $\alpha_{h,i,r}$, hence also a period isomorphism $\alpha_{h,i}$ 
 satisfying Theorem  \ref{comp1B}, is uniquely determined by the induced syntomic period isomorphism $\tilde{\alpha}_{h,i,r}$.
\end{lemma}
\begin{remark}\label{lyon-simplicial}
We also have an analog of Lemma \ref{basic1B} for finite simplicial schemes with components as in that lemma. The proof is analogous to the proof of Lemma \ref{lyon-added1}.
\end{remark}
\subsubsection{Comparison of Tsuji and Beilinson period morphisms.}\label{FM=B} Let $X\in {\mathcal V}ar_{\ovk}$. 
We can $h$-sheafify the Tsuji syntomic period morphism by setting, for $(U,\overline{U})\in\spp^{ss}_{\ovk}$, 
$$\beta_{r,n}^{\rm T}:
\rg_{\eet}((U,\overline{U}),\sss^{\prime}_n(r))\lomapr{\can} \rg_{\eet}((U,\overline{U}),\sss_n(r)) \lomapr{\beta^{\rm T}_r} 
\rg_{\eet}(U,\Z/p^n(r)^{\prime})
$$
from 
(\ref{tea11})  to obtain the compatible maps of $h$-sheaves
\begin{equation}
\label{padova-rain}
\beta_{r,n}^{\rm T}: \sss^{\prime}_n(r)\to \Z/p^n(r)^{\prime}.
\end{equation}
Taking cohomology we get
the induced compatible syntomic period morphisms
$$\beta^{\rm T}_n:\, \R\Gamma(X_h,\sss^{\prime}_n(r))\lomapr{\beta^{\rm T}_{r,n}} \R\Gamma(X_{h},{\mathbf Z}/p^n(r)^{\prime}) \stackrel{\sim}{\leftarrow}\R\Gamma(X_{\eet},{\mathbf Z}/p^n(r)^{\prime}).$$ 
As in the case of the Beilinson period morphism, they   induce  a syntomic period morphism
 \begin{equation}
 \label{Tsuji-ref}
 \wt{\beta}_{h,r}^{\rm T}:=p^{-r}\beta^{\rm T}\pi:
\rg_{\synt}(X_{h},r)\to  \rg_{\eet}(X,{\mathbf Q}_p(r)),\quad \beta^{\rm T}:=(\holim_n\beta^{\rm T}_{n})\otimes {\mathbf Q}. 
\end{equation}
 It is a quasi-isomorphism after truncation $\tau_{\leq r}$: by Remark \ref{padova1}, the map $\pi$ is a quasi-isomorphism after truncation $\tau_{\leq r}$ and, by Corollary \ref{syntomicP}, the map
 (\ref{padova-rain}) is a $p^r$-quasi-isomorphism  hence the map $\beta^T$  is a quasi-isomorphism after truncation $\tau_{\leq r}$ as well.

 \begin{theorem}\label{T=B}Let $r\geq 0$.
\begin{enumerate}
\item Let $X\in {\mathcal V}ar_{\ovk}$. The Tsuji and Beilinson syntomic period morphisms
$$
 \wt{\beta}_{h,r}^{\rm T}, \wt{\beta}_{h,r}^{\rm B}:
\rg_{\synt}(X_{h},r)\to  \rg_{\eet}(X,{\mathbf Q}_p(r))
$$
are equal. 
\item   If $X=(U,\overline{U})\in\spp^{ss}_K$ and is split over $\so_K$,   the period isomorphisms
\begin{align*}
 {\alpha}_{h,i}^{\rm T}, {\alpha}_{h,i}^{\rm B}: \quad &  H^i_{\eet}(U_{\ovk},{\mathbf Q }_p)
\otimes_{{\mathbf Q }_p} \B_{\st} \stackrel{\sim}{\to}
H^{B,i}_{\hk}(X)\otimes_{F} \B_{\st},\\
 {\alpha}_{i}^{\rm T}, {\alpha}_{i}^{\rm B}: \quad &  H^i_{\eet}(U_{\ovk},{\mathbf Q }_p)
\otimes_{{\mathbf Q }_p} \B_{\dr} \stackrel{\sim}{\to}
H^{i}_{\dr}(X_K)\otimes_{K} \B_{\dr},
\end{align*}
where we set $ {\alpha}_{h,i}^{\rm T}:= \kappa^{-1}{\alpha}_{i}^{\rm T}$, 
 are equal as well.
\end{enumerate}
\end{theorem}
\begin{proof}For the first claim, by construction of the syntomic period morphisms $ \wt{\beta}_{h,r}^{\rm T}$ and $ \wt{\beta}_{h,r}^{\rm B}$, it suffices to show that, for all $n\geq 1$,  the maps
\begin{align*}
& \beta_{n}^{\rm B}: \sss^{\prime}_n(r)\stackrel{\sim}{\leftarrow} S^{\prime}_n(r)\lomapr{\tau_n}\Z/p^n(r)^{\prime},\\
& \beta_{n}^{\rm T}: \sss^{\prime}_n(r)\lomapr{\alpha^{\rm T}_r}\Z/p^n(r)^{\prime}
\end{align*}
are equal. Or that so are the maps
\begin{align*}
 \tau_{n}:  & S^{\prime}_n(r)\to \Z/p^n(r)^{\prime},\\
& S^{\prime}_n(r)\to \sss^{\prime}_n(r)\lomapr{\beta^{\rm T}_r}\Z/p^n(r)^{\prime}.
\end{align*}
But this is immediate from the functoriality of $\beta^{\rm T}_{r,n}$: For $(U,\overline{U})\in\spp^{ss}_{\ovk}$, the canonical map $(U,\overline{U})\to (\Spec \ovk, \Spec \so_{\ovk})$ yields the commutative diagram
$$
\xymatrix{
\rg_{\eet}((U,\overline{U}),\sss^{\prime}_n(r))\ar[r]^-{\beta^T_r} & \rg_{\eet}(U,\Z/p^n(r)^{\prime})\\
\rg_{\eet}((\Spec \ovk,\Spec \so_{\ovk}),\sss^{\prime}_n(r))\ar[r]^-{\beta^T_r} \ar[u]& \rg_{\eet}(\Spec \ovk,\Z/p^n(r)^{\prime})\ar[u]\\
S^{\prime}_n(r)\ar[r]^{\tau_n} \ar[u]^{\wr}& \Z/p^n(r)^{\prime}\ar[u]^{\wr}
}
$$

  For the second claim, let $X=(U,\overline{U})\in\spp^{ss}_K$ be split over $\so_K$. By Lemma \ref{basic1B}, it suffices to show that, for $r\geq i$,  the induced maps $ \tilde{\alpha}_{h,i,r}^{\rm T}$ and $\tilde {\alpha}_{h,i,r}^{\rm B}$ 
  from $H^i_{\eet}(U_{\overline{K}},\Q_p(r))$ to $H^i_{\synt}(X_{\ovk, \tr,h}, r)$ are equal.  But, by the first claim of this theorem, it suffices to prove the following lemma. 
  \begin{lemma}
  \label{lyon-leaving}
 \begin{enumerate}
 \item   The map 
  $\tilde {\alpha}_{h,i,r}^{\rm B}$  is the inverse of the map $\tilde {\beta}_{h,i,r}^{\rm B}$. 
  \item The map 
  $\tilde {\alpha}_{h,i,r}^{\rm T}$  is the inverse of the map $\tilde {\beta}_{h,i,r}^{\rm T}$. 
  \end{enumerate}
  \end{lemma}
  \begin{proof}
 The first claim was shown in \cite[(49)]{NN}. 
   The second claim is also basically shown in \cite{NN} (which contains a detailed analysis of the Beilinson-Hyodo-Kato map and its interaction with more classical constructions). However, we could not find there the exact statement we need here so we provide an argument how the proof can be glued from statements proved already in \cite{NN}. 
  
  Consider the following diagram (all the maps are isomorphisms):
 $$
\xymatrix{
H^i_{\eet}(U_{\ovk},\Q_p(r))\ar@/_60pt/[rrd]_{\tilde{\alpha}^T_{i,r}}\ar[r]^-{\alpha^T_{h,i,r}}\ar@/^50pt/[rr]^{\tilde{\alpha}^T_{h,i,r}}\ar[dr]^-{\alpha^T_{i,r}} & C(H^{B,i}_{\hk}(U_{\ovk,h}),r)\ar[d]^{\wr}_{\kappa} & H^i_{\synt}(U_{\ovk,h},r)\ar[l]_-{h_{h,i,r}}\ar@/_20pt/[ll]_-{\tilde{\beta}^T_{h,i,r}}\\
& C(H^i_{\hk}(X),r) & H^i_{\synt}(X_{\ovk},r),\ar[l]_-{h_{i,r}}\ar@/^40pt/[llu]^{\tilde{\beta}^T_{i,r}}\ar[u]^{\can}_{\wr}
}
$$
where we set 
\begin{align*}
C(H^{B,i}_{\hk}(U_{\ovk,h}),r):= & \ker((H^{B,i}_{\hk}(U_{\ovk,h})\otimes_{F^{\rm nr}}\B^+_{\st})^{N=0,\phi=p^r}\lomapr{\rho^B\otimes\iota_p}(H^i_{\dr}(U_{\ovk,h})\otimes_{\ovk}\B^+_{\dr})/F^r),\\
C(H^{i}_{\hk}(X),r):= & \ker((H^{i}_{\hk}(X)\otimes_F\B^+_{\st})^{N=0,\phi=p^r}\lomapr{\rho_{\pi}\otimes\iota_{\pi}}(H^i_{\dr}(X_K)\otimes_K\B^+_{\dr})/F^r),\\
H^i_{\synt}(X_{\ovk},r):= & H^i\rg_{\synt}(X_{\ovk},r):=H^i\rg_{\eet}(X_{\so_{\ovk}},\sss^{\prime}(r))_{\Q}.
\end{align*}
Since, by definition, $\alpha^T_{i,r}=h_{i,r}(\tilde{\beta}^T_{i,r})^{-1}$ and the maps $\tilde{\beta}^T_{h,i,r}$, $\tilde{\beta}^T_{i,r}$ are compatible, a diagram chase shows that 
it suffices to show that the right square in the diagram commutes.

This diagram can be lifted to the $\infty$-derived category, where it takes the following form
$$
\xymatrix{
C(\rg^B_{\hk}(U_{\ovk,h}),r) \ar[d]^{\kappa}_{\wr}& \rg_{\synt}(U_{\ovk,h},r)\ar[l]_-{h_{h,r}}\\
C(\rg_{\hk}(X),r) &  \rg_{\synt}(X_{\ovk},r)\ar[l]_-{h_{r}}\ar[u]^{\wr},
}
$$
where we set 
\begin{align*}
C(\rg^{B}_{\hk}(U_{\ovk,h}),r):= & [[\rg^{B}_{\hk}(U_{\ovk,h})\otimes^L_{F^{\rm nr}}\B^+_{\st}]^{N=0,\phi=p^r}\lomapr{\rho^B\otimes\iota_p}(\rg_{\dr}(U_{\ovk,h})\otimes^L_{\ovk}\B^+_{\dr})/F^r],\\
C(\rg_{\hk}(X),r):= & [[\rg_{\hk}(X)\otimes^L_F\B^+_{\st}]^{N=0,\phi=p^r}\lomapr{\rho_{\pi}\otimes\iota_{\pi}}(\rg_{\dr}(X)\otimes^L_K\B^+_{\dr})/F^r].
\end{align*}
Proceeding now as in the proof of \cite[Lemma 4.7]{NN}, we reduce to proving that, possibly changing the base field $K$,  the following diagram commutes for all $X=(U,\overline{U})\in \spp^{ss}_K$ that are split over $\so_K$:
$$
\xymatrix{
C(\rg^B_{\hk}(X_{\so_{\ovk}}),r) \ar[d]^{\kappa}_{\wr}& \rg_{\synt}(X_{\ovk},r)\ar[l]_-{h^B_{r}}\ar[ld]^-{h_{r}}\\
C(\rg_{\hk}(X),r).
}
$$

    Recall that the map $h^B_r$ is defined as the following composition
     \begin{align*}
h^B_r:\quad \R\Gamma_{\synt}(X_{\ovk},r)
 & \xrightarrow{\sim}
  [\xymatrix{[\R\Gamma_{\crr}(X_{\so_{\ovk}})_{\mathbf Q}]^{\phi=p^r}\ar[r]^-{\gamma^{-1}_r}  & (\R\Gamma_{\dr}(X_{\ovk})\otimes^L_{\ovk}\B^+_{\dr}) /F^r}]\\
&   \stackrel{\sim}{\leftarrow}
   [\xymatrix{[\R\Gamma^B_{\hk}(X)\otimes^L_{F}\B_{\st}^+]^{N=0,\phi=p^r}\ar[r]^-{\rho^B\otimes\iota_p}  & (\R\Gamma_{\dr}(X_{K})\otimes^L_{K}\B^+_{\dr}) /F^r}]\\
 &   =: C(\rg_{\hk}^{\rm B}(X_{\so_{\ovk}}),r),
 \end{align*}
where we have used the quasi-isomorphism  $\gamma_r: (\R\Gamma_{\dr}(X_{\ovk})\otimes^L_{\ovk}\B^+_{\dr}) /F^r\stackrel{\sim}{\to}\rg_{\crr}(X_{\so_{\ovk}})_{\mathbf Q}/F^r$ and the
 second quasi-isomorphism in the definition of $h^B_r$ 
 uses  Beilinson crystalline period quasi-isomorphism  $$
 \rho^B_{\crr}:\quad (\R\Gamma^B_{\hk}(X)\otimes^L_{F}\B_{\st}^+)^{N=0}\stackrel{\sim}{\to}\R\Gamma_{\crr}(X_{\so_{\ovk}})_{\mathbf Q}
 $$ (that is compatible with the action of $N$ and $\phi$) as well as   \cite[Lemma 3.24]{NN} (which shows that  we have the needed commutative diagrams).
Recall that the map $h_r$ is defined as the following composition
     \begin{align}\label{lyon-nov}
h_r:\quad \R\Gamma_{\synt}(X_{\ovk},r)
 & \xrightarrow{\sim}
  [\xymatrix{[\R\Gamma_{\crr}(X_{\so_{\ovk}})_{\mathbf Q}]^{\phi=p^r}\ar[r]^-{\gamma^{-1}_r}  & (\R\Gamma_{\dr}(X_{\ovk})\otimes^L_{\ovk}\B^+_{\dr}) /F^r}]\\\notag
  &   \xrightarrow{\sim}
   [\xymatrix{[\R\Gamma_{\crr}(X_{\so_{\ovk}}/\wh{\A}_{\st})]_{\Q}^{N=0,\phi=p^r}\ar[r]^-{\gamma_{\pi}\otimes\iota_{\pi}}  & (\R\Gamma_{\dr}(X_{K})\otimes^L_{K}\B^+_{\dr}) /F^r}]\\\notag
  &  \xleftarrow[\cup]{\sim}  [\xymatrix{[\R\Gamma_{\crr}(X/R_{\pi})_{\Q}\otimes^L_{R_{\pi,\Q}}\wh{\B}^+_{\st})]_{\Q}^{N=0,\phi=p^r}\ar[r]^-{p_{\pi}\otimes\iota_{\pi}}  & (\R\Gamma_{\dr}(X_{K})\otimes^L_{K}\B^+_{\dr}) /F^r}]\\\notag
&   \xleftarrow[\iota_{\pi}\otimes\id]{\sim}
   [\xymatrix{[\R\Gamma_{\hk}(X)\otimes^L_{F}\wh{\B}_{\st}^+]^{N=0,\phi=p^r}\ar[r]^-{\rho_{\pi}\otimes\iota_{\pi}}  & (\R\Gamma_{\dr}(X_{K})\otimes^L_{K}\B^+_{\dr}) /F^r}]\\\notag
&   \stackrel{\sim}{\leftarrow}
   [\xymatrix{[\R\Gamma_{\hk}(X)\otimes^L_{F}\B_{\st}^+]^{N=0,\phi=p^r}\ar[r]^-{\rho_{\pi}\otimes\iota_{\pi}}  & (\R\Gamma_{\dr}(X_{K})\otimes^L_{K}\B^+_{\dr}) /F^r}].
 \end{align}
 Here the map $\gamma_{\pi}$ is defined as the composition
 $$
 \gamma_{\pi}: \R\Gamma_{\crr}(X_{\so_{\ovk}}/\wh{\A}_{\st})_{\Q}\to \rg_{\crr}(X_{\so_{\ovk}}/\so_K^{\times})_{\Q}/F^r\stackrel{\sim}{\leftarrow}  (\R\Gamma_{\dr}(X_{\ovk})\otimes^L_{\ovk}\B^+_{\dr}) /F^r.
 $$
The fact that the second and the third quasi-isomorphisms in the definition of the map $h_r$ are well-defined follows from the last 
commutative diagram in the proof of \cite[Prop. 3.48]{CDN3}.

 Finally, recall that the map $\kappa$ can be lifted to the $\infty$-derived category as well: we have a commutative diagram (see \cite[(31)]{NN})
 \begin{equation}
 \label{ducky1}
 \xymatrix{
 \rg_{\hk}(X)\ar[r]^{\rho_{\pi}} & \rg_{\dr}(X_K)\\
 \rg_{\hk}^B(X)\ar[u]^{\wr}_{\kappa}\ar[ur]_{\rho^B}.
 }
 \end{equation}
 Using this map $\kappa$ and its analogs one can write the bottom 4 homotopy fibers in the definition of the map $h_{\pi}$ (and the maps between them) using Beilinson-Hyodo-Kato cohomology instead of the original Hyodo-Kato cohomology (this includes a change of $p$ to $\pi$). See the last large diagram in the proof of \cite[Lemma 4.7]{NN} for how this is done. This diagram also shows that the obtained result is isomorphic to the map $h^B_r$, as wanted. 
\end{proof}
\end{proof}
\subsubsection{Period morphisms for motives, I}\label{T-motives}
Recall that the Beilinson period morphism lifts to the Voevodsky triangulated category of (homological) motives ${\rm DM}_{\rm gm}(K,\Q_p)$ \cite[4.15]{DN}. That is, for any Voevodsky motive $M$,   we have the Hyodo-Kato and de Rham comparison quasi-isomorphisms
\begin{align*}
&\alpha^{\rm B}_{\rm pst}: \quad   \rg_{\eet}(M)\otimes^L_{\Q_p}\B_{\st}\stackrel{\sim}{\to}\rg_{\hk}(M)\otimes^L_{F^{\rm nr}}\B_{\st},\\
&\alpha^{\rm B}_{\dr}:  \quad  \rg_{\eet}(M)\otimes^L_{\Q_p}\B_{\dr}\stackrel{\sim}{\to}\rg_{\dr}(M)\otimes^L_{\ovk}\B_{\dr}.
\end{align*}
They are compatible via the Hyodo-Kato quasi-isomorphism $\rho: \rg_{\hk}(M)\otimes^L_{F^{\rm nr}}\ovk\stackrel{\sim}{\to} \rg_{\dr}(M)$ and the map $\iota_p:\B_{\st}\to\B_{\dr}$. The complexes $ \rg_{\eet}(M)$, $ \rg_{\hk}(M)$, and $ \rg_{\dr}(M)$ are the  \'etale, Hyodo-Kato, and de Rham realizations of $M$, respectively. All cohomologies are  geometric. The comparison quasi-isomorphisms are compatible with Galois action, filtrations, monodromy, and Frobenius (when appropriate). If we apply them to the cohomological Voevodsky motive $M(X)^{\vee}=f_*(1_X)$ of any variety $X$ over $K$ with structural morphism $f$, we get back Beilinson period quasi-isomorphisms from Section \ref{Beilinson1}.
\begin{example}
An interesting case is obtained by using the (homological) motive with
 compact support $M^c(X)$ in ${\rm DM}_{\rm gm}(K,\Qp)$ of Voevodsky
 for any $K$-variety $X$, and its
 dual $M^c(X)^\vee=\underline{\Hom}(M^c(X),\Qp)$ which belongs to ${\rm DM}_{\rm gm}(K,\Qp)$ as well.
 Since, in terms of the 6 functors formalism, $M^c(X)^\vee=f_!(1_X)$ \cite[Prop. 8.10]{CD},  $\R\Gamma_{\eet}(M^c(X)^\vee)$  is the  \'etale cohomology
  with compact support (as defined by Grothendieck and Deligne). 
  
       Similarly for the Hyodo-Kato and de Rham cohomology. Let $X$ be a scheme over $\so_K$ with generalized semistable reduction as in Section \ref{reduction-lyon}. Let $D$ be its divisor at $\infty$. Define the Voevodsky motive 
  $
  M(X_K,D_K)\in {\rm DM}_{\rm gm}(K,\Q_p)$ as the cone
  $$
   M(X_K,D_K):=\Cone(M(\wt{D}_{{\scriptscriptstyle\bullet},K})\lomapr{i_*}M(X_K)),
  $$
  where $\wt{D}_{{\scriptscriptstyle\bullet}}$ is the \v{C}ech nerve of the map $\coprod_i D_i\to D$, $D_i$ being an irreducible component of $D$. Hence the dual motive $
  M(X_K,D_K)^{\vee}\in {\rm DM}_{\rm gm}(K,\Q_p)$:
  $$
    M(X_K,D_K)^{\vee}\simeq {\rm Fiber}(M(X_K)^{\vee}\lomapr{i^*}M(\wt{D}_{{\scriptscriptstyle\bullet},K})^{\vee}).
  $$
  \begin{lemma}
  For $U:=X\setminus D$, we have
  $$
  M^c(U_K)^{\vee}\simeq M(X_K,D_K)^{\vee}.
  $$
  \end{lemma}
  \begin{proof}
  This easily follows from the localization property \cite[3.3.10]{CD}:
  $$
  M^c(U_K)^{\vee}\simeq {\rm Fiber}(M(X_K)^{\vee}\lomapr{i^*}M(D_K)^{\vee})
  $$
  and the Mayer-Vietoris property for closed coverings (a special case of cdh-descent \cite[3.3.10]{CD}), which yields
  $$
  M(D_K)\simeq M(\wt{D}_{{\scriptscriptstyle\bullet},K}).
  $$
  \end{proof}
  Hence, by Lemma \ref{lyon1}, the realization $\rg_{\epsilon}(M^c(U_K)^{\vee})$, $\epsilon=\hk, \dr$, represents the compactly supported cohomology of $U_K$. 
  \end{example}

  Similarly, the Tsuji period morphism also  lifts to the Voevodsky triangulated category of (homological) motives ${\rm DM}_{\rm gm}(K,\Q_p)$ \cite[4.15]{DN}. More specifically, for $X\in {\mathcal V}ar_{\ovk}$, the syntomic period morphism from (\ref{Tsuji-ref})
  $$
   \wt{\beta}_{h,r}^{\rm T}:
\rg_{\synt}(X_{h},r)\to  \rg_{\eet}(X,{\mathbf Q}_p(r))
  $$
  extends to a syntomic period morphism
  $$ \wt{\beta}_{r}^{\rm T}:
\rg_{\synt}(M,r)\to  \rg_{\eet}(M,{\mathbf Q}_p(r)),\quad M\in{\rm DM}_{\rm gm}(K,\Q_p).
$$
 It is quasi-isomorphism after truncation $\tau_{\leq r}$. If we apply it to the cohomological Voevodsky motive $M(U)^{\vee}=f_*(1_X)$ for  any proper semistable scheme  $X$ over $\so_K$, and $U=X_K\setminus D_K$ with structural morphism $f$, we get back Fontaine-Messing period quasi-isomorphisms (modulo  identifications of the cohomologies involved and their $h$-localizations). 
  
  For $M\in {\rm DM}_{\rm gm}(K,\Q_p)$, define
  \begin{align*}
 &  \wt{\alpha}_{r}^{\rm T}:\quad \tau_{\leq r}\rg_{\eet}(M,{\mathbf Q}_p(r))\xleftarrow[\sim]{\wt{\beta}_{r}^{\rm T}}\tau_{\leq r}\rg_{\synt}(M,r)\lomapr{h_{h,r}}\tau_{\leq r}(\rg_{\hk}(M)\otimes^L_{F^{\rm nr}}\B_{\st}\{-r\}),\\
 &    {\alpha}_{{\rm pst},r}^{\rm T}:\quad \tau_{\leq r}\rg_{\eet}(M,{\mathbf Q}_p)\to\tau_{\leq r}(\rg_{\hk}(M)\otimes^L_{F^{\rm nr}}\B_{\st}),\quad {\alpha}_{{\rm pst},r}^{\rm T}:=t^{-r}\wt{\alpha}_{r}^{\rm T}\epsilon^r.
  \end{align*}
  Here $h_{h,r}$ is the motivic lift of the $h$-sheafification of the map $h_r$ from (\ref{lyon-nov}). 
  Write $\rg_{\eet}(M,{\mathbf Q}_p(r))\simeq \hocolim_r \tau_{\leq r}\rg_{\eet}(M,{\mathbf Q}_p(r))$ and set 
  $$ {\alpha}_{\rm pst}^{\rm T}:=\hocolim_r {\alpha}_{{\rm pst},r}^{\rm T}: \rg_{\eet}(M,{\mathbf Q}_p)\to \rg_{\hk}(M)\otimes^L_{F^{\rm nr}}\B_{\st}.
  $$ This makes sense since, by \cite[Cor. 4.8.8]{Ts}, we have $t\wt{\alpha}_{r-1}^{\rm T}=\wt{\alpha}_{r}^{\rm T}\epsilon$.

  To sum up, for any Voevodsky motive $M$,   we have the Hyodo-Kato and de Rham comparison quasi-isomorphisms
\begin{align}\label{lyon-later}
&\alpha^{\rm T}_{\rm pst}: \quad   \rg_{\eet}(M)\otimes^L_{\Q_p}\B_{\st}\stackrel{\sim}{\to}\rg_{\hk}(M)\otimes^L_{F^{\rm nr}}\B_{\st},\\\notag
&\alpha^{\rm T}_{\dr}:  \quad  \rg_{\eet}(M)\otimes^L_{\Q_p}\B_{\dr}\stackrel{\sim}{\to}\rg_{\dr}(M)\otimes^L_{\ovk}\B_{\dr}
\end{align}
as in the case of Beilinson comparison quasi-isomorphisms. By Theorem \ref{T=B}, these  comparison quasi-isomorphisms are the same as the ones of Beilinson. If we apply them  to the cohomological Voevodsky motive $M(U)^{\vee}=f_*(1_X)$ for  any proper semistable scheme  $X$ over $\so_K$, and $U=X_K\setminus D_K$ with structural morphism $f$, we get back   Tsuji period quasi-isomorphisms after the identification of the Beilinson-Hyodo-Kato and the original Hyodo-Kato cohomology via the map $\kappa:  \rg_{\hk}^B(X)\to  \rg_{\hk}(X)$ from (\ref{ducky1}).  
\begin{remark}
In \cite{Ts1} Tsuji has shown that the Fontaine-Messing period morphism yields a comparison theorem for $U$ as above. This was done by showing compatibility of the period morphism with the Gysin sequence and thus reducing to the proper case. The period quasi-isomorphisms (\ref{lyon-later}) imply Tsuji's result. But we know now of another way: using Banach-Colmez spaces as in \cite{CN} one can obtain the isomorphism (\ref{formulaB1}) which is enough to prove that the period map is an isomorphism; this way one avoids using  Poincar\'e duality.  
\end{remark}

 The map $\kappa$ and its properties extend to finite proper simplicial schemes with semistable reduction and of Cartier type, which implies that Tsuji 
comparison theorem for cohomology with compact support  from \cite{Tc} agrees with the one of Beilinson (after the identification of Hyodo-Kato cohomologies). Similarly, since  the comparison  theorems of Yamashita for cohomology with (possibly partial) compact support can be also seen as  defined using finite simplicial schemes (use the arguments of Lemma \ref{lyon1}) and the Fontaine-Messing period morphisms they are the same as those of Tsuji and Beilinson. 

 Finally, as shown in \cite[Prop. 4.24]{DN}, the Beilinson period morphisms are compatible with (possibly mixed) products. By the same argument so are the period morphisms (\ref{lyon-later}). It follows that so are the  period morphisms of 
 Tsuji and Yamashita (the change of Hyodo-Kato cohomology map $\kappa$ is compatible with products: pass through the Hyodo-Kato isomorphisms -- which are compatible with products -- to de Rham cohomology). 
\subsection{Comparison of Faltings and Beilinson period morphisms} We will compare now the Faltings and Beilinson period morphisms. 
\subsubsection{Faltings period morphism.} \label{tea1} We will briefly recall the definition of the period morphism of Faltings.

 {\em {\rm (i)} Faltings site.}
 Faltings construction of the period morphism uses  an auxiliary topos,
topos of `` sheaves of local systems'' \cite[III]{Fi}, \cite[3]{Fa} that is now known as the``Faltings topos" (a term first used by Abbes and Gros \cite{AG}). We will  briefly describe it.

 For a scheme $X$, let $X_{\Feet}$ denote the topos defined  by the site of finite \'etale morphisms $U\to X$
with coverings given by surjective maps. For a connected $X$ and a choice of a geometric point $\overline{x}\to X$,
$X_{\Feet}$ is equivalent to the topos of sets  with a continuous action of
the fundamental group $\pi_1(X,\overline{x})$. In particular, for an abelian sheaf $\sff$, the \'etale cohomology $H^*(X_{\Feet},\sff)$ is isomorphic
to the (continuous) group cohomology $H^*(\pi_1(X,\overline{x}),\sff_{\overline{x}})$.
Let $X$ be noetherian. Then  $X_{\Feet}$ is equivalent
to the topos of \'etale sheaves that are inductive limits of locally constant sheaves\footnote{For us, {\em locally constant} is a shorthand for locally constant constructible.}. There is a map of topoi
$$\pi:X_{\eet}\to X_{\Feet}$$
with $\pi_*{\sff}$ given by the restriction of $\sff$ to finite \'etale schemes over $X$ and $\pi^*(\sff)=\sff$ for an ind-locally constant sheaf $\sff$.

 Recall the following notion.
\begin{definition}
A noetherian scheme $X$ is a $K(\pi,1)$-space if for every integer $n$ invertible on $X$ and any locally constant sheaf $\sll$ of ${\mathbf Z}/n$-modules,
the natural map $\sll\to R\pi_*\pi^*(\sll)$ is an isomorphism.
\end{definition}
The following analogue of a classical result of Artin \cite[Exp. XI, 4.4]{Ar} on the existence of a base for the Zariski topology consisting of $K(\pi,1)$-spaces was proved by Faltings \cite[2.1]{Fa0} in the good reduction case and by Achinger \cite[Th. 9.5]{Ach} in general.
\begin{theorem}{\rm (Faltings, Achinger)}\label{FA}
Let $X$ be a log-smooth $\so_K^{\times}$-log-scheme such that $X_K$ is smooth over $K$. For every geometric point $\overline{x}$ of $X$, $X_{(\overline{x})}\times_XX_{\tr,\ovk}$ is a  $K(\pi,1)$-space.
\end{theorem}

 Let $X$  be a noetherian $\so_K$-scheme. The Faltings topos $\wt{X}_{\ovk,\eet}$ is defined\footnote{We use here the modification of the original definition of Faltings presented by Abbes and Gros in \cite{AG}.} by a site which has for objects pairs $(U,V)$, where $U$ is an \'etale $X$-scheme and $V\to X_{\ovk}$ is a finite \'etale morphism; morphisms are compatible
 pairs of maps, and coverings are pairs of surjective maps (see \cite{AG} for details). 
 
    There is a canonical map $$\rho : X_{\ovk,\eet}\to \wt{X}_{\ovk,\eet}$$  from the \'{e}tale topos of $X_{\ovk}$
   to $\widetilde{X}_{\ovk,\eet}$.  On the level of sites, this map is given by sending $(U,V)$ to $V$. If $X$ is a log-smooth log-scheme over $\so^{\times}_K$ with a smooth generic fiber, it follows \cite[III]{Fa}, \cite[Cor. 9.6]{Ach} from Theorem \ref{FA}   that, for a locally constant  sheaf $\sll$ on $X_{\ovk}$, the natural map 
   \begin{equation}
   \label{lyon-3}
   \R\Gamma(\widetilde{X}_{\ovk,\eet},\rho_*\sll)\to \R\Gamma(X_{\ovk,\eet},\sll)
   \end{equation}
   is a quasi-isomorphism.

{\em {\rm (ii)} Faltings period morphism.}
Let $X$ be a saturated, log-smooth,  and proper  log-scheme over $\so^{\times}_K$. Then, by \cite[Cor. 3.1]{Fa},
 we have a natural almost quasi-isomorphism
$$
v_{r,n}: \R\Gamma(\wt{X}_{\ovk,\eet},{\mathbf Z}/p^n)\otimes^{L} F^r\A_{\crr,n}\stackrel{\sim}{\rightarrow} \R\Gamma(\wt{X}_{\ovk,\eet},F^r\sa_{\crr,n}),\quad r\geq 0,
$$
where $\sa_{\crr,n}$ is a relative version of the crystalline period ring (equipped with the log-stracture $({\mathbf N}\to \sa_{\crr,n},1\mapsto [\pi^{\flat}])$).
For $r\geq 0$, there is a natural morphism
$$\beta_{r,n}: \R\Gamma_{\crr}(X_n/R_{\pi,n},\sj^{[r]})\to   \R\Gamma(\wt{X}_{\ovk,\eet},F^r\sa_{\crr,n}).
$$

 Faltings main comparison result is the following:
\begin{theorem}{\rm (Faltings, \cite[Cor. 5.4]{Fa})} 
The almost morphism $$
\tilde{\beta}_n: \R\Gamma_{\crr}(X_n/R_{\pi,n})\otimes^{L}_{R_{\pi,n}}\A_{\crr,n}
\to \R\Gamma_{\eet}(X_{\tr,\ovk},{\mathbf Z}/p^n)\otimes^{L} \A_{\crr,n},\quad \tilde{\beta}_n:=\rho^*v^{-1}_{0,n}\beta_{0,n},
$$
has an inverse up to $t^{d}$ (that is, composition either way is the multiplication by $t^{d}$), $d=\dim X_K$.
It is  compatible with Frobenius and filtration.
\end{theorem}
The map $R_{\pi,n}\to \A_{\crr,n}$ above is induced by $x\mapsto [\pi^{\flat}]$. This is not Galois equivariant hence, for the period morphism $\tilde{\alpha}$ to be compatible with the Galois action, this action has to be twisted (using monodromy) on the domain (see \cite[p. 259]{Fa} for details).  Passing to the limit over $n$ and tensoring with ${\mathbf Q}$ in the above yields an almost morphism
$$\tilde{\beta}: \R\Gamma_{\crr}(X/R_{\pi})\otimes^{L}_{R_{\pi}}\B^+_{\crr}
\to \R\Gamma_{\eet}(X_{\tr,\ovk},{\mathbf Z}_p)\otimes^{L} \B^+_{\crr}.
$$
Taking cohomology we  get an isomorphism
$$
\tilde{\beta}_{i}: H^i_{\crr}(X/R_{\pi})_{\Q}\otimes_{R_{\pi,\Q}}\B_{\crr}
\stackrel{\sim}{\to }H^i_{\eet}(X_{\tr,\ovk},{\mathbf Q}_p)\otimes \B_{\crr}.
$$
Faltings period isomorphism
$$
\alpha_{i}^F: H^i_{\eet}(X_{\tr,\ovk},{\mathbf Q}_p)\otimes \B_{\crr}\stackrel{\sim}{\to }H^i_{\hk}(X)\otimes_{F}\B_{\crr}
$$
is defined as $\alpha_{i}^{\rm F}:=(\beta_i^{\rm F})^{-1}$, $\beta_i^{\rm F}:=\wt{\beta}_{i}\iota_{\pi}$, where $\iota_{\pi}: H^i_{\hk}(X)\to H^i_{\crr}(X/R_{\pi})_{\Q}$ is the Hyodo-Kato section.

 {\em {\rm (iii)} Faltings syntomic period morphism.} Let $r\geq 0$. The definition of the  map $\beta_{r,n}$ above can be generalized easily to obtain an almost  map
 $$
 \beta_{r,n}: \R\Gamma_{\crr}(X_{\so_{\ovk},n}/R_{\pi,n},\sj^{[r]})\to   \R\Gamma((\wt{X}_{\so_{\ovk}})_{\ovk,\eet},F^r\sa_{\crr,n})\stackrel{a. is.}{\longleftarrow}
  \R\Gamma(\wt{X}_{\ovk,\eet},F^r\sa_{\crr,n}).
 $$
 Here we set $ \R\Gamma((\wt{X}_{\so_{\ovk}})_{\ovk,\eet},F^r\sa_{\crr,n}):=\hocolim_{K^{\prime}}\R\Gamma((\wt{X}_{\so_{K^{\prime}}})_{\ovk,\eet},F^r\sa_{\crr,n})$, where the limit is over finite extensions $K^{\prime}/K$. 
 In an analogous way we define almost maps\footnote{We note that these maps do not depend on the choice of the uniformizer $\pi$.}
 $$
 \tilde{\beta}_{r,n}: \R\Gamma_{\crr}({X}_n,\sj^{[r]})\to  \R\Gamma(\wt{X}_{\ovk,\eet},F^r\sa_{\crr,n}),\quad  \tilde{\beta}_{r,n}: \R\Gamma_{\crr}(X_{\so_{\ovk},n},\sj^{[r]})\to  \R\Gamma(\wt{X}_{\ovk,\eet},F^r\sa_{\crr,n}).
 $$
 All these maps are compatible.

 Recall that we have  the fundamental exact sequence
 \begin{equation}
 \label{seq11}
0\to \Z/p^n(r)^{\prime}_s \to  F^r_p\sa_{\crr,n}\lomapr{\phi_r-1} F^r\sa_{\crr,n}\to 0
 \end{equation}
 Here  $F^r_p\sa_{\crr,n}$  denotes the Frobenius ``divisible" filtration and, for a sheaf $\sff$ on $\wt{X}_{\ovk,\eet}$, $\sff_s$ stands for its restriction to the special fiber,
i.e., to the complement of the generic fiber (the site consisting of objects with trivial special fiber). For $X$ proper and $\sff$ torsion,
 proper base change theorem yields that the cohomologies of $\sff$ and $\sff_s$ coincide.

 Using the map $\tilde{\beta}_{r,n}$ and the above sequence, we obtain a map 
 $$
   \tilde{\beta}_{r,n}:\R\Gamma_{\eet}(X_{\so_{\ovk}},\sss^{\prime}_n(r))\to  \R\Gamma_{\eet}(\wt{X}_{\ovk}, \Z/p^n(r)^{\prime}_s).
 $$
More precisely, we get a canonical map from $\rg_{\eet}(X_{\so_{\ovk}},\sss'_n(r))$ to the
$\wt{X}_{\ovk}$-cohomology of the mapping fiber  of $\phi-p^r : F^r\sa_{\crr,n}\to \sa_{\crr,n}$,
which in turn maps via multiplication by $p^r$ on $F^r\sa_{\crr,n}$ to the $\wt{X}_{\ovk}$-cohomology of the mapping fiber  of $\phi_r-1: F^r_p\sa_{\crr,n}\to \sa_{\crr,n}$.
 But the last mapping fiber,  by
the fundamental exact sequence  (\ref{seq11}), is quasi-isomorphic to ${\mathbf Z}/p^n(r)'_s$.

  Hence Faltings period isomorphism induces a morphism (a genuine morphism not just an almost morphism, see \cite[Sec. 5.1]{N10})
\begin{equation}
\label{tea12}
\beta^{\rm F}_{r,n}: \R\Gamma_{\eet}(X_{\so_{\ovk}},\sss'_n(r))\to \R\Gamma_{\eet}(X_{\tr,\ovk},{\mathbf Z}/p^n(r)')
\end{equation}
as the composition {\small$$\beta^{\rm F}_{r,n}:
\rg_{\eet}(X_{_{\so_{\ovk}}},\sss^{\prime}_n(r))\lomapr{\tilde{\beta}_{r,n}}
\rg_{\eet}(\wt{X}_{\overline{K}},{\mathbf
Z}/p^n(r)'_s)\stackrel{\sim}{\leftarrow}
\rg_{\eet}(\wt{X}_{\overline{K}},{\mathbf
Z}/p^n(r)')\stackrel{\sim}{\rightarrow}\rg_{\eet}(X_{\tr,\ovk},{\mathbf
Z}/p^n(r)').
$$}
The first quasi-isomorphism holds because $X$ is proper. The last quasi-isomorphism holds by (\ref{lyon-3}). 
Consider now the composition ($\beta^{\rm F}_r:=(\holim_n\beta^{\rm F}_{r,n})_{\Q}$)
$$
\tilde{\beta}_r^{\rm F}:\quad \rg_{\eet}(X_{\so_{\ovk}},\sss^{\prime}(r))_{\Q}
\lomapr{\beta^{\rm F}_r} \rg_{\eet}(X_{\ovk},{\mathbf
Q}_p(r))\lomapr{p^{-r}}\rg_{\eet}(X_{\ovk},{\mathbf Q}_p(r)).
$$
For $r\geq i$, using the diagram (\ref{lyon-diag}) and the discussion in  \cite{N10} preceding Theorem 5.8, it is easy to check that, on degree $i$ cohomology,  $(\tilde{\beta}^{\rm F}_{i,r})^{-1} $ is the syntomic period morphism $\tilde{\alpha}^{\rm F}_{i,r}$ induced from the Faltings period morphism $\alpha^F_{i,r}$ via the procedure described in Section \ref{comp2}.
  
\subsubsection{Comparison of Faltings and Beilinson period morphisms.} Let $X\in {\mathcal V}ar_{\ovk}$. 
We can $h$-sheafify the Faltings period morphism by setting, for $(U,\overline{U})\in\spp^{ss}_{\ovk}$, 
$$\beta_{r,n}^{\rm F}:
\rg_{\eet}((U,\overline{U}),\sss^{\prime}_n(r))\lomapr{\can} \rg_{\eet}((U,\overline{U}),\sss_n(r)) \lomapr{\beta^{\rm F}_{r,n}}, 
\rg_{\eet}(U,\Z/p^n(r)^{\prime})
$$
where the morphism $\beta^{\rm F}_{r,n}$ is the one from (\ref{tea12}),
 to obtain the compatible maps of $h$-sheaves
\begin{equation}
\label{padova-rain1}
\beta^{\rm F}_{r,n}: \sss^{\prime}_n(r)\to \Z/p^n(r)^{\prime}.
\end{equation}
Taking cohomology we get
the induced compatible syntomic period morphisms
$$\beta^{\rm F}_n:\, \R\Gamma(X_h,\sss^{\prime}_n(r))\lomapr{\beta^{\rm F}_{r,n}} \R\Gamma(X_{h},{\mathbf Z}/p^n(r)^{\prime}) \stackrel{\sim}{\leftarrow}\R\Gamma(X_{\eet},{\mathbf Z}/p^n(r)^{\prime}).$$ 
As in the case of the Beilinson period morphism, they   induce  a syntomic period morphism
 $$
 \wt{\beta}_{h,r}^{\rm F}:=p^{-r}\beta^{\rm F}\pi:
\rg_{\synt}(X_{h},r)\to  \rg_{\eet}(X,{\mathbf Q}_p(r)),\quad \beta^{\rm F}:=(\holim_n\beta^{\rm F}_{n})\otimes {\mathbf Q}. 
$$
 It is a quasi-isomorphism after truncation $\tau_{\leq r}$: by Remark \ref{padova1}, the map $\pi$ is a quasi-isomorphism after truncation $\tau_{\leq r}$ and, by Corollary \ref{syntomicP}, the map
 (\ref{padova-rain1}) is a $p^r$-quasi-isomorphism  hence the map $\beta^{\rm F}$  is a quasi-isomorphism after truncation $\tau_{\leq r}$ as well.

Since the Faltings syntomic period morphism ${\beta}_{r,n}^{\rm F}$ is functorial,  
an argument analogous to the one we used in the proof of Theorem \ref{T=B} shows that $\tilde{\beta}_{h,r}^{\rm F}=\tilde{\beta}_{h,r}^{\rm B}$. We have obtained the first claim of the following: 
\begin{theorem}\label{Faltings=beilinson}Let $r\geq 0$.
\begin{enumerate}
\item Let $X\in {\mathcal V}ar_{\ovk}$.  The induced Faltings and Beilinson syntomic period morphisms
$$
 \wt{\beta}_{h,r}^{\rm F}, \wt{\beta}_{h,r}^{\rm B}:
\rg_{\synt}(X_{h},r)\to  \rg_{\eet}(X,{\mathbf Q}_p(r))
$$
are equal. 
\item If $X=(U,\overline{U})\in\spp^{ss}_K$ and is split over $\so_K$, the  period morphisms 
\begin{align*}
{\alpha}_{h,i}^{\rm F}, {\alpha}_{h,i}^{\rm B}: \quad &  H^i_{\eet}(U_{\ovk},{\mathbf Q }_p)
\otimes_{{\mathbf Q }_p} \B_{\st} \stackrel{\sim}{\to}
H^{B,i}_{\hk}(X)\otimes_{F} \B_{\st},\\
 {\alpha}_{i}^{\rm F}, {\alpha}_{i}^{\rm B}: \quad &  H^i_{\eet}(U_{\ovk},{\mathbf Q }_p)
\otimes_{{\mathbf Q }_p} \B_{\dr} \stackrel{\sim}{\to}
H^{i}_{\dr}(X_K)\otimes_{K} \B_{\dr}
\end{align*}
are equal as well.
\end{enumerate}
\end{theorem}
\begin{proof}
Let $X=(U,\overline{U})\in\spp^{ss}_K$ be split over $\so_K$. By Lemma \ref{basic1B}, it suffices to show that, for $r\geq i$,  the induced maps $ \tilde{\alpha}_{h,i,r}^{\rm F}$ and $\tilde {\alpha}_{h,i,r}^{\rm B}$ 
  from $H^i_{\eet}(U_{\overline{K}},\Q_p(r))$ to $H^i_{\synt}(X_{\ovk, h}, r)$ are equal.  But by Lemma \ref{lyon-leaving}, the map 
  $\tilde {\alpha}_{h,i,r}^{\rm B}$  is the inverse of the map $\tilde {\beta}_{h,i,r}^{\rm B}$. Hence, by the first claim of our theorem it suffices to prove the lemma below.
  \end{proof} 
\begin{lemma}
The map 
  $\tilde {\alpha}_{h,i,r}^{\rm F}$  is the inverse of the map $\tilde {\beta}_{h,i,r}^{\rm F}$. 
  \end{lemma}
\begin{proof}
Identical to the proof of the second claim of Lemma \ref{lyon-leaving} (recall that the main issue there was a relation between syntomic cohomology and the Hyodo-Kato and Beilinson-Hyodo-Kato cohomologies).
\end{proof}
\subsubsection{Period morphisms for motives, II} The content of  Section \ref{T-motives} goes through practically verbatim for Faltings period morphism. 
We obtain that, for any Voevodsky motive $M$,   we have the Hyodo-Kato and de Rham comparison quasi-isomorphisms
\begin{align*}
&\alpha^{\rm F}_{\rm pst}: \quad   \rg_{\eet}(M)\otimes^L_{\Q_p}\B_{\st}\stackrel{\sim}{\to}\rg_{\hk}(M)\otimes^L_{F^{\rm nr}}\B_{\st},\\\notag
&\alpha^{\rm F}_{\dr}:  \quad  \rg_{\eet}(M)\otimes^L_{\Q_p}\B_{\dr}\stackrel{\sim}{\to}\rg_{\dr}(M)\otimes^L_{\ovk}\B_{\dr}
\end{align*}
as in the case of Beilinson comparison quasi-isomorphisms. By Theorem \ref{Faltings=beilinson}, these  comparison quasi-isomorphisms are the same as the ones of Beilinson. If we apply them  to the cohomological Voevodsky motive $M(U)^{\vee}=f_*(1_X)$ for  any proper semistable scheme  $X$ over $\so_K$, and $U=X_K\setminus D_K$ with structural morphism $f$, we get back   Faltings period quasi-isomorphisms after the identification of the Beilinson-Hyodo-Kato and the original Hyodo-Kato cohomology via the map $\kappa:  \rg_{\hk}^B(X)\to  \rg_{\hk}(X)$ from (\ref{ducky1}).  

Hence we recover Theorem \ref{FTN} comparing Faltings and Fontaine-Messing period morphisms for cohomology with compact support. But we also get:
\begin{enumerate}
\item Faltings and Fontaine-Messing period morphisms are equal for open varieties: because they are equal to Beilinson period morphisms.
\item Faltings period morphisms are compatible with (mixed) products (which recovers \cite{Fa}): use the argument for Tsuji products in Section \ref{T-motives}.
\end{enumerate}

\end{document}